\documentclass[12pt]{article}
\usepackage{amsmath,amsfonts,amsthm,amssymb}
\usepackage{graphicx,psfrag,epsf}
\usepackage{enumerate}
\usepackage[numbers]{natbib}
\usepackage{color}
\usepackage{bm}
\usepackage[pdftex,bookmarks=true,breaklinks,hypertexnames=false]{hyperref} 
\usepackage{tikz}
\usepackage{xr}
\usepackage{algorithm}
\usepackage[noend]{algpseudocode}
\usepackage{multirow}
\usepackage{subcaption}
\newcommand{\blind}{1}

\newtheorem{theorem}{Theorem}

\newtheorem{lemma}[theorem]{Lemma}
\newtheorem{remark}[theorem]{Remark}%
\newtheorem{corollary}[theorem]{Corollary}

\newtheorem{assumption}[theorem]{Assumption}

\usepackage{xcolor}

\usepackage{stmaryrd}

\def\g{\gamma}

\newcommand{\eps}{\varepsilon}

\newcommand{\RR}{\mathbb{R}}

\newcommand{\NN}{\mathbb{N}}

\newcommand{\HH}{\mathbb{H}}

\renewcommand{\b}{\beta}
\newcommand{\cov}{\text{cov}}
\renewcommand{\epsilon}{\varepsilon}

\newcommand{\bs}{\boldsymbol}
\DeclareMathOperator{\E}{\mathbb{E}}
\newcommand{\e}{\varepsilon}

\renewcommand{\g}{\gamma}

\begin{document}

\def\spacingset#1{\renewcommand{\baselinestretch}%
{#1}\small\normalsize} \spacingset{1}


\if1\blind
{
  \title{\bf Adaptation using spatially distributed Gaussian Processes}
  \author{Botond Szabo \thanks{Co-funded by the European Union (ERC, BigBayesUQ, project number: 101041064).
Views and opinions expressed are however those of the author(s) only and do not necessarily reflect
those of the European Union or the European Research Council. Neither the European Union nor the
granting authority can be held responsible for them. }\hspace{.2cm}\\
    Department of Decision Sciences and BIDSA, Bocconi University, Italy\\
    and \\
    Amine Hadji \\
    Mathematical Institute, Leiden University, The Netherlands\\
and\\
Aad van der Vaart\thanks{This research was partially funded by a Spinoza grant of the Dutch Research Council (NWO).}\\
Delft Institute of Applied Mathematics, DIAM,\\ Delft University of Technology, The Netherlands}
  \maketitle
} \fi

\if0\blind
{
  \bigskip
  \bigskip
  \bigskip
  \begin{center}
    {\LARGE\bf Adaptation using spatially distributed Gaussian Processes}
\end{center}
  \medskip
} \fi

\bigskip
\begin{abstract}
We consider the accuracy of an approximate posterior distribution in nonparametric regression problems by combining posterior distributions
computed on subsets of the data defined by the locations of the independent variables. We show that this approximate posterior retains the rate of recovery of the full data posterior distribution, where the rate of recovery adapts to the smoothness of the true regression function. As particular examples we consider Gaussian process priors based on integrated Brownian motion and the Mat\'ern kernel augmented with a prior on the length scale. Besides theoretical guarantees we present a numerical study of the methods  both on synthetic and real world data. We also propose a new aggregation technique, which numerically outperforms previous approaches. Finally, we demonstrate empirically that spatially distributed methods can adapt to local regularities, potentially outperforming the original Gaussian process.
\end{abstract}

\noindent%
\vfill

\newpage
\spacingset{1.9}
\section{Introduction}
Gaussian processes (GPs) are standard tools in statistical and machine learning. They provide a particularly effective prior distribution over the space of functions and are routinely used  in regression and classification tasks, amongst others. The monograph  \cite{rasmussen:williams:2006} gives an in-depth overview of the foundations and practical applications of this approach. However, GPs scale poorly with the sample size $n$. For instance, in regression the computational complexity and memory requirements are of the orders $O(n^3)$ and $O(n^2)$, respectively. This limitation has triggered the development of various approximation methods, including sparse approximations of the empirical covariance matrices \cite{gibbs:poole:stockmeyer:1976,saad:1990,quinonero:rasmussen:2005}, variational Bayes approximations \cite{titsias:2009, burt2019, nieman2022} or distributed methods. We focus on the latter method in this paper.

In distributed (or divide-and-conquer) methods, the compu\-tational burden is reduced by splitting the data over ``local'' machines (or servers, experts  or cores). Next the computations are carried out locally, in parallel to each other, before transmitting the outcomes to a ``central'' server or core, where the partial, local results are combined, forming the final outcome of the procedure. This distributed architecture occurs naturally when data is collected and processed locally and only a summary statistic is transmitted to a central server. Besides speeding up the computations and reducing the memory requirements, distributed methods can also help in protecting privacy, as the data do not have to be stored in a central data base, but are processed locally.

In the literature various distributed methods were proposed to speed up Bayesian computation, in particular in the context of Gaussian Processes. One can distinguish two main strategies depending on the data-splitting technique. The first approach is to partition the data randomly over the servers, computing a posterior distribution on each server and finally aggregating these local distributions by some type of averaging. Examples include Consensus Monte Carlo \cite{scott:blocker:et.al:16}, WASP \cite{srivastava:2015a}, Generalized Product of Experts \cite{cao:fleet:14}, Bayesian Committee Machine \cite{tresp:2000,deisenroth2015}, Distributed Bayesian Varying Coefficients \cite{guhaniyogi2022distributed} and Distributed Kriging \cite{guhaniyogi2017divide}. The second approach takes advantage of the spatial structure of the data and splits the observations based on a partition of the design space. Each machine is assigned a specific region of the space, a local posterior distribution is computed using the data in this region, and these are glued together to form the final answer.  This approach is referred to as  the Naive-Local-Experts model, see \cite{kim:mallick:holmes,vasudevan2009gaussian}. We discuss various methods of combining the local outputs in Section~\ref{sec:synthetic}, including a new proposal, which outperforms its competitors in the case that the length scale of the priors is determined from the data.

A number of papers in the literature studied the randomly split data approach, deriving theoretical guarantees, but also limitations,  for a range of methods and models. Under the assumption that the regularity of the underlying functional parameter is known, minimax rate-optimal contraction rates and frequentist coverage guarantees for Bayesian credible sets were derived in the context of the Gaussian white noise \cite{szabo:vzanten:19} and nonparametric regression models \cite{guhaniyogi2017divide, shang:cheng:2015, hadji2022optimal}. However, in practice  the latter regularity is typically not known, but inferred from the data in some way. In the mentioned references it is shown that with randomly split data, standard adaptation techniques necessarily lead to highly sub-optimal inference, see \cite{szabo:vzanten:19}. In the numerical analysis in the present paper we observe this on both synthetic and real world data sets.

In contrast, spatially partitioned distributed approaches have received little theoretical attention, despite their popularity in applications. In this paper we aim to fill this gap in the literature. We derive general contraction rate theorems under mild assumptions and apply them in the context of the nonparametric regression and classification models. We also consider two specific GP priors: the rescaled integrated Brownian motion and the Mat\'ern process and show that both priors (augmented with an additional layer of prior on the scale parameter) lead to rate-adaptive posterior contraction rates. This is in sharp contrast to the randomly split data framework, which necessarily results in sub-optimal estimation. Thus we provide the first adaptive distributed Bayesian method with theoretical guarantees. We also demonstrate the superior performance of spatially distributed methods on synthetic and real world data sets. Furthermore, we propose a novel aggregation technique, which  numerically outperforms its close competitors, especially in the realistic situation that the length scales of the local posteriors are adapted to the data. Finally, we also demonstrate numerically that the spatially distributed GP methods can adapt to different local regularities in contrast to the original GP. Therefore, spatially distributing the data can not only speed up the computations, but can potentially improve the accuracy of the GP method.

The paper is organised as follows. In Section~\ref{sec: main} we introduce the spatially distributed general framework with GP priors and recall the regression and classification models, considered as examples in our paper. Then in Sections~\ref{subsec: nonadapt} and~\ref{sec:adapt} we derive general contraction rate results under mild conditions in the non-adaptive and adaptive frameworks, respectively, using the hierarchical Bayesian method in the latter one. As specific examples we consider the rescaled integrated Brownian motion and the Mat\'ern process in Sections~\ref{sec:IBM} and~\ref{sec:Matern}, respectively. For both priors we derive rate-adaptive contraction rates in the regression and classification models using the fully Bayesian approach. The theoretical guarantees are complemented with a numerical analysis. In Section~\ref{sec:aggregation} we discuss various aggregation techniques. We investigate their numerical properties compared to benchmark distributed and non-distributed methods on synthetic and real world data sets in Sections~\ref{sec:synthetic} and~\ref{sec:realworld}, respectively. We discuss our results and future directions in Section~\ref{sec:discussion}. The proofs for the general theorems are given in Section~\ref{sec:proof:general} and for the specific examples in Section~\ref{sec:proof:examples}. A collection of auxiliary lemmas is presented in Section~\ref{sec:proof:lemmas}. Finally, additional numerical analysis on real and simulated data sets are provided in Sections \ref{sec: realworld:extra} and \ref{sec: simulations:extra}, respectively. We highlight that local adaptation of the process is investigated in Section \ref{sec:spatial} of the supplement.


We write $C^\beta([a,b])$ for the H\"older space of order $\beta>0$: the space of functions $f: [a,b]\to\RR$ that are $b$ times differentiable, for $b$ the largest integer strictly smaller than $\beta$, with  highest order derivative $f^{(b)}$ satisfying $|f^{(b)}(x)-f^{(b)}(y)|\lesssim |x-y|^{\beta-b}$, for every $x,y\in [a,b]$. We also write $H^\beta([a,b])$ for the Sobolev space of order $\beta$.

\section{Spatially distributed Bayesian inference with GP priors}\label{sec: main}
We consider general nonparametric regression models. The observations are independent pairs $(x_1, Y_1),\ldots, (x_n,Y_n)$, where the covariates $x_i$ are considered to be fixed and the corresponding dependent variables $Y_i$ random. We state our abstract theorems in this general setting, but next focus on two commonly used models: nonparametric regression with Gaussian errors and logistic regression.

In the standard nonparametric regression model the observed data $\textbf{Y}=(Y_1,Y_2,\ldots,Y_n)\in\mathbb{R}^n$  satisfy the relation
\begin{align}
Y_i= f_0(x_i)+Z_i,\qquad Z_i\stackrel{iid}{\sim}\mathcal{N}(0,\sigma^2),  \qquad i=1,\ldots,n.\label{def:regression}
\end{align}
The goal is to estimate the unknown regression function $f_0$, which is assumed to be smooth, but not to take a known parametric form.
In the logistic regression model the observed data $\textbf{Y}=(Y_1,Y_2,\ldots,Y_n)\in\{0,1\}^n$ are binary with likelihood function
\begin{align}
\Pr(Y_i=1|x_i)= \psi\big(f_0(x_i)\big),  \qquad i=1,\ldots,n,\label{def:log:regression}
\end{align}
where $\psi(x)=1/(1+e^{-x})$ denotes the logistic function. Additional examples include Poisson regression, binomial regression, etc.

We consider the distributed version of these models. We assume that the data is spatially distributed over $m$ machines in the following way. The $k$th machine, for $k\in\{1,\ldots,m\}$, receives the observations $Y_i$ with design points $x_i$ belonging to the $k$th subregion $\mathcal{D}^{(k)}$ of the design space $\mathcal{D}$, i.e.\  $x_i\in \mathcal{D}^{(k)}$. In the examples in Section~\ref{sec:examples} we take the domain of the regression function $f_0$ to be the unit interval $[0,1]$ and split it into equidistant sub-intervals $I^{(k)}= (\frac{k-1}{m}, \frac{k}{m}]$. We use the shorthand notations $\textbf{x}^{(k)}=\{x_i:\, x_i\in \mathcal{D}^{(k)}\}$, and $\textbf{Y}^{(k)}=\{Y_i:\, x_i\in \mathcal{D}^{(k)}\}$. However, our results hold more generally, also in the multivariate setting. For simplicity of notation, we assume that $|\textbf{x}^{(k)}|=n/m$, but in general it is enough that the number of points in each subregion is proportional to $n/m$ (with respect to a universal constant).

We endow the functional parameter $f_0$ in each machine with a Gaussian Process prior $(G_t^{(k)}: t\in \mathcal{D})$, with sample paths supported on the full covariate space, identical in distribution but independent across the machines. Gaussian prior processes often depend on regularity and/or scale hyper-parameters, which can be adjusted to achieve bigger flexibility. Corresponding local posteriors are computed based on the data $\textbf{Y}^{(k)}=\{Y_i:\, x_i\in \mathcal{D}^{(k)}\}$ available at the local machines, independently across the machines. These define stochastic processes (supported on the full covariate space), which we aggregate into a single one by restricting them to the corresponding subregions and pasting them together, i.e.\ a draw $f$ from the ``aggregated posterior'' is defined as
\begin{align}
 f(x)=\sum_{k=1}^m 1_{\mathcal{D}^{(k)}}(x)  f^{(k)}(x),\label{def:aggr:spatial}
\end{align}
where $ f^{(k)}$ is a draw from the $k$th local posterior. Formally, an ``aggregated posterior measure'' is defined  as
\begin{align}
\Pi_{n,m}( B|\textbf{Y})=\prod_{k=1}^m\Pi^{(k)}( B_k|\textbf{Y}^{(k)}),\label{def:aggr:post}
\end{align}
where $B$ is a measurable set of functions,
$\Pi^{(k)}(\cdot| \textbf{Y}^{(k)})$ is the posterior distribution in the $k$th local problem corresponding to the prior $\Pi^{(k)}(\cdot)$, and
$B_k$ is the set of all functions whose restriction to $\mathcal{D}^{(k)}$ agrees with the corresponding restriction of some element of $B$, i.e.
$$B_k=\bigl\{\vartheta\colon [0,1]\to\RR \Big|\,  \exists f\in B\text{ such that }\vartheta(x)= f(x),\, \forall x\in \mathcal{D}^{(k)}\bigr\}.$$

\subsection{Posterior contraction for distributed GP for independent observations}\label{subsec: nonadapt}

We investigate the contraction rate of the aggregate posterior $\Pi_{n,m}(\cdot|\textbf{Y})$ given in \eqref{def:aggr:post}. Our general result is stated in terms of a local version of the concentration function originally introduced in \cite{vaart:zanten:2008} for the non-distributed model. This local concentration function is attached to the restriction of the Gaussian prior process to the subregion $\mathcal{D}^{(k)}$. For $k=1,\ldots, m$, let $\| f\|_{\infty,k}=\sup_{x\in\mathcal{D}^{(k)}}| f(x)|$ denote the $L_{\infty}$-norm restricted to $\mathcal{D}^{(k)}$, and define
\begin{align}
\phi_{ f_0}^{(k)}(\eps)=\inf_{h\in \mathbb{H}^{(k)}:\| f_0-h\|_{\infty,k}\leq\eps}\|h\|^2_{\mathbb{H}^{(k)}}-\log\Pi^{(k)}\left( f:\| f\|_{\infty,k}<\eps\right),\label{def:conc:function}
\end{align}
where   $\|\cdot\|_{\mathbb{H}^{(k)}}$ is the norm corresponding to the Reproducing Kernel Hilbert Space (RKHS) $\mathbb{H}^{(k)}$ of the Gaussian process $(G_t^{(k)}: t\in \mathcal{D}^{(k)})$ and $\Pi^{(k)}$ denotes the law of this process.

We consider contraction rates relative to the semimetric $d_n$, whose square is given by
\begin{align}
d_n(f,g)^2&=\frac{1}{n}\sum_{i=1}^n h_i(f,g)^2,\label{def:Hellinger}\\
\noalign{\noindent with}
 h_i(f,g)^2&=\int \bigl(\sqrt{p_{f,i}}-\sqrt{p_{g,i}}\bigr)^2 d\mu_i,\nonumber
\end{align}
	where  $p_{f,i}$ denotes the density of $Y_i$ given $x_i$ and $f$ with respect to some dominating measure $\mu_i$, for $i=1,\ldots,n$. 
The semimetrics $d_n$ are convenient for general theory, but as we discuss below, our results can be extended to other semimetrics as well, for instance to the empirical $L_2$-distance
$\|f-g\|_n$, for $\|f\|_n^2=n^{-1}\sum_{i=1}^n f^2(x_i)$,  in the case of the nonparametric regression model. 

The following standard and mild assumption relate the supremum norm to the $d_n$ semimetric and Kullback-Leibler divergence and variation.
	
\begin{assumption}\label{ass:metric}
For all bounded functions $f, g$, 
\begin{align*}
\max\big\{  h_i(f,g)^2, K(p_{f,i},p_{g,i}), V(p_{f,i},p_{g,i})\big\}\lesssim \|f-g\|_{\infty,k}^2,
\end{align*}
where $K(p_{f,i},p_{g,i})=\int \log(p_{f,i}/p_{g,i}) p_{f,i}\,d\mu_i$ and $V(p_{f,i},p_{g,i})=\int \log(p_{f,i}/p_{g,i})^2 p_{f,i}\,d\mu_i$.
\end{assumption}

For instance, in the case of the nonparametric regression model, the maximum in the left hand side is bounded above by a multiple of $(f(x_i)-g(x_i))^2\leq \|f-g\|_{\infty,k}^2$, for any $x_i\in \mathcal{D}^{(k)}$; see for example page~214 of \cite{ghoshal:vaart:2007}. The condition also holds for the logistic regression model; see for instance the proof of Lemma~2.8 of \cite{ghosal2017fundamentals}. 

The preceding assumption suffices to express the posterior contraction rate with the help of the local concentration functions. The proof of the following theorem can be found in Section~\ref{sec:proof:nonadapt}.

\begin{theorem}\label{thm: nonadaptive}
Let $ f_0$ be a bounded function and assume that there exists a sequence $\eps_n\rightarrow 0$ with $(n/m^2)\eps_n^2\rightarrow\infty$ such that $\phi_{ f_0}^{(k)}(\eps_n)\leq (n/m)\eps_n^2$, for $k=1,\ldots,m$. Then under Assumption~\ref{ass:metric},
the aggregated posterior given in \eqref{def:aggr:post} contracts around the truth with rate $\eps_n$, i.e.
$$\E_0\Pi_{n,m}\bigl( f:d_n( f, f_0)\geq M_n \eps_n|\textbf{Y}\bigr)\rightarrow 0,$$
for arbitrary $M_n\rightarrow\infty$. In the distributed nonparametric regression model \eqref{def:regression} 
or the classification model \eqref{def:log:regression}, the condition $(n/m^2)\eps_n^2\rightarrow\infty$ may be relaxed to
 $m= o(n\eps_n^2/\log n)$.
\end{theorem}

{
\begin{remark}
We note that the frequentist contraction rate guarantees for the spatially distributed methods, given in Theorem \ref{thm: nonadaptive} above, hold more generally, beyond Gaussian Process priors. The general results can be stated with the help of the appropriately adapted versions of the prior small ball, remaining mass and entropy conditions to the distributed setting. We provide such a general result in Theorem \ref{thm: adaptation} below, for the adaptive, hierarchical choice of the prior.
\end{remark}}

\subsection{Adaptation}\label{sec:adapt}
It is common practice to tune the prior GP by changing its ``length scale'' and consider the process $t\mapsto G^{\tau}_t:=G_{\tau t}$, for a given parameter $\tau$ instead of the original process. Even though the qualitative smoothness of the sample paths does not change, a dramatic impact on the posterior contraction rate can be observed when $\tau=\tau_n$ tends to infinity or zero with the sample size $n$. A length scale $\tau>1$ entails shrinking a process on a bigger time set to the time set $[0,1]$, whereas $\tau<1$ corresponds to stretching. Intuitively, shrinking makes the sample paths more variable, as the randomness on a bigger time set is packed inside a smaller one, whereas stretching creates a smoother process. We show in our examples that by optimally choosing the scale hyper-parameter, depending on the regularity of the true function $f_0$ and the GP, one can achieve rate-optimal contraction for the aggregated posterior. We also show that this same rate-optimal contraction (up to an arbitrary level set by the user) is achieved in a data-driven way by choosing the length scale from a prior, without knowledge of the regularity of the underlying function $f_0$. Thus we augment the model with another layer of prior, making the scale parameter $\tau$ a random variable,  in a fully Bayesian approach.

Each local problem, for $k=1,\ldots,m$, receives its own scale parameter, independently of the other problems, and a local posterior is formed  using the local data of each problem independently across the local problems. For simplicity we use the same prior for $\tau$ in each local problem. If this is given by a Lebesgue density $g$ and $\Pi^{\tau,(k)}$ is the prior on $f$ with scale $\tau$ used in the $k$th local problem, then the hierarchical prior for $f$ in the $k$th local problem takes the form
\begin{align}
\Pi^{g,(k)}(\cdot)=\int \Pi^{\tau,(k)}(\cdot) g(\tau)\,d\tau.\label{def:local:hier:prior}
\end{align}
After forming a local posterior using this prior and the corresponding local data in each local problem, an aggregated posterior is constructed as in the non-adaptive case, i.e.\ a draw $f$ from the aggregated posterior is given in \eqref{def:aggr:spatial}.
The corresponding aggregated posterior measure takes the form
\begin{align}
\Pi_{n,m}^{g}( B|\textbf{Y})=\prod_{k=1}^m\Pi^{g,(k)}(B_k|\textbf{Y}^{(k)}),\label{def:hier:posterior}
\end{align}
for $\Pi^{g,(k)}(\cdot| \textbf{Y}^{(k)})$ the $k$th posterior distribution corresponding to the prior \eqref{def:local:hier:prior}.

\begin{theorem}\label{thm: adaptation}
Let $ f_0$ be a bounded function and assume that there exist measurable sets of functions $B_{n,m}^{(k)}$ such that for all local hierarchical priors $\Pi^{g,(k)}$ given in \eqref{def:local:hier:prior} and $\eps_n\rightarrow 0$ such that $(n/m^2)\eps_n^2\rightarrow\infty$, it holds that,  for some $c,C>0$,
\begin{align}
\label{eq: rem-mass}
\Pi^{g,(k)}( f: f\notin B_{n,m}^{(k)})\leq e^{-4(n/m) \eps_n^2},\\
\label{eq: sb-prior}
\Pi^{g,(k)}( f:\| f- f_0\|_{\infty,k}\leq\eps_n)\geq e^{-(n/m) \eps_n^2},\\
\label{eq: entropy}
\log N(c\eps_n,B_{n,m}^{(k)},\|\cdot\|_{\infty,k})\le C(n/m)\eps^2_n.
\end{align}
Then under Assumption~\ref{ass:metric}, the aggregated hierarchical posterior  given in \eqref{def:hier:posterior}, contracts around the truth with rate $\eps_n$, i.e.
\begin{align}
\E_0\Pi_{n,m}^{g}\bigl( d_n( f, f_0)\geq M_n \eps_n |\textbf{Y}\bigr)\rightarrow 0,
\end{align}
for arbitrary $M_n\rightarrow \infty$. In the distributed nonparametric regression model \eqref{def:regression} 
or the classification model \eqref{def:log:regression}, the condition $(n/m^2)\eps_n^2\rightarrow\infty$ may be relaxed to 
$m= o(n\eps_n^2/\log n)$.
\end{theorem}

The proof of the theorem is deferred to Section~\ref{sec:proof:adaptation}.
{
\begin{remark}
One can consider adaptation to other type of hyper-parameters as well, for instance choosing the regularity or truncation parameters in a data driven way. Our results can be extend to these cases as well in a straightforward way. However, such approaches are, typically, computationally substantially more expensive and hence less popular in practice than rescaling the process. Therefore, we have refrained from including such cases in our analysis.
\end{remark}}

\section{Examples}\label{sec:examples}
In this section we apply the general results of the preceding section to obtain (adaptive) minimax contraction rates 
for regression and classification, with priors built on integrated Brownian motion and the Mat\'ern process.

\subsection{Rescaled Integrated Brownian Motion}\label{sec:IBM}
The "released" $\ell$-fold integrated Brownian motion is defined as 
\begin{align}
G_t:=B\sum_{j=0}^{\ell}\frac{Z_jt^j}{j!}+(I^{\ell}W)_t,\quad t\in[0,1],\label{def:IBM}
\end{align}
with $B>0$, and  i.i.d.\ standard normal random variables $(Z_j)_{j=0}^{\ell}$ independent from a Brownian motion $W$. The functional operator $I^{\ell}$ denotes taking repeated indefinite integrals and has the purpose of smoothing out the Brownian motion sample paths. Formally we define $(If)_t=\int_0^t f(s)\,ds$ and next $I^1=I$ and $I^{\ell}=I^{\ell-1}I$ for  $\ell\ge 2$. Because the sample paths of Brownian motion are  H\"older continuous of order almost $1/2$ (almost surely), the process $t\mapsto (I^\ell W)_t$ and hence the process $t\mapsto G_t$ has sample paths that are $\ell$ times differentiable with $\ell$th derivative H\"older of order almost $1/2$. The polynomial term in $t\mapsto G_t$ allows this process to have nonzero derivatives at zero, where the scaling by $B$ of this fixed-dimensional part of the prior is relatively inessential. The prior process $t\mapsto G_t$ in \eqref{def:IBM} is an appropriate model for a function that is regular of order $\ell+1/2$: it is known  from \cite{vaart:zanten:2008} that the resulting posterior contraction rate is equal to the minimax rate for a $\beta$-H\"older function $f_0$ if and only if $\beta=\ell+1/2$. For $\beta\not=\ell+1/2$, the posterior still contracts, but at a suboptimal rate. To remedy this, we introduce additional flexibility by rescaling the prior.

Because the integrated Brownian motion is self-similar, a time rescaling is equivalent to a space rescaling with another coefficient. We consider a time rescaling and introduce, for a fixed $\tau>0$,
\begin{align}
\label{eq: intBMresc}
G^{\tau,(k)}_t:=B_n\sum_{j=0}^{\ell} \frac{Z^{(k)}_j(\tau t)^j}{j!}+(I^{\ell}W^{(k)})_{\tau t}, \quad t\in[0,1].
\end{align}
The $(Z_j^{(k)})$ and $W^{(k)}$ are standard normal variables and a Brownian motion, as in \eqref{def:IBM}, but independently
across the local problems.
This process has been studied in \cite{vaart:zanten:2007} (or see Section~11.5 of \cite{ghosal2017fundamentals}) in the non-distributed nonparametric regression setting. The authors demonstrated that for a given $\beta\leq \ell+1$, the scale parameter $\tau:=\tau_n=n^{(\ell+1/2-\beta)/((\ell+1/2)(2\beta+1))}$ leads to the optimal contraction rate in the minimax sense at a $\beta$-regular function $f_0$, i.e.
$$\Pi^{\tau_n}\left(f:\| f- f_0\|_n\geq M_n n^{-\beta/(2\beta+1)}|\textbf{Y}\right)\rightarrow 0,$$
for arbitrary $M_n$ tending to infinity. Our first result shows that this same choice of length scale in the local prior distributions leads to the same contraction rate for the distributed, aggregated posterior distribution.

\begin{corollary}\label{cor:IBM:nonadaptresc}
Consider the distributed nonparametric regression model \eqref{def:regression} or the classification model \eqref{def:log:regression} with a function $f_0\in C^\beta([0,1])$, for $\beta>1/2$. In each local problem endow $f$ with the rescaled integrated Brownian motion prior \eqref{eq: intBMresc} with $\ell+1/2\ge \beta$ with 
 $\tau=\tau_n\asymp n^{(\ell+1/2-\beta)/((\ell+1/2)(2\beta+1))}$ and $\exp\{n^{1/(1+2\beta)}/m\} \geq B_n^2\geq n^{\frac{-1+2(\ell-\beta)\vee 0}{1+2\beta}}m$. 
Then for $m=o( n^{1/(2\beta+1)}/\log n)$, the aggregated posterior \eqref{def:aggr:post} achieves the minimax contraction rate, i.e.
$$\E_0\Pi_{n,m}\left( f: d_n( f, f_0)\geq M_n n^{-\beta/(2\beta+1)}|\textbf{Y}\right)\rightarrow 0,$$
for arbitrary $M_n\rightarrow\infty$. In case of the regression model \eqref{def:regression}, the pseudo-metric $d_n$ can be replaced by the empirical $L_2$-metric $\|\cdot\|_n$.
\end{corollary}

Thus the aggregated posterior contracts at the optimal rate, provided that the number of machines does not increase more than a certain polynomial in the number of data points. 

Unfortunately, the corollary employs a scaling rate $\tau_n$ that depends on the smoothness $\beta$ of the true function, which is typically unknown in practice. To remedy this, we consider a data-driven procedure for selecting $\tau$. In each local problem we choose a random scale factor $\tau$, independently from the variables $(Z_j^{(k)})$ and $W^{(k)}$ and independently across the problems, from a hyper-prior distribution with Lebesgue density $g_{\ell,n,m}$ satisfying, for every $\tau>0$,
\begin{align}\label{def:hyperprior}
C_1 \exp\{-D_1n^{\frac{1}{2(\ell+1)}}\tau^{\frac{\ell+1/2}{\ell+1}}/m\}
\leq g_{\ell,n,m}(\tau)\leq C_2 \exp\{-D_2n^{\frac{1}{2\ell+2}} \tau^{\frac{\ell+1/2}{\ell+1}}/m\},
\end{align}
where  $C_1$, $D_1$, $C_2$, $D_2$ are positive constants. The following corollary shows that this procedure results in rate-optimal recovery of the underlying truth.

\begin{corollary}\label{cor:adaptive:IBM}
Consider the distributed nonparametric regression model \eqref{def:regression} or the classification model \eqref{def:log:regression}  with a function $f_0\in C^\beta([0,1])$, for $\beta>1/2$. In each local problem endow $f$ with the hierarchical prior \eqref{def:local:hier:prior} built on the randomly rescaled integrated Brownian motion prior given in \eqref{eq: intBMresc} with 
$\ell+1/2\ge \beta$ and $\exp\{n^{1/(2+2\ell)}/m\}\geq B_n^2\geq  n^{(\ell-1)\vee 0}m$ and hyper-prior density $g_{\ell,n,m}$ satisfying \eqref{def:hyperprior}. Then, for $m=o(n^{1/(2\ell+2)}/\log n)$, the aggregated posterior \eqref{def:hier:posterior} adapts to the optimal minimax contraction rate, i.e
$$\E_0\Pi_{n,m}^{g}\left( f:\, d_n( f, f_0)\geq M_n n^{-\beta/(2\beta+1)}|\textbf{Y}\right)\rightarrow 0,$$
for arbitrary $M_n\rightarrow\infty$. In case of the regression model \eqref{def:regression}, the pseudo-metric $d_n$ can be replaced by the empirical $L_2$-metric $\|\cdot\|_n$.
\end{corollary}

The corollary shows that the aggregated posterior with randomly rescaled local priors contracts at the optimal rate for a true function of given H\"older smoothness level, as long as the hyper-prior and the number of experts are chosen appropriately.

Proofs for the results in this section are given in Section~\ref{sec:proof:adaptive:IBM}.

\subsection{Mat\'ern process}\label{sec:Matern}
The Mat\'ern process is a popular prior, particularly in spatial statistics (see e.g. \cite{rasmussen:williams:2006}, page 84).
It is a stationary mean zero  Gaussian process with spectral density
\begin{align}\label{def:Matern:spectral}
\rho_{\alpha,\tau}(\lambda)=C_{\alpha,d}\tau^{2\alpha}(c_{\alpha,d}\tau^{2}+\|\lambda\|^2)^{-\alpha-d/2},
\end{align}
where $\alpha,\tau >0$ are parameters,  $d$ is the dimension (we shall restrict to $d=1$) and $c_{\alpha,d},C_{\alpha,d}>0$ are constants. The sample paths of the Mat\'ern process are Sobolev smooth of order $\alpha$, and $\tau$ is a scale parameter: if $t\mapsto G_t$ is Mat\'ern with parameter $\tau=1$, then $t\mapsto G_{\tau t}$ is Mat\'ern with parameter $\tau$. (For consistency of notation we took $\tau=1/\ell$ in  \cite{rasmussen:williams:2006}, page 84.) The present time-rescaled Mat\'ern process is different from the space-rescaled version $t\mapsto \tau^{\alpha}G_t$ (for any $\alpha$) and has been less studied. In Section~\ref{sec:tech:lemma} we derive bounds on its small ball probability and the entropy of the unit ball of its reproducing kernel Hilbert space. These quantities,  in their dependence on $\tau$, are important drivers of posterior contraction rates, and of independent interest. For computation the Mat\'ern process can be spatially represented with the help of Bessel functions. 

First we consider the non-adaptive setting where the regularity parameter $\beta>0$ of the unknown function of interest $ f_0$ is assumed to be known. We choose each local prior equal to a Mat\'ern process with regularity parameter $\alpha$ satisfying $\beta\leq \alpha$, scaled by $\tau_n=n^{\frac{\alpha-\beta }{\alpha(1+2\beta)}}$ to compensate for the possible mismatch between $\alpha$ and $\beta$. It is known that the Mat\'ern prior gives minimax optimal contraction rates if used on the full data \cite{vdVvZJMLR}. The following corollary asserts that, in the distributed setting, the aggregated posterior corresponding to this choice of prior also achieves the minimax contraction rate.

\begin{corollary}\label{cor:Matern:nonadapt}
Consider the distributed nonparametric regression model \eqref{def:regression} or the classification model \eqref{def:log:regression} with a function $f_0\in C^\beta([0,1])$, for $\beta>1/2$. In each local problem endow $f$ with the rescaled Mat\'ern process prior with regularity parameter $\alpha$  satisfying $\alpha\ge \beta$ and $\alpha+1/2 \in\mathbb{N}$ and scale parameter $\tau_n=n^{(\alpha-\beta)/(\alpha(1+2\beta))}$. Then for $m=o(n^{1/(1+2\beta)}/\log n)$, the corresponding aggregated posterior \eqref{def:aggr:post} achieves the minimax contraction rate $n^{-\beta/(1+2\beta)}$, i.e.
$$\E_0\Pi_{n,m}\big( f: d_n(f,f_0)\geq M_n n^{-\beta/(2\beta+1)}|\textbf{Y}\big)\rightarrow 0,$$
for arbitrary $M_n\rightarrow\infty$. In case of the regression model \eqref{def:regression}, the pseudo-metric $d_n$ can be replaced by the empirical $L_2$-metric $\|\cdot\|_n$.
\end{corollary}

Next we consider the local hierarchical priors \eqref{def:local:hier:prior} with hyper-prior density satisfying, for every $\tau>0$,
\begin{align}
\label{def:hyperprior:Mat}
 c_1\exp\{-d_1n^{\frac{1}{2\alpha+1}} \tau^{\frac{\alpha}{\alpha+1/2}}/m\}
\leq g_{\alpha,n,m}(\tau)\leq c_2 \exp\{-d_2n^{\frac{1}{2\alpha+1}} \tau^{\frac{\alpha}{\alpha+1/2}}/m\},
\end{align}
where $c_1,d_1,c_2,d_2$ are positive constants. The priors in the local problems are then chosen to be Mat\'ern with random scales drawn from $g_{\alpha,n,m}$, and the aggregated distributed posterior follows our general construction in \eqref{def:hier:posterior}. The following corollary shows that using Mat\'ern processes yields rate-optimal contraction  over a range of regularity classes, similarly to the integrated Brownian motion prior case.

\begin{corollary}\label{cor:matern:adapt}
Consider the distributed nonparametric regression model \eqref{def:regression} or the classification model \eqref{def:log:regression} with a function $f_0\in C^\beta([0,1])$, for  $\beta>1/2$. In each local problem endow $f$
with the hierarchical prior built on the randomly rescaled Mat\'ern Process with regularity parameter $\alpha$ satisfying $\alpha\ge \beta$ and $\alpha+1/2\in\mathbb{N}$ and scale drawn from a density satisfying \eqref{def:hyperprior:Mat}. Then for $m=o(n^{1/(1+2\alpha)}/\log n)$ the aggregated posterior \eqref{def:hier:posterior} adapts to the optimal minimax contraction rate, i.e.
$$\E_0\Pi_{n,m}^{g}\left( f:\, d_n(f, f_0)\geq M_n  n^{-\beta/(2\beta+1)}|\textbf{Y}\right)\rightarrow 0,$$
for arbitrary $M_n\rightarrow\infty$. In case of the regression model \eqref{def:regression}, the pseudo-metric $d_n$ can be replaced by the empirical $L_2$-metric $\|\cdot\|_n$.
\end{corollary}

Proofs for the results in this section are given in Section~\ref{sec:proof:adapt:matern}.

\section{Numerical analysis}
In this section we investigate the distributed methods numerically by simulation and illustrate it 
on a real data problem. We start by a discussion of more refined aggregation techniques than \eqref{def:aggr:spatial}.
Our numerical analysis was carried out using the MatLab package {\tt gpml}.

\subsection{Aggregation techniques}\label{sec:aggregation}
In spatially distributed GP regression a draw from the aggregated posterior takes the form 
\eqref{def:aggr:spatial}, where the $\mathcal{D}^{(k)}$ are the sub-regions into which the design points are partitioned and $f^{(k)}\sim \Pi^{(k)}(\cdot|\textbf{Y}^{(k)})$. The output can be considered a weighted average of the local posteriors, with the indicator functions $1_{\mathcal{D}^{(k)}}(x)$ as weights. Although the procedure provides optimal recovery of the underlying truth, as shown in the preceding sections, the sample functions \eqref{def:aggr:spatial} are discontinuous at the boundaries of the regions $\mathcal{D}{(k)}$. The optimality implies that the discontinuities are small, but  they are visually unappealing. 

Various approaches in the literature palliate this problem. In the Patched GP method neighbouring local GPs are constrained to share nearly identical predictions on the boundary, see \cite{park:huang:2016,park:apley:2018}. In \cite{tresp:2001,rasmussen:ghahramani:2002,meeds:osindero:2006} a two-step mixture procedure was proposed, following the mixture of experts architecture of \cite{jacobs1991adaptive}. A prediction at a given point is drawn from an expert (local posterior) that is selected from a pool of experts by a latent variable, which is endowed with a prior to provide a dynamical, Bayesian procedure. 

Another method, more closely related to \eqref{def:aggr:spatial}, is to consider continuous weights instead of the discontinuous indicators $ 1_{\mathcal{D}^{(k)}}(x)$. Since the pointwise variances of a local posterior is smaller in the region where the local data lies than outside of this region, inverse pointwise variances are natural as weights. Following this idea, \cite{ng:deisenroth:14} introduced as aggregation technique 
\begin{align}
 f(x)=\sum_{k=1}^m \frac{ f^{(k)}(x)}{\sigma^2_{k}(x)}\Big/\sum_{k=1}^m\frac1{\sigma^{2}_{k}(x)},\label{def:spatial:ng}
 \end{align}
where $\sigma^{2}_k(x)$ is  the variance of $f^{(k)}(x)$ if $f^{(k)}\sim\Pi^{(k)}(\cdot|\textbf{Y}^{(k)})$. This approach provides data-driven and continuous weights. However, as shown in our numerical analysis, this leads to sub-optimal behaviour in the adaptive setting, where the scale parameter is tuned to the data. Perceived local smoothness in the data in region $\mathcal{D}^{(k)}$ will induce a small variance in the induced local posterior distribution, due to the adaptive bandwidth choice. This posterior variance will then also be relatively small outside the local region, where the local posterior is not informed by data, no matter the nature of the data in this region. The inverse variance weights then lead to overly large weights even outside of the experts' domain. That is to say that an expert will be overly confident about their knowledge of the true function in the whole space when this function is particularly smooth in this expert's own domain. 

In view of these observations we propose a new approach, which introduces more severe shrinkage outside of the local domain. As samples from the aggregated posterior, consider the weighted average 
\begin{align}
 f(x)=\sum_{k=1}^m w_k(x) f^{(k)}(x)\Bigl/\sum_{k=1}^m w_k(x),\label{def:novel:method}
 \end{align}
with weights, for $c_k$ being the geometric center of $\mathcal{D}^{(k)}$,
$$w_k(x)=\frac{e^{-\rho m^2(x-c_k)^2}}{\sigma^2_k(x)},$$
for some $\rho>0$. These weights are also continuous and data-driven and impose an exponential shrinkage, which depends on the distance from the subregion. {Furthermore, we note, that by choosing $\rho=C\log n$ and considering the one dimensional, unit interval case $\mathcal{D}^{(k)}=I^{(k)}$, the proportion of the exponential weights for points in the $k$th region and away from it can be bounded by $e^{-\rho m^2(y-c_k)^2}/e^{-\rho m^2(x-c_k)^2}\leq n^{-C} $, for $x\in I^{(k)}$ and $y\in I^{(j)}$, with $|j-k|> 2$. Hence the contribution of the local posteriors built on data sets not in the neighbourhood of the $k$th interval is negligible for $x\in I^{(k)}$. This approach provides therefore a continuous aggregated posterior which at the same time better resembles the localization properties of the standard, glue together approach than the one proposed in \cite{ng:deisenroth:14}.}  We show below numerically, both on synthetic and real world data sets, that this new aggregation technique substantially improves the performance of the distributed GP procedure, especially when the scale hyper-parameter is selected in a data-driven way.

\subsection{Synthetic datasets}\label{sec:synthetic}
In this section we investigate the performance of various distributed Gaussian process regression methods on synthetic data sets, and compare them to the benchmark: the non-distributed approach that computes the posterior distribution on all data. 

We consider recovery and confidence statements for the functional parameter $f_0$ based on $n$ independent data points  $(X_1,Y_1),\ldots,(X_n,Y_n)$
from the model
\begin{align*}
Y_i= f_0(X_i)+Z_i,\qquad Z_i\stackrel{iid}{\sim} \mathcal{N}(0,\sigma^2),\quad X_i\stackrel{iid}{\sim}U(0,1).
\end{align*}
We simulated the data with noise standard deviation $\sigma=1$ and the function $f_0$ defined by  coefficients $f_{0,j}$ relative to the cosine basis $\psi_j(x)=\sqrt{2}\cos(\pi(j-1/2)x)$, $j=1,2,...$. 

Next to the true posterior distribution, based on all data, we computed distributed posterior distributions, using four methods. Method 1 (M1) is the consensus Monte Carlo method proposed by \cite{scott:blocker:et.al:16} and applied to Gaussian Processes in \cite{szabo:vzanten:19,hadji2022optimal}. This method splits the data randomly between the machines (i.e.\ the $k$th machine receives a random subset of size $n_k=n/m$ from the observations) and compensates for working with only partial data sets by scaling the priors in the local machines by a factor $1/m$. A draw from the aggregated posterior is constructed as the average $f(x)=m^{-1}\sum_{k=1}^mf^{(k)}(x)$ of independent draws $f^{(k)}$ from each (modified) local posterior. Methods 2--4 all split the data spatially (i.e.\ the $k$th machine receives the pairs $(X_i,Y_i)$ for which $X_i\in I^{(k)}=\big(\frac{k-1}{m},\frac{k}{m}\big]$), and differ only in their aggregation technique. Method 2 (M2)  uses the standard ``glue together''  approach displayed in \eqref{def:aggr:spatial}, Method 3 (M3) uses the inverse variance weighted average \eqref{def:spatial:ng}, and Method 4 (M4) uses the exponential weights \eqref{def:novel:method}. 

All distributed methods were carried out on a single core, drawing sequentially from the local posteriors. Parallelising them over multiple cores or machines
would have  shortened the reported run times substantially, approximately by a factor $m$.

First we considered the Mat\'ern covariance kernel. We studied both a version with sample paths rescaled deterministically by the optimal length scale $\tau_n$ for the given true function $f_0$ and versions with data-based rescaling (via both empirical and hierarhical Bayes approaches) that do not use any information about $f_0$. While for the oracle choice  (depending on the typically unknown regularity $\beta$ of the underlying function $f_0$) of the scaling parameter all methods performed similarly well, for the data driven choices of the hyper-parameter there were substantial differences between the distributed Methods 1--4.  Spatially distributing the data (Methods 2 and 4) clearly outperformed random distribution (Method 1). This is in agreement with the theory, and can be explained by the inability to determine suitable scale parameters from the data in the randomly distributed case. However, it was also observed that the benefits of spatial distribution can be destroyed when smoothing out the inherent spatially discontinuities using aggregation weights that depend on the local length scales in the wrong way (Method 3).

Then we investigated whether the methods can adapt to different local regularities. We considered a true function $f_0$, which is rough in the first half of the co-domain and smooth in the second half. It is well known that stationary Gaussian processes are not appropriate for picking up different local behaviour as they localize the signal at the spectral not at the spatial domain. This can be also observed in our numerical analyzis for the non-distributed and the randomly distributed methods (BM and M1). However, by spatially dividing the data over the local machines one can pick up different local behaviours and can achieve substantially better estimations and uncertainty quantification for the smoother part of the signal than using the standard, non-distributed approach. Hence, in addition to significantly speeding up the computations, spatially distributed methods have the additional advantage of better learning the local properties of the signal by adapting to the local regularity. We deferred the corresponding numerical analysis to the supplement.

Finally, we have also investigated the popular squared exponential covariance kernel with data driven rescaling hyper-parameter, using both the empirical and hierarchical Bayes methods. Although this prior is not explicitly covered in our examples, we observe similar phenomenas as for the Mat\'ern covariance kernel. The corresponding simulation studies are deferred to the supplement.

As mentioned earlier, for adaptation we have used both the hierarchical and empirical Bayes procedures. In the empirical Bayes method we took the maximum marginal likelihood estimator (MMLE) of the scale parameter $\tau$, while in hierarchical Bayes we have endowed it with another layer of prior distribution.   The (MMLE) empirical Bayes method has been shown to behave similarly to the hierarchical Bayes method, considered in our theoretical study (see for instance \cite{sz:vaart:zanten:2013, Sniekers2,rousseau:szabo:2017,SniekersvdV2020}). In both approaches we computed first the marginal log-likelihood function on a (fine enough) grid using the  {\tt gpml} Matlab package. Then in the empirical Bayes method we selected the maximizer of this likelihood. In the hierarchical Bayes approach we used an exponential hyper-prior distribution on $\tau$ (which was approximated by a truncated geometric distribution on the chosen grid) and derived the corresponding marginal posterior of the hyper-parameter. Alternatively, one could also use the {\tt minimize} function built in the  {\tt gpml} package for estimating the hyper-parameters of the GP prior. However, this Matlab function approximates simultaneously various additional hyper-parameters as well. Since in our theoretical framework we have tuned only the length scale parameter $\tau$, for better connection to the preceding sections,  we have refrained from using this built in optimizer in the synthetic data set.simulation study.

To assess the quality of the recovery we report the $L_2$ error of estimating $f_0$ with the posterior mean. As a measure of the size of $L_2$-credible balls we report twice the root average posterior variance  $$r=2\sqrt{\int_0^1 \sigma^2(x| \textbf{X,Y})\,dx}.$$
We consider the true function to be in the credible ball if its $L_2$-distance to the posterior mean is smaller than $r$. Furthermore, we also investigate the point-wise behaviour of the posterior. We report both the length of the $95\%$ confidence interval $4\sigma(x)$ for some selected points $x$ and the corresponding local coverage probabilities. In case of the hierarchical Bayes approach we report the average credible bands with respect to the hyper-posterior.

 \subsubsection{Mat\'ern kernel}\label{sec:matern:synthetic}
 In our study with the Mat\'ern prior, we used this kernel  with regularity hyper-parameter $\alpha=3$ (see \eqref{def:Matern:spectral}),
and generated the data from the true parameter $f_0$ with coefficients $f_{0,j}=1.5\sin(j)j^{-1/2-\beta}$, with $\beta=1$, for $j\geq 4$ and $ f_{0,j}=0$ for $j\leq 3$, with respect to the cosine basis. This function $f_0$ is essentially $\beta$ smooth: $f_0$ belongs to the Sobolev space $H^{\gamma}([0,1])$ for all $\gamma<\beta$. The optimal length scale parameter of the prior is then $\tau_n=n^{(\alpha-\beta)/(1+2\beta)/\alpha}$, as seen in Section~\ref{sec:Matern}. 

We considered pairs $(n,m)$ of sample sizes and numbers of machines  equal to  $(2000,10)$, $(5000,20)$ and $(10000,50)$. 
In all test cases we repeated the experiment 100 times, except in the adaptive settings with $n\geq 5000$, where we considered only 20 repetitions, due to the overly slow non-distributed approach.
Posterior means and $95\%$ point-wise credible bands for a single experiment are visualized in Figures~\ref{fig:Mat:nonadapt},~\ref{fig:Mat:EB}, and~\ref{fig:Mat:HB}, for
the oracle (deterministic, optimal rescaling), empirical Bayes and hierarchical Bayes scaling, respectively. The average 
  $L_2$-errors, the sizes and frequentist coverages of the $L_2$ credible sets and the run times are reported in 
Tables~\ref{table: errorMatern_nonadapt}-\ref{table: timeMatern_nonadapt} and Table \ref{table:credible_nonadapt} in the supplement for the deterministic scaling, in 
Tables~\ref{table: errorMatern_EB}-\ref{table: timeMatern_EB} for the empirical Bayes method, and in Tables~\ref{table: errorMatern_HB}-\ref{table: timeMatern_HB}  for the hierarchical Bayes approach. Due to space restriction we report the point-wise analysis of these approaches in the supplement.

In the non-adaptive setting, where the GP was optimally rescaled, all methods performed similarly well. They all resulted in good estimators and reliable uncertainty statements. The run time of the distributed algorithms were similar and substantially shorter than for the non-distributed counterpart (on average below 1s in all cases).

\begin{figure*}[!t]%
\centering
\includegraphics[width=\textwidth]{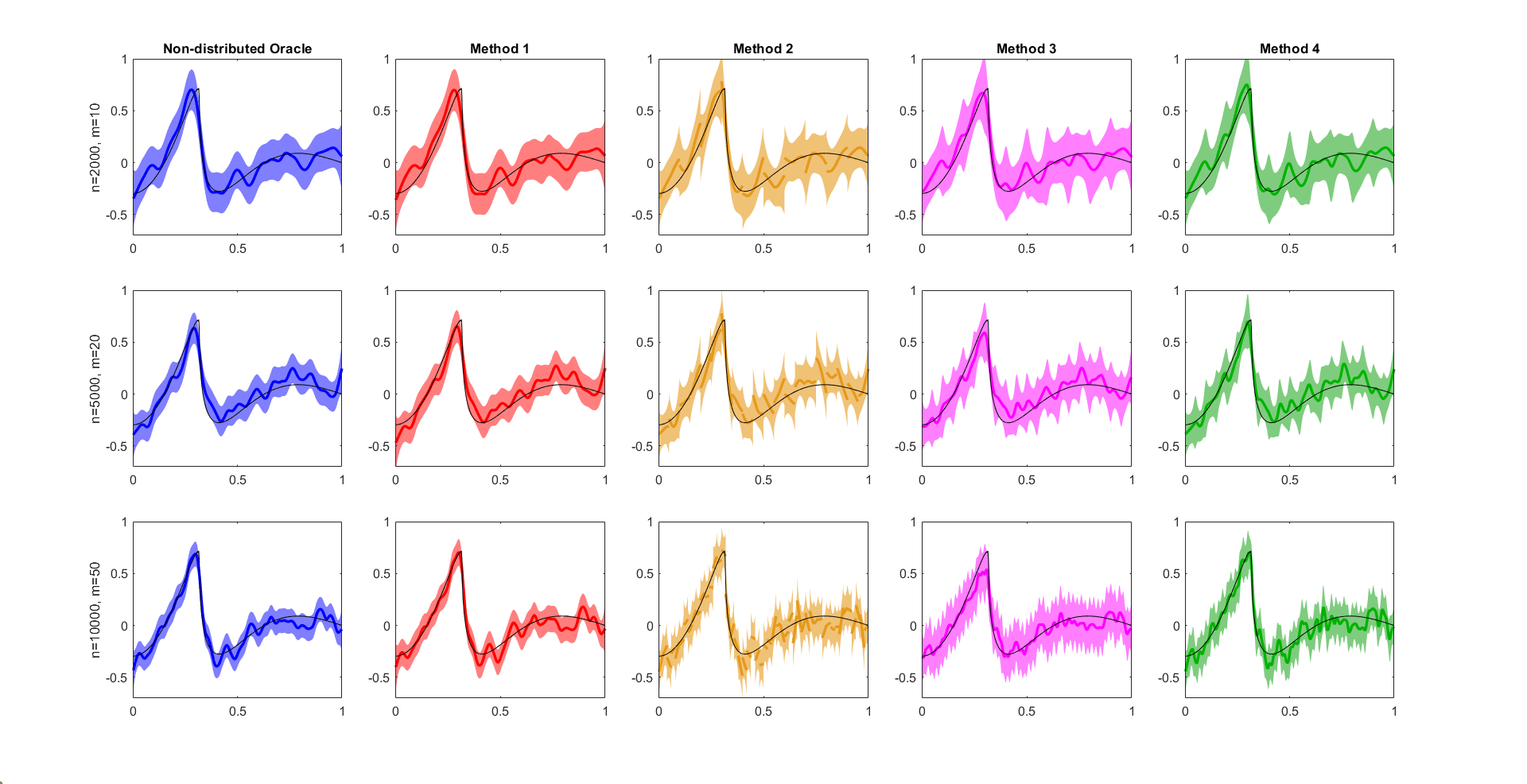}
	\caption{\scriptsize  Deterministic (oracle)  rescaling of the Mat\'ern process prior ($\alpha=3$). Benchmark and distributed GP posteriors. True function $ f_0(x)=\sum_{j=4}^{\infty}1.5j^{-3/2}\sin(j)\psi_j(x)$  drawn in black. Posterior means drawn by solid lines, surrounded by $95\%$ point-wise credible sets shaded between two dotted lines. The five columns correspond (left to right) to the non-distributed method, the distributed method with random partitioning, and the distributed methods with spatial partitioning without smoothing, with inverse variance weights and with exponential weights. From top to bottom the sample sizes are $n=2000,5000,10000$ and the number of experts  $m=10,20,50$.}
	\label{fig:Mat:nonadapt}
\end{figure*}

\begin{table}[!h]

\centering
\begin{tabular}{c|c|c|c}
(n,m)&$ (2000,10)$&  $(5000,20)$ & $(10000,50)$\\ \hline
 BM& 0.091 (0.014) & 0.068 (0.008) & 0.054 (0.007)\\
 M1& 0.093  (0.014) & 0.070 (0.008) & 0.057 (0.007)\\
 M2&  0.105 (0.016) & 0.086 (0.009) & 0.080 (0.007)\\
 M3& 0.090 (0.015) & 0.070 (0.008) & 0.069 (0.007)\\
 M4& 0.094  (0.015)  & 0.075 (0.008) & 0.065 (0.007)\\
\end{tabular}
	\caption{\scriptsize  Average $L_2$-distance between $f_0$ and posterior mean for deterministic (oracle)  rescaling of the Mat\'ern process prior (with $\alpha=3$). BM: Benchmark, Non-distributed method. M1: Random partitioning, M2: Spatial partitioning, M3: Spatial partitioning with inverse variance weights, M4: Spatial partitioning with exponential weights. Average values over 100 replications of the experiment with standard error in brackets.}
\label{table: errorMatern_nonadapt}
\end{table}

\begin{table}[!h]
\centering
	\begin{tabular}{c|c|c|c}
(n,m)&$ (2000,10)$&  $(5000,20)$ & $(10000,50)$\\ \hline
Benchmark & 0.897s (0.406s) & 9.379s (3.986s) & 53.86s (15.49s) \\
 Random & 0.121s  (0.081s) & 0.260s (0.120s) & 0.46s (0.43s)\\
Spatial & 0.114s (0.084s) & 0.235s (0.098s) & 0.44s (0.45s)\\
\end{tabular}
	
	\caption{\scriptsize Deterministic (oracle)  rescaling of the Mat\'ern process prior ($\alpha=3$). Average run time for computing the posterior. Benchmark: Non-distributed method. Method 1: Random partitioning, Method 2: Spatial partitioning.}
\label{table: timeMatern_nonadapt}
\end{table}

In the adaptive setting we considered both the empirical and hierarchical Bayes approaches. In the first method we estimate the scaling hyper-parameter with the MMLE, while in the second one we endow it with another layer of prior, resulting in a fully Bayesian, hierarchical procedure. In the latter case, as hyper-prior, we chose the exponential distribution with parameter $\lambda=1/5$ and approximated the hyper-posterior on a fine enough grid. One can observe that both data driven Bayesian methods performed similarly. In case of randomly distributing the data to local machines (M1) the aggregated posterior is over-smoothed and provides too narrow, overconfident uncertainty quantification. The standard spatially distributed approach (M2) performed well, but produced visible discontinuities. The aggregation approach (M3) provided poor and overconfident estimator using empirical Bayes method as is very evident in Figure~\ref{fig:Mat:EB}. Using hierarchical Bayes, Method 3 performed better, but the estimation accuracy and the size of the credible sets were still sub-optimally large, see Figure \ref{fig:Mat:HB} and teh corresponding tables. Our approach (M4) combined the best of both worlds: it provided continuous sample paths and maintained (and even improved) the performance of the standard glue-together spatial approach (M2), while substantially reducing the computational burden compared to the non-distributed approach. 

Here again we note, that by parallelized implementation of the algorithms the run time could be further reduced by a factor of $m$. For instance, in the last scenario of the hierarchical Bayes approach with $(n,m)=(10000,50)$ this would reduce the computation time of 10000 seconds needed for the non-distributed method to around 1 second.

\begin{figure*}[!t]%
\centering
\includegraphics[width=\textwidth]{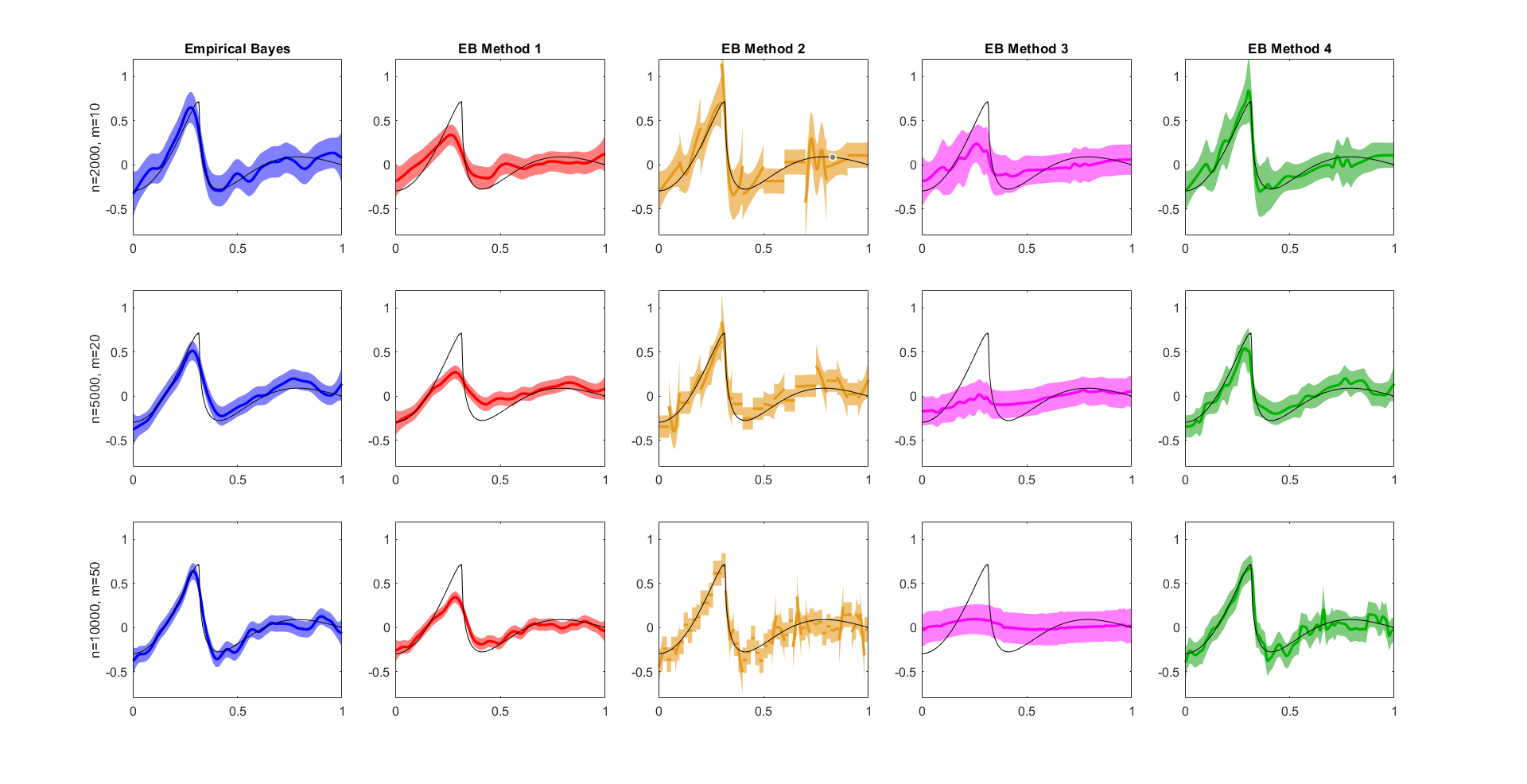}
	\caption{\scriptsize Empirical Bayes (MMLE) approach for the rescaled Mat\'ern process prior ($\alpha=3$). Benchmark and distributed GP posteriors. True function $f_0(x)=\sum_{j=4}^{\infty}1.5j^{-3/2}\sin(j)\psi_j(x)$  drawn in black. Posterior means drawn by solid lines, surrounded by $95\%$ point-wise credible sets, shaded between two dotted lines. The five columns correspond (left to right) to the non-distributed method, the distributed method with random partitioning, and the distributed methods with spatial partitioning without smoothing, with inverse variance weights and with exponential weights. From top to bottom the sample sizes are $n=2000,5000,10000$ and the number of experts is $m=10,20,50$.}
	\label{fig:Mat:EB}
\end{figure*}

\begin{table}[!h]
\begin{subtable}[c]{\textwidth}
\centering
\begin{tabular}{c|c|c|c}
(n,m)&$ (2000,10)$&  $(5000,20)$ & $(10000,50)$\\ \hline
 BM& 0.092 (0.013) & 0.067 (0.008) & 0.053 (0.006)\\
 M1& 0.136  (0.027) & 0.109 (0.026) & 0.118 (0.015)\\
 M2&  0.102 (0.018) & 0.083 (0.009) & 0.084 (0.009)\\
 M3& 0.184 (0.019) & 0.197 (0.011) & 0.205 (0.003)\\
 M4& 0.091 (0.017)  & 0.069 (0.009) & 0.057 (0.006)\\
\end{tabular}
	\caption{\scriptsize Average $L_2$-distance between $ f_0$ and posterior mean.}
\end{subtable}
\newline
\begin{subtable}[c]{\textwidth}
\centering
\begin{tabular}{c|c|c|c}
(n,m)&$ (2000,10)$&  $(5000,20)$ & $(10000,50)$\\ \hline
 BM& 0.160 (0.010) & 0.118 (0.007) & 0.093 (0.005)\\
 M1& 0.125 (0.020) & 0.090 (0.010) & 0.067 (0.006) \\
 M2& 0.182 (0.011) & 0.151 (0.006) &  0.155 (0.003)\\
 M3& 0.187 (0.008)  & 0.162 (0.002) &  0.176 (0.001))\\
 M4& 0.172 (0.010) & 0.143 (0.006) & 0.149 (0.002)\\
\end{tabular}
	\caption{\scriptsize Average radius of the  $L_2$-credible ball.}
	\label{table: radiusMatern_adapt}
\end{subtable}
	\caption{\scriptsize  Empirical (MMLE) Bayes rescaling of the Mat\'ern Gaussian process prior. BM: Benchmark, Non-distributed method. M1: Random partitioning, M2: Spatial partitioning, M3: Spatial partitioning with inverse variance weights, M4: Spatial partitioning with exponential weights.}
	\label{table: errorMatern_EB}
\end{table}

\begin{table}[!h]
\centering
\begin{tabular}{c|c|c|c}
(n,m)&$ (2000,10)$&  $(5000,20)$ & $(10000,50)$\\ \hline
 BM& 1.00 & 1.00 & 1.00\\
 M1& 0.49 & 0.25 & 0.00 \\
 M2& 0.98 & 1.00 & 1.00\\
 M3& 0.45   & 0.00 & 0.00\\
 M4& 0.96 & 1.00 & 1.00\\
\end{tabular}
	\caption{\scriptsize Empirical (MMLE) Bayes rescaling of the Mat\'ern Gaussian process prior. Proportion of experiments when the true function $f_0$ was inside in the $L_2$-credible ball.}
	\label{table: radiusMatern_EB}
\end{table}

\begin{table}[!h]
\centering
	\begin{tabular}{c|c|c|c}
(n,m)&$ (2000,10)$&  $(5000,20)$ & $(10000,50)$\\ \hline
Benchmark & 90.32s (17.51s) & 895.1s (86.5s) & 4767s (611s) \\
 Random & 6.31s  (1.95s) & 14.4s (3.3s) & 22.5s (4.0s)\\
Spatial & 6.36s (2.49s) & 13.8s (3.0s) & 21.2s (4.9s)\\
\end{tabular}
	
	\caption{\scriptsize Empirical (MMLE) Bayes rescaling of the Mat\'ern Gaussian process prior. Average run time for computing the posterior. Benchmark: Non-distributed method. Method 1: Random partitioning, Method 2: Spatial partitioning.}
\label{table: timeMatern_EB}
\end{table}

\begin{figure*}[!t]%
\centering
\includegraphics[width=\textwidth]{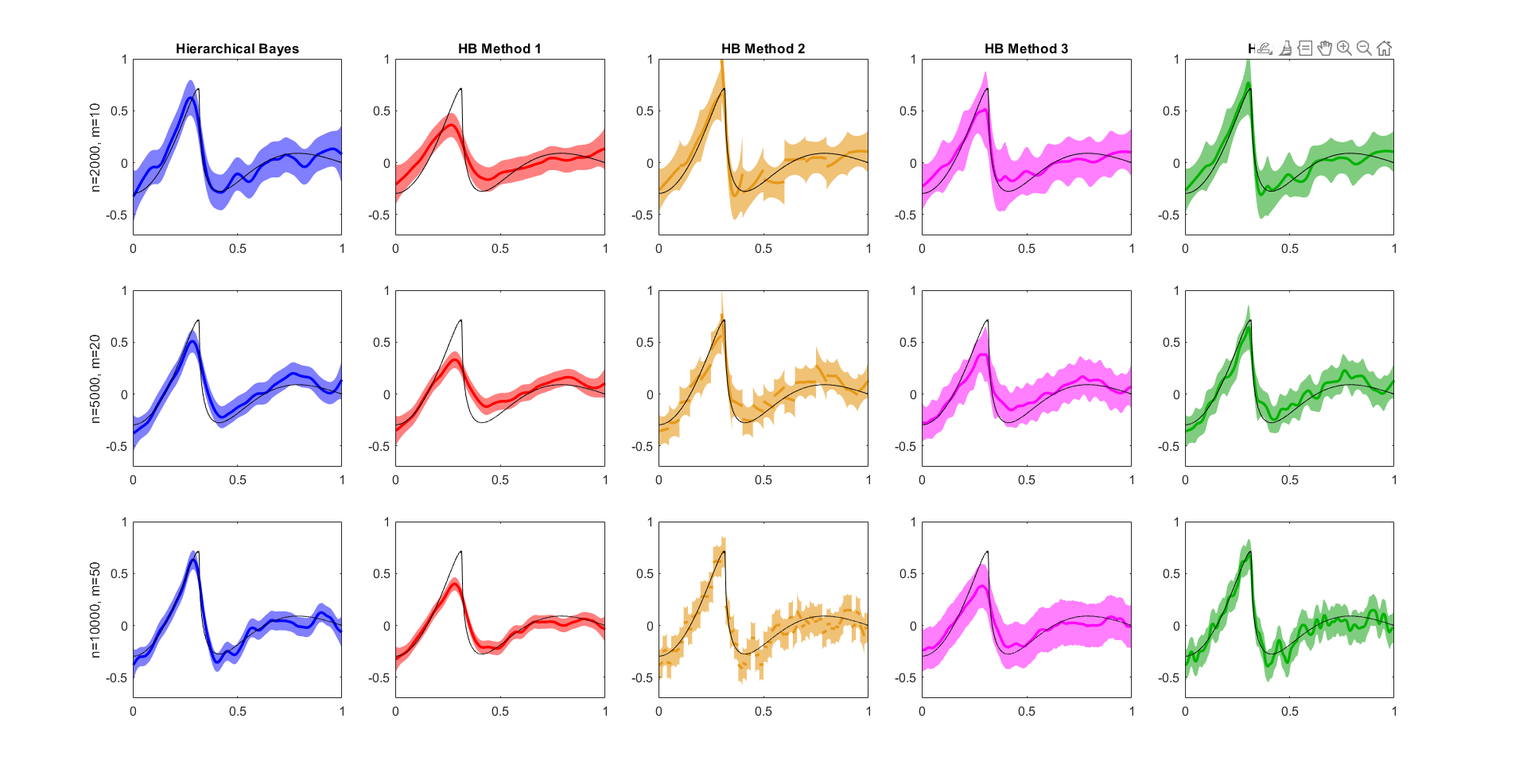}
	\caption{\scriptsize Hierarchical Bayes methods for the rescales Mat\'ern process prior ($\alpha=3$). Benchmark and distributed GP posteriors. True function $f_0(x)=\sum_{j=4}^{\infty}1.5j^{-3/2}\sin(j)\psi_j(x)$  drawn in black. Posterior means drawn by solid lines, surrounded by $95\%$ point-wise credible sets, shaded between two dotted lines. The five columns correspond (left to right) to the non-distributed method, the distributed method with random partitioning, and the distributed methods with spatial partitioning without smoothing, with inverse variance weights and with exponential weights. From top to bottom the sample sizes are $n=2000,5000,10000$ and the number of experts is $m=10,20,50$.}
	\label{fig:Mat:HB}
\end{figure*}

\begin{table}[!h]
\begin{subtable}[c]{\textwidth}
\centering
\begin{tabular}{c|c|c|c}
(n,m)&$ (2000,10)$&  $(5000,20)$ & $(10000,50)$\\ \hline
 BM& 0.093 (0.016) & 0.067 (0.008) & 0.056 (0.006)\\
 M1& 0.189  (0.019) & 0.095 (0.016) & 0.096 (0.007)\\
 M2&  0.120 (0.020) & 0.075 (0.009) & 0.074 (0.007)\\
 M3& 0.225 (0.001) & 0.087 (0.011) & 0.099 (0.005)\\
 M4& 0.100  (0.021)  & 0.060 (0.007) & 0.055 (0.007)\\
\end{tabular}
	\caption{\scriptsize Average $L_2$-distance between $ f_0$ and posterior mean.}
\end{subtable}
\newline
\begin{subtable}[c]{\textwidth}
\centering
\begin{tabular}{c|c|c|c}
(n,m)&$ (2000,10)$&  $(5000,20)$ & $(10000,50)$\\ \hline
 BM& 0.183 (0.025) & 0.117 (0.005) & 0.091 (0.004)\\
 M1& 0.122 (0.069) & 0.091 (0.006) & 0.066 (0.002) \\
 M2& 0.194 (0.031) & 0.147 (0.011) &  0.148 (0.001)\\
 M3& 0.100 (0.003)  & 0.184 (0.001) &  0.211 (0.001))\\
 M4& 0.157 (0.019) & 0.151 (0.001) & 0.152 (0.001)\\
\end{tabular}
	\caption{\scriptsize Average radius of the  $L_2$-credible ball}
	\label{table: radiusMatern_HB}
\end{subtable}
	\caption{\scriptsize Hierarchical Bayes rescaling of the  Mat\'ern Gaussian process prior. BM: Benchmark, Non-distributed method. M1: Random partitioning, M2: Spatial partitioning, M3: Spatial partitioning with inverse variance weights, M4: Spatial partitioning with exponential weights.}
	\label{table: errorMatern_HB}
\end{table}

\begin{table}[!h]
\centering
\begin{tabular}{c|c|c|c}
(n,m)&$ (2000,10)$&  $(5000,20)$ & $(10000,50)$\\ \hline
 BM& 1.00 & 1.00 & 1.00\\
 M1& 0.18 & 0.45 & 0.00 \\
 M2& 0.98 & 1.00 & 1.00\\
 M3& 0.00  & 1.00 & 1.00\\
 M4& 0.96 & 1.00 & 1.00\\
\end{tabular}
	\caption{\scriptsize Hierarchical Bayes rescaling of the  Mat\'ern Gaussian process prior. Proportion of experiments when the true function $f_0$ was inside in the $L_2$-credible ball.}
	\label{table: radiusMatern_HB}
\end{table}

\begin{table}[!h]
\centering
	\begin{tabular}{c|c|c|c}
(n,m)&$ (2000,10)$&  $(5000,20)$ & $(10000,50)$\\ \hline
Benchmark & 35.85s (6.48s) & 1292.2s (84.9s) & 10000s (1362s) \\
 Random & 4.97s  (1.32s) & 26.7s (0.8s) & 61.4s (12.8s)\\
Spatial & 5.25s (2.49s) & 25.8s (1.7s) & 58.6s (12.1s)\\
\end{tabular}
	
	\caption{\scriptsize Hierarchical Bayes rescaling of the Mat\'ern Gaussian process prior. Average run time for computing the posterior. Benchmark: Non-distributed method. Method 1: Random partitioning, Method 2: Spatial partitioning.}
\label{table: timeMatern_HB}
\end{table}

\subsection{Real world dataset: Superconductivity}\label{sec:realworld}

 Superconducting materials lose their resistance when they are cooled down below a certain temperature, called critical temperature, and as a consequence can conduct current with zero resistance. Materials with this property are used for instance in magnetic resonance imaging (MRI) and nuclear magnetic resonance (NMR) applications. Therefore, predicting the critical temperature is an important problem of wide interest. 

In our analysis we consider the Superconductivity data set \cite{superconductivty_data_464,Hamidieh2018ADS}. It contains 81 covariates describing the superconductor’s elemental properties and the goal is to predict the critical temperature based on them. In total  21263 measurement points were collected. In our analysis we have divided the data randomly into a training and testing data set consisting of 15000 and 6263 measurements, respectively. We have repeated the experiment ten times to measure the variability of the result. We compared the different distributed GP approaches (M1-M4) to the benchmark non-distributed approach (BM). In the spatially distributed methods we have split the data amongst the machines with respect to the wtd$\_$mean$\_$atomic$\_$mass variable. We have considered the squared exponential covariance kernel and selected the hyper-parameters with the minimize subroutine built in the gpml MATLAB package. In Method 4 we set the weight parameter $\rho=4$. The results are reported in Table \ref{table: superconductivity}. 

One can conclude that the naive (M2) and the exponentially re-weighted (M4) spatially distributed methods performed the best, similarly well to the benchmark non-distributed approach. At the same time, the product of experts method (M1) and the spatially distributed method with aggregation weights proportional to the inverse of the posterior variance (M3) performed sub-optimally, providing around $30\%$ bigger error. At the same time the distributed methods were around two magnitudes faster than the non-distributed counterpart, and their speed could further increase by parallelizing the computations instead of sequentially executing them, as in our analysis. Finally, we have also considered splitting the data in the spatially distributed methods with respect to other covariates as well. In view of Table \ref{table: superconductivity:spatial}, methods M2 and M4 are robust with respect to the splitting approach and provide similarly accurate predictions. The only requirement is that the feature used for splitting does not contain too many repetitions, which would result in imbalanced group sizes.

\begin{table}[!h]
\centering
\begin{tabular}{c|c|c|c|c|c}
Methods &BM&M1&M2&M3&M4\\
\hline\hline
RMSE & 12.63 (0.21)  & 15.58 (0.19) & 13.14 (0.35) &  16.63 (1.40) & 12.69 (0.32)\\
\hline
runtime & 18740s (1720s)  & 248s (82s) & 220s (23s) & 220s (23s) & 220s (23s)
\end{tabular}
\caption{\scriptsize{RMSE and runtime of non-distributed (BM) and distributed (M1-M4) GP regression with squared exponential covariance kernel for the  Superconductivity data set \cite{superconductivty_data_464}.} }\label{table: superconductivity}
\end{table}

\begin{table}[!h]
\centering
\begin{tabular}{c|c|c|c|c|c}
Methods &wtd\_range &mean &wtd\_mean&gmean&wtd\_entropy\\
\hline\hline
M2& 13.72 (0.51)  & 13.46 (0.35) &  13.14 (0.35) & 13.69 (0.64)&  13.49 (0.31) \\
\hline
M3& 19.25 (1.28) & 24.34 (0.88) & 16.63 (1.40) & 27.57 (1.50) & 28.16 (1.14) \\
\hline
M4& 13.27 (0.42)  & 13.14 (0.42) & 12.69 (0.32) & 13.10 (0.41s) & 12.91 (0.23)\\
\end{tabular}
\caption{\scriptsize{RMSE for spatially distributed GP regression methods with squared exponential covariance kernel using different feature variables (in all cases we omitted the "\_atomic\_mass" from their names) for splitting the data in the Superconductivity data set \cite{superconductivty_data_464}.} }\label{table: superconductivity:spatial}
\end{table}

\section{Discussion}\label{sec:discussion}
The paper provides the first theoretical guarantees for the method of spatial distribution applied to Gaussian processes. Our general results show that the resulting approximation to the posterior provides optimal recovery (both in the case of known and unknown regularity parameter) of the underlying functional parameter of interest in a range of models, including the nonparametric regression model with Gaussian errors and the logistic regression model. As specific examples of priors we considered the popular Mat\'ern process and integrated Brownian motion, but in principle other GP priors could be covered as well. The theoretical findings are complemented with a numerical analysis both on synthetic and real world data sets, where we also proposed a novel aggregation technique for aggregating the local posteriors together, which empirically outperformed the close competitors.

The main advantage of spatial distribution of the data is the ability to adapt the length scale of the prior in a data-driven way, which was highlighted by both theory and numerical illustration. The latter showed that the combination technique of the local posteriors is highly important and can substantially influence the performance of the method.  We also demonstrated that spatially distributed GPs can adapt to different local regularities of the true function, hence can potentially outperform the original GP.

Our results, although formulated for Gaussian processes, in principle rely on general Bayesian nonparametric techniques, adapted to the spatially distributed architecture and hence can be potentially extended to other classes of priors. Also, in the theoretical results univariate functional parameters were considered for simplicity, but the results could be extended to higher dimensional covariates. Another interesting extension is to derive theoretical guarantees for the proposed aggregation techniques beyond the ``glue together'' approach covered by this paper. These extensions, although of interest, are left for future work.

\clearpage

\appendix

\section{Comparison of Scalable Gaussian Process approximations}

There is a wide and growing literature on Gaussian Process approximations. However, these procedures until now had limited theoretical, frequentist underpinning and was unclear how they have to be tuned to achieve optimal statistical inference. Recent years have seen several new results on understanding the frequentist properties of these approaches, however, there are still many questions left unanswered. Here we collect a few recent results and compare with the guarantees derived in our paper.

One of the most popular and frequently used GP approximation technique is the inducing variable variational Bayes approach introduced by  \cite{titsias:2009}. Here, to compress the information present in the data, $m$ linear functionals $u_1,....,u_m,$ of the GP $f$ is considered, called the inducing variables. Then the variational class is defined by endowing the inducing variables $(u_1,...,u_m)$ with an $m$ dimensional Gaussian distribution $N_m(\mu,\Sigma)$, and integrating out the conditional distribution of $f|(u_1,...,u_m)$ with respect to this distribution. This results in a class of Gaussian Processes indexed by $\mu\in\mathbb{R}^{m}$ and $\Sigma\in\mathbb{R}^{m\times m}$ with mean and covariance functions of the form
\begin{align*}
x \mapsto K_{x\bs{u}}K_{\bs{uu}}^{-1}{\mu},\\
(x,z) \mapsto k(x,z)-K_{x\bs{u}}K_{\bs{uu}}^{-1}(K_{\bs{uu}}-\Sigma)K_{\bs{uu}}^{-1}K_{\bs{u}z},
\end{align*}
where $k(\cdot,\cdot)$ is the covariance kernel of the prior and the matrices are defined as $K_{x\bs{u}} = \cov_\Pi(f(x),\bs{u}) = K_{\bs{u}x}^T$ and $K_{\bs{uu}} = [\cov_\Pi(u_i,u_j)]_{1 \leq i,j \leq m}$. One can consider different choices of inducing variables. Examples include the inducing points method where $u_j=f(z_j)$ for some appropriately selected $z_j\in\{x_1,...,x_n\}$. After selecting the inducing points $f(z_1),...,f(z_m)$ the computational complexity is $O(nm^2)$. Various approaches were considered to select the inducing variables, including the $m$-Determinantal Point Processes method. However, sampling from this method can be computationally infeasible or at least very challenging \cite{belabbas2009spectral,kulesza2011k}, hence in practice often its approximations are used, see for instance \cite{anari2016monte} where an $\eps$ $m$-DPP point process method was introduced with computational complexity $O(nm^4\log(n)+nm^3\log(1/\eps))$, where $\eps$ denotes the approximation error. Another approach is to consider population spectral feature inducing variables (or eigenfunction inducing features) of the form $u_j=\int f\psi_j \lambda(x)dx$, where $\psi_j$ is the $j$th eigenfunction of the covariance kernel $k$ of the GP $f$ and $\lambda(x)$ is the density of the random design. This approach has computational complexity of $O(nm^2)$, but requires the knowledge of the eigenfunctions of the covariance kernel, which is often not available in an explicit form (e.g. the Mat\'ern kernel). A more practical version of this method is the empirical spectral features inducing variables method with $u_j=\boldsymbol{f}^T \hat{v}_j$, where $\boldsymbol{f}=\big(f(x_1),...,f(x_n)\big)^T\in\mathbb{R}^n$ denotes the GP evaluated at the design points and $ \hat{v}_j$ is the $j$th eigenvector of the empirical covariance matrix $\text{Cov}_{\Pi}(f(x_i),f(x_j))_{i,j}$. This approach does not require the knowledge of the eigenfunctions of the covariance kernel, however, involves the computation of the first $m$ eigenvectors of the covariance matrix. This in practice is done by using Lanczos iteration or conjugate gradient descent algorithms, see e.g. the IterGP method \cite{wenger2022posterior} from probabilistic numerics for a general approach and references therein.

In \cite{burt2019} it was shown that  the expected Kullback-Leibler divergence (with respect ot the prior) of the variational approximation and the posterior tends to zero if the number of inducing variables achieves a sufficient limit. Then, in \cite{nieman2022}, general frequentist contraction rates were derived for these inducing variables methods. In particular, it was shown that for $\beta$ H\"older smooth functions optimal contraction rates can be achieved by optimally rescaled Gaussian processes with Mat\'ern or squared exponential kernels with the population and empirical spectral features variational methods if the number of inducing variables exceeds $m\geq n^{d/(2\beta+d)}$. In \cite{nieman2023} it was shown that the preceding bound on the inducing variables, in context of the population spectral features method, is also a necessary condition. Furthermore, under this assumption, also frequentist coverage guarantees were derived. For the Lanczos and conjugate gradient approaches minimax posterior contraction rates were derived in \cite{stankewitz2024} if the number of iterations exceed $m\geq n^{d/(2\beta+d)}\log n$ for estimating $\beta$-smooth functions with optimally rescaled GP priors. In summary, if one has access to the eigenfunctions of the covariance kernel, then the computational complexity for optimally estimating $\beta$-smooth functions with rescaled GP priors is $O(n^{(2\beta+3d)/(2\beta+d)})$, but in case the eigenfunctions are not explicitly available, then the computational complexity exceeds $O(n^{(4\beta+3d)/(2\beta+2d)})$ for the empirical version of the algorithm and $O(n^{(2\beta+4d)/(2\beta+d)}\log n)$ for the approximate $m$-DPP process. In contrast the spatially distributed GP approach with $m=n^{d/(2\beta+d)}$ machines has computational complexity $O(n^{6\beta/(2\beta+d)})$. Therefore, one can observe that for higher regularities the inducing variable method, while for lower regularities the distributed approach has lower computational complexity. We also note that currently there are no frequentist guarantees available for inducing variable variational Bayes methods in the adaptive setting where the regularity parameter $\beta$ is not available.

Another popular and frequently used approach is the Vecchia approximation for GPs, see e.g. \cite{datta2016hierarchical,katzfuss:guiness:2021,peruzzi2022}. The underlying idea of the approximation is to introduce sparsity in the spatial dependency structure that is characterized by a directed acyclic graph (DAG). Therefore, this approach is not directly approximating the posterior, but considers a computationally more efficient version of the GP prior distribution. Let us denote the size of the parent sets, i.e. the number of points influencing the conditional distribution of the Gaussian process at a given point, by $m$. Then the computational complexity of the Vecchia posterior is $O(n m^3)$. In \cite{szabo:zhu:2024} it was shown that both the deterministic (oracle) and hierarchical Bayes rescaling of the GP prior with Mat\'ern covariance kernel results in optimal contraction rate for well chosen parent sets with constant size $m$ (depending only on $d$, $\alpha$ and $\beta$ only). Hence Vecchia GPs with linear computational complexity can achieve minimax contraction rates. However, it is worth emphasizing again, that in this case a different, spatially sparse prior is used in the procedure. 

There are many other GP approximation methods proposed in the literature, including for instance techniques based on covariance tapering \cite{furrer2006covariance,kaufman2008covariance,stein2013statistical} or compositional likelihoods \cite{bai2012joint,bevilacqua2015comparing}. However, these approaches typically do not have frequentist underpinning and hence it is unclear how they should be tuned  to achieve rate optimal inference on the functional parameter of interest and what is the corresponding computational complexity.

\section{Proofs of the general results}\label{sec:proof:general}
In this section we collect the proofs of the general contraction rate theorems for the non-adaptive and adaptive frameworks with known and unknown regularity parameter, respectively.

\subsection{Proof of Theorem~\ref{thm: nonadaptive}}\label{sec:proof:nonadapt}
The $k$th local problem is a non-i.i.d.\ regression problem of the type considered in
\cite{ghoshal:vaart:2007}, with a Gaussian process prior as considered in \cite{vaart:zanten:2008}, but
with the number of observations reduced to $n/m$. These papers give a rate of contraction of the
local posterior distribution $\Pi^{(k)}( \cdot|\textbf{Y}^{(k)})$
relative to the root of the average square Hellinger distance defined as 
$$d_n^{(k)}(f,g)^2=\frac{m}{n} \sum_{x_i\in \mathcal{D}^{(k)}}\int (\sqrt{p_{f,i}}-\sqrt{p_{g,i}})^2 d\mu_i.$$
Under Assumption~\ref{ass:metric} this square semimetric is bounded above by the square of the local
uniform metric $\|\cdot\|_{\infty,k}$ and so are the average local Kullback-Leibler divergence and variation.
Therefore Theorem~4 of \cite{ghoshal:vaart:2007} shows that the solution $\eps_n$ to the inequality
$\phi_{ f_0}^{(k)}(\eps_n)\leq (n/m)\eps_n^2$ gives a rate of contraction relative to $d_n^{(k)}$.
(Theorem~3.3 in \cite{vaart:zanten:2008} shows that in the Gaussian regression problem the  same
is true with the local root average square Hellinger distance replaced by the local empirical $L_2$-distance.)
Inspection of the proof (or see (b) of Theorems~8.19-8.20 in \cite{ghosal2017fundamentals}) shows that this
rate statement can be understood in the sense that 
\begin{align}
\E_0 \Pi^{(k)}\left( f:\, d_{n}^{(k)}(f, f_0)^2\geq M_n \eps_n^2|\textbf{Y}^{(k)}\right)\lesssim \frac{1}{(n/m)\eps_n^2},
  \label{EqLocalRate}
\end{align}
thus including the assertion that the left side tends to zero at the speed $m/(n\eps_n^2)$. 
Then, because $d_n^2=m^{-1}\sum_{k=1}^m(d_n^{(k)})^2$,
\begin{align*}
&\E_0\Pi_{n,m}\left( f:\, d_{n}(f, f_0)^2\geq M_n \eps_n^2|\textbf{Y}\right)\\
&\quad\leq  \E_0\Pi_{n,m}\left( f:\,\exists k \ d_{n}^{(k)}(f, f_0)^2 \geq M_n \eps_n^2|\textbf{Y}\right)\\
&\quad\leq \sum_{k=1}^m \E_0\Pi^{(k)} \left(f:\,d_{n}^{(k)}(f, f_0)^2 \geq M_n \eps_n^2|\textbf{Y}^{(k)}\right)\lesssim \frac{m^2}{n\eps_n^2},
\end{align*}
which tends to zero by the assumption that $n\eps_n^2/m^2\rightarrow \infty$. This concludes the proof of the first statement of the theorem.

Theorems~8.19-8.20 in \cite{ghosal2017fundamentals} show that the left side of  \eqref{EqLocalRate} is bounded by the probability of failure of the evidence lower bound (the left side of Lemma~8.21) plus terms of  the order $e^{-D(n/m)\e_n^2}$,  for some positive constant $D$. Under Assumption~\ref{ass:metric} the former probability gives the order on the right side of \eqref{EqLocalRate}. We show below in Lemma~\ref{lem: LB:denom} that in the nonparametric and logistic regression models the probability of failure of the evidence lower bound is actually also of the smaller exponential order, whence the right side of \eqref{EqLocalRate} can be replaced by $e^{-D(n/m)\e_n^2}$. The right side of the final display in the preceding proof then improves to $m e^{-D(n/m)\e_n^2}$, which tends to zero if $m\le n\e_n^2/(D'\log n)$ for a sufficiently large
constant $D'$. This proves the second statement of the theorem.

\subsection{Proof of Theorem~\ref{thm: adaptation}}\label{sec:proof:adaptation}
Under Assumption~\ref{ass:metric} the average square local Hellinger distance $d_n^{(k)}$ and the local average
Kullback-Leibler divergence and variation are bounded above by the square local uniform norm.  
Therefore conditions \eqref{eq: rem-mass}-\eqref{eq: entropy} imply the conditions of Theorem~4
of \cite{ghoshal:vaart:2007} applied to the $k$th local problem, with $n/m$ observations and prior
$\Pi^{g,(k)}$. Thus this theorem gives a contraction rate $\eps_n$ of the $k$th local posterior distribution relative
to $d_n^{(k)}$. In view of (b) of Theorems~8.19-8.20 in \cite{ghosal2017fundamentals}, 
this rate can be understood in the sense of \eqref{EqLocalRate}, with $\Pi^{g,(k)}$ taking the place of $\Pi^{(k)}$,
and improved to exponential order in the case of the Gaussian and logistic regression models.
The proof can be finished as the proof of Theorem~\ref{thm: nonadaptive}.

\section{Proofs for the examples}\label{sec:proof:examples}
In this section we collect the proofs of the corollaries for the integrated Brownian motion
and Mat\'ern prior processes.

\subsection{Proof of Corollary~\ref{cor:IBM:nonadaptresc}}\label{sec:IBM:nonadaptresc}
Assumption~\ref{ass:metric} holds both in the nonparametric and logistic regression models, 
and for $\eps_n=n^{-\beta/(2\beta+1)}$ we have $n\e_n^2=n^{1/(2\beta+1)}$ so that $m=o(n\eps_n^2/\log n)$ by assumption. Therefore the corollary is a consequence of Theorem~\ref{thm: nonadaptive} provided the remaining condition
of this theorem, the modulus inequalities $\phi_{ f_0}^{(k)}(\eps_n)\leq (n/m)\eps_n^2$, is satisfied.

Let $(G^{\tau,(k)}_t: t\in I^{(k)})$ be rescaled Integrated Brownian motion \eqref{eq: intBMresc} on the interval $I^{(k)}$ and let $\mathbb{H}^{\tau,(k)}$ be its RKHS.  
We prove below that, for $f_0\in C^{\beta}([0,1])$ and $\beta\leq \ell+1/2$, if $m\le \tau/\eps^{1/(\ell+1/2)} $ and $B_n\ge \eps$ with
\begin{align}
&\frac{B_n^2}m\gtrsim  \eps^{2(1+(\ell-\beta)\wedge 0)/\beta}(\tau\vee\tau^{2\ell+1}),\label{eq:cond:B}
\end{align}
then the following bounds are satisfied up to constants that do not depend on $\eps$, $\tau$, $m$ and $n$
\begin{align}
\label{eq: resc-inf-norm}
\inf_{h:\| f_0-h\|_{\infty,k}\leq\eps}\|h\|^2_{\mathbb{H}^{\tau,(k)}}&\lesssim \frac{\tau^{-2\ell-1}}{m\eps^{\frac{2\ell+2-2\beta}{\beta}}},\\
\label{eq: resc-log-prior}
-\log \Pr\bigl( \| G^{\tau,(k)}\|_{\infty,k}\!<\!\eps\bigr)&\lesssim\frac{\tau}{m\eps^{\frac{1}{\ell+1/2}}}+\log\!\frac{B_n(\tau^{\ell}\!\vee\! 1)}{\eps^2}.
\end{align}
(More precisely, \eqref{eq: resc-inf-norm} needs \eqref{eq:cond:B}, but \eqref{eq: resc-log-prior} is valid without it.)
Next we choose $\tau=\tau_n\asymp n^{\frac{\ell+1/2-\beta}{(\ell+1/2)(2\beta+1)}}$ and  $\eps=\eps_n= n^{-\beta/(2\beta+1)}$,
so that condition \eqref{eq:cond:B} is satisfied if $ B_n^2\geq m n^{\frac{-1+2(\ell-\beta)\vee 0}{1+2\beta}}$, which is true
by assumption. Furthermore, 
with these choices $\tau_n/\eps_n^{1/(\ell+1/2)}= n^{1/(1+2\beta)}\ge m$, by assumption,
and $(\tau_n^\ell\vee 1)/\eps_n^2$ is a power of $n$, so that the second term on the right side of \eqref{eq: resc-log-prior} is bounded by a multiple of the 
first if  $(\log B_n\vee 0)+\log n\le n^{1/(1+2\beta)}/m$, which is also true by assumption. Combining \eqref{eq: resc-inf-norm} and \eqref{eq: resc-log-prior}, we then get
\begin{align*}
 \phi_{ f_0}^{(k)}(\eps_n)\lesssim \frac{\tau_n^{-2\ell-1}}{m\eps_n^{(2\ell+2-2\beta)/\beta}}+\frac{\tau_n}{m\eps_n^{1/(\ell+1/2)}}\lesssim
\frac {n\eps_n^2}m.
\end{align*}

\paragraph{Proof of \eqref{eq: resc-inf-norm}}
By Whitney's theorem, a function $f_0\in C^{\beta}([0,1])$ can be extended to a function $f_0$ on the full line $\mathbb{R}$ of the same H\"older norm and with compact support. Let $\psi$ be a smooth $\ell$th order kernel (a function with $\int\psi(s)\,ds=1$, $\int s^l\psi(s)\,ds=0$ for all $l\leq\ell$, and $\int|s|^{\ell+1}|\psi(s)|\,ds<\infty$), 
and let $f_0\ast\psi_{\sigma}$ be the convolution between $ f_0$ and the scaled version 
$\psi_{\sigma}(\cdot):=\sigma^{-1}\psi(\cdot/\sigma)$ of $\psi$, for $\sigma>0$. Then, for $l\in\mathbb{N}$,
\begin{align}
\sup_{0\leq t\leq 1}| f_0- f_0\ast\psi_{\sigma}|(t)&\lesssim\sigma^{\beta},\label{EqConvDistance}\\
\sup_{0\leq t\leq 1}|( f_0\ast\psi_{\sigma})^{(l)}(t)|&\lesssim\Bigl(\frac1\sigma\Bigr)^{(l-\beta)\vee 0}.\label{eq: UB:supnorm:approx}
\end{align}
For proofs, see the proof of Lemma~11.31 in \cite{ghosal2017fundamentals}, and at the end of this section. 

The local RKHS $\mathbb{H}^{\tau,(k)}$ of $G^{\tau,(k)}$ is characterized in 
Lemma~\ref{lem: def-rescRKHS} and~\ref{lem: def-locRKHS} below. Set $I^{(k)}=\left(\frac{k-1}{m},\frac{k}{m}\right]$
and consider the function $h$ defined by
\begin{align}
h(t):=I^{\ell+1}\left(( f_0\ast\psi_{\sigma})^{(\ell+1)}\textbf{1}_{I^{(k)}}\right)(t)+\sum_{r=0}^{\ell}\frac{(t-\frac{k-1}{m})^r}{r!}( f_0\ast\psi_{\sigma})^{(r)}\left(\frac{k-1}{m}\right).\label{def:rkhs:approx}
\end{align}
The function $I^{\ell+1}\bigl(( f_0\ast\psi_{\sigma})^{(\ell+1)}\textbf{1}_{I^{(k)}}\bigr)$ has $(\ell+1)th$ derivative equal to
$( f_0\ast\psi_{\sigma})^{(\ell+1)}\textbf{1}_{I^{(k)}}$ and derivatives of orders $0,1,\ldots,\ell$ at $(k-1)/m$ equal to 0,
while the polynomial part of $h$ is set up to have vanishing $(\ell+1)$th derivative and derivatives of orders $0,1,\ldots,\ell$ at $(k-1)/m$ equal to the derivatives of $f_0 \ast \psi_\sigma$ at this point. If follows that $h= f_0\ast\psi_{\sigma}$ on $I^{(k)}$, whence
$\|h- f_0\|_{\infty,k}\leq\sigma^{\beta}$ by \eqref{EqConvDistance}. Moreover, in view of Lemmas~\ref{lem: def-rescRKHS} and~\ref{lem: def-locRKHS}, 
the local RKHS norm of $h$ satisfies
$$\|h\|^2_{\mathbb{H}^{\tau,(k)}}\le \sum_{j=0}^{\ell}\frac{h^{(j)}(0)^2}{B_n^2\tau^{2j}}+\frac1{\tau^{2\ell+1}}\int_{\frac{k-1}{m}}^{\frac{k}{m}}( f_0\ast\psi_{\sigma})^{(\ell+1)}(t)^2\,dt.$$
In view of \eqref{eq: UB:supnorm:approx} the second term in the right side is bounded from above by $\tau^{-2\ell-1}\sigma^{-2(\ell+1-\beta)}/m$. Because the first, integral part of $h$ vanishes at 0, only the polynomial part contributes to the first term of the RKHS norm.
Without the factor $1/B_n^2$ this is equal to 
\begin{align*}
&\sum_{j=0}^{\ell}\frac{1}{\tau^{2j}}\biggl(\sum_{r=j}^{\ell}\Bigl(-\frac{k-1}{m}\Bigr)^{r-j}\frac1{(r-j)!}( f_0\ast\psi_{\sigma})^{(r)}\Big(\frac{k-1}{m}\Big)\biggr)^2\\
&\quad\leq\sum_{j=0}^{\ell}\frac1{\tau^{2j}}\Bigl(\sum_{r=j}^{\ell}\frac1{\sigma^{(r-\beta)\vee 0}}\Bigr)^2 
\lesssim \Bigl(\frac1\sigma\Bigr)^{(2\ell-2\beta)\vee 0}(\tau^{-2\ell}\vee 1).
\end{align*}
Taking these together and choosing $\sigma=\eps^{1/\beta}$, we see that
\begin{align*}
\inf_{h:\| f_0-h\|_{\infty,k}\leq\eps}\|h\|^2_{\mathbb{H}^{\tau,(k)}}
  &\lesssim \frac{\tau^{-2\ell}\vee 1}{B_n^2\eps^{\frac{2\ell-2\beta}{\beta}\vee0}}
    +\frac{\tau^{-2\ell-1}}{m\eps^{\frac{2\ell+2-2\beta}{\beta}}}.
\end{align*}
Under assumption \eqref{eq:cond:B}, the right side of this display is dominated by its second term.
This concludes the proof of assertion \eqref{eq: resc-inf-norm}.

\paragraph{Proof of \eqref{eq: resc-log-prior}}
By the independence of the two components of the process $G^{\tau,(k)}$, defined in \eqref{eq: intBMresc},
\begin{align*}
\Pr\bigl(\| G^{\tau,(k)}\|_{\infty,k}<\eps\bigr)
 \geq \Pr\Bigl(\|I^{\ell}W_{\tau t}\|_{\infty,k}<\frac\eps2\Bigr)\Pr\Big(\Bigl\|B_n\sum_{j=0}^{\ell}\frac{Z_j(\tau t)^j}{j!}\Bigr\|_{\infty,k}<\frac\eps2\Big).
\end{align*}
In view of  Lemma~\ref{lem: IBM-int} (with $x_0=\tau(k-1)/m$ and $x_1=\tau k/m$), the first term on the right,
the small ball probability of the scaled $\ell$-fold integrated Brownian motion on $I^{(k)}$, is bounded from below by a multiple of 
$$\exp\Big\{-q\frac{\tau/m }{\eps^{1/(\ell+1/2)}}-\log \frac 2\eps\Big\},$$
for some $q>0$,  provided that $(\tau/m)\eps^{-1/(\ell+1/2)}\ge 1$.
On the other hand, by the independence of the random variables $Z_j$, the second term is bounded below by
\begin{align*}
\Pr\Bigl(\max_{0\le j\le \ell}|Z_j|<\frac{\eps}{2B_ne(\tau^{\ell}\lor 1)}\Bigr)
=\Big(2\Phi\Bigl(\frac{\eps}{2B_ne(\tau^{\ell}\lor 1)}\Bigr)-1\Big)^{\ell+1}
\gtrsim\Big(\frac{\eps}{B_ne(\tau^{\ell}\lor 1)}\Big)^{\ell+1},
\end{align*}
provided $4B_ne>\eps$, since $2\Phi(x)-1\ge xe^{-1}$, for $0<x<2$. Combining the last two displays,
we find \eqref{eq: resc-log-prior}.

\paragraph{Proof of \eqref{eq: UB:supnorm:approx}}
The derivative of the convolution $f\ast g$ is given as $\left(f\ast g\right)'=f'\ast g=f\ast g'$, which can be iterated to obtain derivatives of higher degree.
In the case that  $\beta> l$ we use this to see that $(f_0\ast\psi_{\sigma})^{(l)}= f_0^{(l)}\ast\psi_{\sigma}$, which is uniformly bounded by $\|f_0^{(l)}\|_\infty$.
In the case that $\beta\le l$, we write $( f_0\ast\psi_{\sigma})^{(l)}= f_0^{(b)}\ast\psi^{(l-b)}_{\sigma}$, for $b$ the largest integer strictly
smaller than equal $\beta$.  
Since $\psi$ is smooth and integrates to 1, we have $\int \psi^{(l)}(s)\,ds=0$, for $l=1,2,\ldots$. Therefore
$$(f_0\ast\psi_\sigma)^{(l)}(t)
=\int \bigl(f_0^{(b)}(t-\sigma s)-f_0^{(b)}(t)\bigr)\psi^{(l-b)}(s)\,ds\frac1{\sigma^{l-b}}.$$
Since $f_0^{(b)}\in C^{\beta-b}$, this is bounded in absolute value by 
$\int |\sigma s|^{\beta-b}|\psi^{(l-b)}(s)|\,ds\,\sigma^{-l+b}\lesssim \sigma^{\beta-l}$.

\subsection{Proof of Corollary~\ref{cor:adaptive:IBM}}\label{sec:proof:adaptive:IBM}
In view of Theorem~\ref{thm: adaptation} it is sufficient to construct a sieve $B_{n,m}^{(k)}$ such that assumptions \eqref{eq: rem-mass}, \eqref{eq: sb-prior} and \eqref{eq: entropy} hold. This may be done analogously to the construction in \cite{vandervaart2009} for the squared exponential Gaussian process.

For $\mathbb{H}_1^{\tau,(k)}$ the unit ball in the RKHS of the local rescaled integrated Brownian motion $(G^{\tau,(k)}_t: t\in I^{(k)})$
with fixed scale $\tau$,  and  $\mathbb{B}^{(k)}_1$ the unit ball in the Banach space of continuous functions on $I^{(k)}$ 
equipped with the supremum norm $\|\cdot\|_{\infty,k}$, we consider sieves of the form
\begin{align}
B_{n,m}^{(k)}=\bigl\{  f1_{I^{(k)}} : f\in \bigcup_{q_n<\tau< r_n} K_{n,m}\mathbb{H}_1^{\tau,(k)}+\eps_n\mathbb{B}^{(k)}_1\bigr\}.
\label{EqDefinitionBnm}
\end{align}
In the present proof we set  $\eps_n=Cn^{-\beta/(2\beta+1)}$, $q_n=0$, $r_n=c n^{\frac{\ell+1/2-\beta}{(\ell+1/2)(2\beta+1)}}$, 
and $K^2_{n,m}:=K (n/m)\eps_n^2$, for large constants $C$, $c$ and $K$ to be determined.

Write $G^{A,(k)}$ for the rescaled integrated Brownian motion process with random scale $\tau$.
It will be shown below that 
the following bounds are satisfied
\begin{align}
\label{eq: log-rem-mass:IBM}
\log\Pr\left(G^{A,(k)}\notin B^{(k)}_{n,m}\right)&\leq -4\frac nm\eps_n^2,\\
\label{eq: log-sb-prior:IBM}
-\log \Pr\left(\|G^{A,(k)}- f_0\|_{\infty,k}\leq\eps_n\right)&\leq \frac nm\eps_n^2,\\
\label{eq: log-entropy:IBM}
\log N\left(\eps_n,B^{(k)}_{n,m},\|\cdot \|_{\infty,k}\right)&\lesssim\frac nm\eps_n^2.
\end{align}
This verifies \eqref{eq: rem-mass}--\eqref{eq: entropy} and completes the proof of the corollary.

\paragraph{Proof of \eqref{eq: log-rem-mass:IBM}}
 The local remaining masses satisfy
\begin{align*}
\Pr&(G^{A,(k)}\notin B_{n,m}^{(k)})\\
&\leq \Pr(\tau>r_n)+\int_0^{r_n}\Pr(G^{\tau,(k)}\notin B_{n,m}^{(k)})g(\tau)\,d\tau.
\end{align*}
By assumption \eqref{def:hyperprior}, the first term is bounded above by a multiple of 
$$\exp\bigl(-D_1 n^{1/(2\ell+2)}r_n^{(\ell+1/2)/(\ell+1)}/(2m)\bigr)\le \exp(-5(n/m)\eps_n^2),$$ if the constant 
$c$ in $r_n$ is chosen small enough relative to $C$ in $\eps_n$ and $D_1$. It suffices
to bound the second term. 

Let $\phi_0^{\tau,(k)}(\eps) =-\log\Pr\bigl(\| G^{\tau,(k)}\|_{\infty,k}<\eps\bigr)$ be the centered small ball exponent 
of the rescaled process $G^{\tau,(k)}$. Then
\begin{align}
\Pr(G^{\tau,(k)}\notin B_{n,m}^{(k)})&\leq\Pr(G^{\tau,(k)}\notin K_{n,m}\mathbb{H}_1^{\tau,(k)}+\eps_n\mathbb{B}^{(k)}_1)\nonumber\\
&\leq 1 -\Phi(\Phi^{-1}(e^{-\phi_0^{\tau,(k)}(\eps_n)})+K_{n,m}),\label{eq:Borel:IBM}
\end{align}
in view of Borell's inequality (see \cite{borell1975brunn} or Theorem~5 in \cite{vaart:zanten:2008:reproducing}).
For $\tau\le r_n$ and $\tau/m\ge \eps_n^{1/(\ell+1/2)}$,  relation \eqref{eq: resc-log-prior} gives that
$\phi_0^{\tau,(k)}(\eps_n)\lesssim r_n\eps_n^{-1/(\ell+1/2)}/m+(0\vee \log B_n)+\log n$, which is bounded above
by $K n\eps_n^2/(8m)=K_{n,m}^2/8$, for sufficiently large $K$.
For $\tau\le r_n$ and $\tau/m\le \eps_n^{1/(\ell+1/2)}$,  relation \eqref{eq: resc-log-prior} applied with 
$\eps=(\tau/m)^{\ell+1/2}\le \eps_n$, gives that
$\phi_0^{\tau,(k)}(\eps_n)\le \phi_0^{\tau,(k)}(\eps)\lesssim 1+(0\vee \log B_n)+\log n\ll K_{n,m}^2/8$.
Therefore, Lemma~\ref{lem: norm-cdf} gives that $\Phi^{-1}(e^{-\phi_0^{\tau,(k)}(\eps_n)})\ge -K_{n,m}/2$, and consequently
the preceding display is bounded above by 
$1-\Phi(K_{n,m}/2)\leq e^{-K_{n,m}^2/8}$,  by  Lemma~\ref{lem: norm-cdf-tail}. This implies \eqref{eq: log-rem-mass:IBM}, 
for $K\ge 32$.

\paragraph{Proof of \eqref{eq: log-sb-prior:IBM}}
For $\tau\in [r_n,2r_n]$, we have that $\tau/\eps_n^{1/(\ell+1/2)}\asymp n\e_n^2$ and
hence $\tau/m\ge \eps_n^{1/(\ell+1/2)}$, by the assumptions on $m$, and 
condition \eqref{eq:cond:B} is satisfied at $\eps=\eps_n/2$, by the assumptions on $B_n$. 
Therefore, assertions  \eqref{eq: resc-inf-norm} and  \eqref{eq: resc-log-prior} give
\begin{align*}
\phi^{\tau,(k)}_{ f_0}(\eps_n/2)&\lesssim \frac{\tau^{-2\ell-1}}{m\eps_n^{\frac{2\ell+2-2\beta}{\beta}}}+\frac{\tau}{m\eps_n^{\frac{1}{\ell+1/2}}}
+\log \frac{B_n(\tau^\ell\vee 1)}{\eps_n^2}.
\end{align*}
The first two terms on the right are both of the order $n\eps_n^2/m$, where the multiplicative constant can be adjusted to be
arbitrarily small by choosing the constant $C$ in $\eps_n$ large, for a given constant $c$ in $r_n$. 
The third term is bounded above by a multiple of $(0\vee \log B_n)+ \log n\le n^{1/(2\beta+1)}/m+\log n$, by assumption, and hence is also bounded by 
an arbitrarily small constant times $n\eps_n^2/m$ if the constant $C$ in $\eps_n$ is sufficiently large.  
Thus $C$ can be adjusted so that  right side  and hence  $\phi^{\tau,(k)}_{ f_0}(\eps_n/2)$
is bounded above by  $n\eps_n^2/(2m)$. Then, in view of  Lemma~5.3 in \cite{vaart:zanten:2008:reproducing},
\begin{align*}
\Pr\big(&\|G^{A,(k)}- f_0\|_{\infty,k}\leq\eps_n\big) \geq\int_{0}^{\infty}e^{-\phi^{\tau,(k)}_{ f_0}(\eps_n/2)}g(\tau)\,d\tau\\
&\geq\int_{r_n}^{2r_n}e^{-n\eps_n^2/(2m)}g(\tau)\,d\tau.
\end{align*}
In view of \eqref{def:hyperprior}, minus the logarithm of this is bounded above by 
$$ -\log r_n+\frac{n\eps_n^2}{2m}+D_2\frac{n^{1/(2\ell+2)} (2r_n)^{(\ell+1/2)/(\ell+1)}}{m}.$$
The first term is of order $\log n$ and is much smaller than the second by the assumptions on $m$.
The third term is of order $n\eps_n^2/m$, where the multiplicative constant can be adjusted to be arbitrarily small by choosing $C$ in $\eps_n$ sufficiently large. Thus the whole expression is bounded above by $n\eps_n^2/m$.

\paragraph{Proof of \eqref{eq: log-entropy:IBM}}
By the combination of Lemmas~\ref{lem: def-rescRKHS} and~\ref{lem: def-locRKHS}, the RKHS $\mathbb{H}^{\tau,(k)}$ is
the set of functions $h: I^{(k)}\to\RR$ that can be extended to a function in the Sobolev space $H^{\ell+1}([0,1])$, equipped with the
norm with square 
$$\|h\|_{\mathbb{H}^{\tau,(k)}}^2=\inf_{\substack{g\in H^{\ell+1}([0,1])\\g1_{I^{(k)}}=h}}\sum_{j=0}^\ell\frac{g^{(j)}(0)^2}{B^2\tau^{2j}}+\frac1{\tau^{2\ell+1}}\int_0^1g^{(\ell+1)}(s)^2\,ds.$$
Because the norm $\tau\mapsto \|h\|_{\mathbb{H}^{\tau,(k)}}$ is decreasing in $\tau$, the unit balls satisfy $\mathbb{H}^{\tau,(k)}_1\subset \mathbb{H}^{r_n,(k)}_1$ 
for $\tau\le r_n$. It follows that
\begin{align*}
&N\Bigl(2\eps_n,B_{n,m}^{(k)},\|\cdot\|_{\infty,k}\Bigr)
\leq N\left(\eps_n/K_{n,m}, \mathbb{H}^{r_n,(k)}_1,\|\cdot\|_{\infty,k}\right).
\end{align*}
Next we bound the covering number on the right using the estimate \eqref{eq: resc-log-prior} on the small
ball probability and the duality between these quantities, proved by \cite{KuelbsLi} (see Lemma~I.29(ii) in \cite{ghosal2017fundamentals}).
By \eqref{eq: resc-log-prior} the small ball exponent satisfies $\phi_0^{r_n,(k)}(\eps_n)\lesssim (r_n/m)\eps_n^{-1/(\ell+1/2)}$.
Because $\Pr\bigl(\|G^{\tau,(k)}\|_{k,\infty}<\eps\bigr)\le \Pr\bigl(\|I^lW\|_{k,\infty}<\eps\bigr)$, by Anderson's lemma,
Lemma~\ref{lem: IBM-int} gives that also $\phi_0^{r_n,(k)}(\eps_n)\gtrsim (r_n/m)\eps_n^{-1/(\ell+1/2)}$.
Thus Lemma~I.29(ii) gives
\begin{align*}
&\log N\Bigl(\frac{\eps_n}{\sqrt{D_2\frac{r_n}m\eps_n^{-1/(\ell+1/2)}}},\mathbb{H}^{r_n,(k)}_1,\|\cdot\|_{\infty,k}\Bigr)
\lesssim \!
\frac{r_n}m\eps_n^{-1/(\ell+1/2)}.
\end{align*}
Because $K_{n,m}^2\asymp (r_n/m)\eps_n^{-1/(\ell+1/2)}=n\eps_n^2/m$, this gives the desired bound
on the right side of the second last display and hence concludes the proof of 
\eqref{eq: log-entropy:IBM}.

\subsection{Proof of Corollary~\ref{cor:Matern:nonadapt}}\label{sec:Matern:nonadapt}
Let $G^{\tau,(k)}$ be the Mat\'ern process with regularity parameter $\alpha$ scaled by $\tau$, write $\mathbb{H}^{\tau,(k)}$ for its RKHS 
and let $\phi_{ f_0}^{\tau,(k)}(\eps)=-\log \Pr(\|G^{\tau,(k)}-f_0\|_{\infty,k}<\eps)$ be its small ball probability when restricted to $I^{(k)}$.
We show below that for $f_0\in C^{\beta}([0,1])$ with $1/2< \beta\leq \alpha$, and 
$(\tau/m)\eps^{-1/\alpha}\gtrsim \log (m/(\tau\eps^2))$ and $\tau\lesssim \eps^{-1/\beta}$, the following bounds are satisfied with multiplicative constants not depending on $\alpha$, $\eps$, $m$, $\tau$ and $n$:
\begin{align}
\label{eq: inf-norm:mat}
\inf_{h:\|f_0-h\|_{\infty,k}\leq\eps}\|h\|^2_{\mathbb{H}^{\tau,(k)}}&\lesssim\frac{\eps^{-\frac{2\alpha+1-2\beta}{\beta}}}{m\tau^{2\alpha}}+\frac{\tau}m,\\ 
\label{eq: log-prior:mat}
-\log\Pr\bigl(\|G^{\tau,(k)}\|_{\infty,k}<\eps\bigr)&\lesssim\frac{\tau}{m}\eps^{-1/\alpha}.
\end{align}
This implies that the concentration function \eqref{def:conc:function} satisfies
\begin{align}
\phi_{ f_0}^{\tau,(k)}(\eps)\lesssim \frac{\eps^{-\frac{2\alpha+1-2\beta}{\beta}}}{m\tau^{2\alpha}}+ \frac{\tau}{m}\eps^{-1/\alpha}.\label{eq: UB:conc:Matern}
\end{align}
For $\tau_n=n^{\frac{\alpha-\beta}{\alpha(1+2\beta)}}$ and $\eps_n\asymp n^{-\beta/(1+2\beta)}$, the right hand side 
of this inequality is a multiple of $n\eps_n^2/m$ and the conditions above \eqref{eq: inf-norm:mat} hold for $m=o(n^{1/(1+2\beta)}/\log n)$. Since Assumption~\ref{ass:metric} holds both in the nonparametric and logistic regression models, Corollary~\ref{cor:Matern:nonadapt} is a consequence of Theorem~\ref{thm: nonadaptive}.

\paragraph{Proof of \eqref{eq: inf-norm:mat}}
Define a map $T: C(I^{(k)})\to C([0,1])$ by $Th(t)=h((k-1)/m+t/m)$, for $0\le t\le1$.
Then the RKHS of the process $\bigl((TG^{\tau,(k)})_t: 0\le t\le 1)$ is $\mathbb{H}:=T\mathbb{H}^{\tau,(k)}$, with norm
$\|Th\|_{\mathbb{H}}=\|h\|_{\mathbb{H}^{\tau,(k)}}$ (see Lemma~I.16 in \cite{ghosal2017fundamentals}).
By the stationarity of the Mat\'ern process, the process $\bigl((TG^{\tau,(k)})_t: 0\le t\le 1)$ is equal in
distribution to $(G^\tau_{t/m}: 0\le t\le 1)$, for $G^\tau$ a Mat\'ern process with spectral density \eqref{def:Matern:spectral}, 
which in turn is equal in distribution to $(G^{\tau/m}_t: 0\le t\le 1)$, since $\tau$ is a scale parameter to the process. The
RKHS of the process $G^{\tau/m}$ is described in Lemma~\ref{LemmaMaternRKHS} (for $m=1$). Let $\|\cdot\|_\infty$ be
the uniform norm on $C([0,1])$, so that $\|Th\|_\infty=\|h\|_{\infty,k}$, and  conclude 
that the left side of \eqref{eq: inf-norm:mat} is equal to
\begin{align*}
&\inf_{\substack {h\in \mathbb{H}^{\tau,(k)}:\\ \|Tf_0-Th\|_\infty<\eps}}\|Th\|^2_{\mathbb{H}}
=\inf_{\substack {h\in \mathbb{H}:\\ \|Tf_0-h\|_{\infty}<\eps}}\|h\|^2_{\mathbb{H}}\\
&\qquad\qquad=\inf_{g: \|g(\tau \cdot/m)-\bar f_0(\cdot/m)\|_\infty<\eps} \|g\|_{H^{\alpha+1/2}([0,\tau/m])}^2,
\end{align*}
where $\bar f_0(t/m)=f_0((k-1)/m+t/m)$. For a smooth $(\alpha+1/2)$th order kernel $\psi$ as
in the proof of \eqref{eq: resc-inf-norm} and $\sigma=\eps^{1/\beta}$, define $g(t)=\bar f_0\ast\psi_\sigma(t/\tau)$.
Then by \eqref{EqConvDistance} $\|g(\tau \cdot/m)-\bar f_0(\cdot/m)\|_\infty<\eps$ and 
by \eqref{eq: UB:supnorm:approx}
\begin{align*}
\|g\|_{H^{\alpha+1/2}([0,\tau/m])}^2
&=\sum_{l=0}^{\alpha+1/2}\int_0^{\tau/m} ( \bar f_0\ast\psi_{\sigma})^{(l)}\Bigl(\frac t\tau\Bigr)^2\frac1{\tau^{2l}}\,dt\\
&\le \sum_{l=0}^b\frac\tau m\frac1{\tau^{2l}}+
\sum_{l=b+1}^{\alpha+1/2}\frac\tau m\frac1{\tau^{2l}} \Bigl(\frac1\sigma\Bigr)^{2l-2\beta},
\end{align*}
where $b$ is the biggest integer smaller than $\beta$.
For $\tau\ge 1$, the first sum on the right is bounded above by a multiple of its first term, which is $\tau/m$. For $\tau\leq 1$, the first sum is bounded above by a multiple of its last term $(\tau/m)\tau^{-2b}$, which is then dominated by the first term of the second sum, since $2(b+1)-2\b\ge 0$ and $\sigma<1$, 
so that the second sum dominates the first. 
Under the assumption $\tau\lesssim 1/\sigma$, the second sum is bounded above by a multiple of its last term, which is
$m^{-1}\tau^{-2\alpha}(1/\sigma)^{2\alpha+1-2\beta}$.
This concludes the proof.

\paragraph{Proof of \eqref{eq: log-prior:mat}}
Because the Mat\'ern process is stationary, its supremum over $I^{(k)}$ is distributed as its supremum over
$[0,1/m]$. As $\tau$ is a scale parameter to the process, it follows that
\begin{align*}
\Pr\Bigl( \sup_{t\in I^{(k)}} |G^{\tau,(k)}_t|<\eps\Bigr)&= \Pr\Bigl( \sup_{0<t\leq 1} |G_t^{\tau/m}|<\eps\Bigr),
\end{align*}
for $t\mapsto G_t^\tau$ the Mat\'ern process with spectral density \eqref{def:Matern:spectral}.  Then in view of the assumptions $(\tau/m)\eps^{-1/\alpha}\gtrsim \log (m/(\tau\eps^2))\ge \log(m/\eps^{2-1/\b})$, the statement follows from Lemma \ref{lem:small:ball:rescaled:Matern} with scaling parameter $\tau/m$.

\subsection{Proof of Corollary~\ref{cor:matern:adapt}}\label{sec:proof:adapt:matern}
We proceed similarly to the proof of Corollary~\ref{cor:adaptive:IBM} in Section~\ref{sec:proof:adaptive:IBM}. 
We again consider the sieves $B_{n,m}^{(k)}$ defined in \eqref{EqDefinitionBnm}, 
where presently $\mathbb{H}_1^{\tau,(k)}$ is the unit ball in the RKHS of the local rescaled Mat\'ern process $t\mapsto G^{\tau,(k)}_t$ on $I^{(k)}$, with fixed scale $\tau$.
Presently we set $\eps_n=Cn^{-\beta/(2\beta+1)}$,  $q_n= e^{-5(n/m)\eps_n^2}$, $r_n=cn^{(\alpha-\beta)/((2\beta+1)\alpha)}$, $K^2_{n,m}=K(n/m)\eps_n^2$, with suitable constants $C$, $c$ and $K$.

It suffices to show that the conditions of Theorem~\ref{thm: adaptation} are satisfied. The verifications of \eqref{eq: rem-mass} and \eqref{eq: sb-prior} follow the lines of the proof of Corollary~\ref{cor:adaptive:IBM}, where we employ the correspondences between the pair \eqref{eq: resc-inf-norm} and \eqref{eq: inf-norm:mat} and the pair \eqref{eq: resc-log-prior} and \eqref{eq: log-prior:mat}, with the substitution $\ell+1/2=\alpha$, as well as the correspondence between conditions \eqref{def:hyperprior} and \eqref{def:hyperprior:Mat} on the hyper-prior density. (See the paragraphs ``Proof of \eqref{eq: log-rem-mass:IBM}'' and ``Proof of \eqref{eq: log-sb-prior:IBM}''.)  A minor difference is that in the bound of the remaining mass, we also
separate out the probability.
\begin{align*}
\Pr(\tau<q_n)=\int_0^{q_n} g_{\alpha,n,m}(\tau)\, d\tau\leq q_n c_2\lesssim e^{-5(n/m)\eps_n^2}.
\end{align*}

\paragraph{Verification of \eqref{eq: entropy}}
The $\eps_n$-entropy of $B_{n,m}^{(k)}$ is bounded above by the $2\eps_n$ entropy of $\cup_{q_n<\tau<r_n} K_{n,m}\mathbb{H}^{\tau,(k)}_1$.
 As argued in the proof of \eqref{eq: inf-norm:mat}, the RKHS $\mathbb{H}^{\tau,(k)}$ of $(G^\tau_t: t\in I^{(k)})$ is isometric
to the RKHS of the Mat\'ern process with scale $\tau/m$ on the interval $[0,1]$. The RKHS  $\mathbb{H}^{\tau/m}$ of this Mat\'ern process is described in Lemma~\ref{LemmaMaternRKHS}.
Because the local uniform norm on $I^{(k)}$ maps to the uniform norm on $[0,1]$ under this correspondence, we can estimate
the entropy of $K_{n,m}\mathbb{H}_1^{\tau,(k)}$ under $\|\cdot\|_{\infty,k}$ by the entropy of $K_{n,m}\HH_1^{\tau/m}$ for the uniform norm on $[0,1]$.
In Lemma~\ref{LemmaMaternRKHS}, the entropy of the union of these spaces over $\tau\in (q_n,r_n)$ is bounded above by a multiple of
\begin{align*}
\Bigl(\frac{r_n}m\Bigr)^{\alpha/(\alpha+1/2)}\Bigl(\frac{\eps_n}{K_{n,m}}\Bigr)^{-1/(\alpha+1/2)}+\log\frac{K_{n,m}}{\sqrt {q_n/m}\, \eps_n}.
\end{align*} 
For the given choices of $r_n$, $q_n$, $K_{n,m}$ and $\eps_n$, this is bounded above by 
a multiple of $(n/m)\eps_n^2$.

\subsection{Technical lemmas}\label{sec:tech:lemma}

\begin{lemma}\label{lem: def-rescRKHS}
The RKHS $\mathbb{H}$ of the process $(G^{\tau,(k)}_t: t\in [0,1])$ given in \eqref{eq: intBMresc} is the Sobolev space $H^{\ell+1}([0,1])$ equipped with inner product
\begin{align*}
\langle \vartheta_1,\vartheta_2\rangle_{\mathbb{H}}&=\sum^{\ell}_{j=0}\frac{\vartheta_1^{(j)}(0)\vartheta_2^{(j)}(0)}{B_n^2\tau^{2j}}\\
&\qquad+\frac1{\tau^{2\ell+1}}\int^1_0\vartheta_1^{(\ell+1)}(s)\vartheta_2^{(\ell+1)}(s)\,ds.
\end{align*}
\end{lemma}

\begin{proof}
With scales set equal to $\tau=1$ and $B_n=1$, this is exactly Lemma~11.29 in \cite{ghosal2017fundamentals}. With general scaling by $\tau$, but without the polynomial part, the lemma follows from Lemma~11.52 in the same reference, as integrated Brownian motion is self-similar of order $\ell+1/2$. The polynomial part can next be incorporated by the arguments given in the proof of Lemma~11.29 in \cite{ghosal2017fundamentals}, based on their Lemma~I.18.
\end{proof}

\begin{lemma}\label{lem: def-locRKHS}
Given a centered Gaussian process $(G_t: t\in [0,1])$ with RKHS $\mathbb{H}$, the RKHS $\mathbb{H}^{I}$ of the process $(G_t: t\in I)$ for $I\subset[0,1]$ is equal to the set of functions $h\in\mathbb{H}$ restricted to $I$, with the norm
\begin{align*}
\|h\|_{\mathbb{H}^{I}}=\inf_{h^*\in\mathbb{H};h^*(t)=h(t):t\in I}\|h^*\|_{\mathbb{H}}.
\end{align*}
\end{lemma}

\begin{proof}
By definition, the RKHS  $\mathbb{H}^{I}$ consists of all functions $z=z_H$ that can be represented as
$$z_H(t)=\E HG_t,\quad t\in I,$$
for a random element $H\in\overline{\textrm{lin}}(G_t:t\in I)$, and its RKHS norm is equal to the $L_2$-norm
$\sqrt {\E H^2}$. Since $H$ is also an element of $\overline{\textrm{lin}}(G_t:t\in [0,1])$, the right side of the display with $t$ ranging over the larger domain $[0,1]$ is also an element of $\mathbb{H}$, by the definition of $\HH$. Thus any  $z\in\mathbb{H}^{I}$ is the restriction of a function $z\in\mathbb{H}$. 

Any other $z^*\in\mathbb{H}$ with restriction to $I$ equal to the given $z=z_H\in\mathbb{H^I}$ can be represented as $z^*(t)=\E  H^*G_t$, for $t\in [0,1]$, for some $H^*\in\overline{\textrm{lin}}(G_t:t\in [0,1])$ and possesses $\mathbb{H}$-norm equal to $\sqrt {\E (H^*)^2}$. Because $\E  H^*G_t=\E HG_t$ for all $t\in I$, it follows that $H^*-H$ is orthogonal to $\overline{\textrm{lin}}(G_t:t\in I)$,
whence $H$ is the projection of $H^*$ on this space and has a smaller $L_2$-norm. This proves the assertion on the norm.
\end{proof}

\begin{lemma}
\label{lem: IBM-int}
Let $(W_t: t\ge 0)$ be a standard Brownian motion process.
For any $\ell\in\mathbb{N}$, there exist constants $q_1,q_2>0$
such that for any $0\le x_0<x_1\leq 1$  and $\eps>0$ with $\eps^{-\frac{1}{\ell+1/2}}(x_1-x_0)\ge 1$, 
\begin{align*}
q_1(x_1-x_0)\eps^{-\frac{1}{\ell+1/2}} &\leq-\log\Pr\Big(\sup_{x_0\leq t\leq x_1}\bigl|(I^{\ell}W)_t\bigr|<\eps\Big)\\
&\qquad \le q_2(x_1-x_0)\eps^{-\frac{1}{\ell+1/2}} +q_2\log\frac2\eps.
\end{align*}
\end{lemma}

\begin{proof}
By the independence and stationarity of the increments of Brownian motion, the process
$(W^*_t)_{t\geq 0}$ defined by $W^*_t:=W_{x_0+t}-W_{x_0}$  is also a standard Brownian motion,
independent from $(W_t)_{0\leq t\leq x_0}$. 
First, we prove by induction that, for every $\ell\in\mathbb{N}$ and $t\geq 0$,
\begin{equation} (I^{\ell}W^*)_t=(I^{\ell}W)_{x_0+t}-\sum_{j=0}^{\ell}\frac{t^j}{j!}(I^{\ell-j}W)_{x_0}.
\label{EqWWstar}
\end{equation}
The identity is true by definition for $\ell=0$. If the claim is true for $\ell$, then
\begin{align*}
I^{\ell+1}W^*_t&
=\int_0^t (I^{\ell}W)_{x_0+s}ds-\sum_{j=0}^{\ell}\int_0^t\frac{s^j}{j!}(I^{\ell-j}W)_{x_0}ds\\
&=(I^{\ell+1}W)_{x_0+t}-(I^{\ell+1}W)_{x_0}-\sum_{j=0}^{\ell}\frac{t^{j+1}}{(j+1)!}(I^{\ell-j}W)_{x_0}\\
&=(I^{\ell+1}W)_{x_0+t}-\sum_{j=0}^{\ell+1}\frac{t^j}{j!}(I^{\ell+1-j}W)_{x_0},
\end{align*}
which confirms the claim for $\ell+1$.

In view of \eqref{EqWWstar} it follows that
\begin{align*}
\Pr&\Big(\sup_{x_0\leq t\leq x_1}|(I^{\ell}W)_t|<\eps\Big)\\
&=\Pr\Bigl(\sup_{0\leq t\leq x_1-x_0}\Big|(I^{\ell}W^*)_t+\sum_{j=0}^{\ell}\frac{t^j}{j!}(I^{\ell-j}W)_{x_0}\Bigr|<\eps\Bigr).
\end{align*}
The two variables in the sum in the right side are independent. The first is an ordinary Brownian motion,
while the second is a polynomial of fixed degree with (dependent) Gaussian coefficients. By the self-similarity of
integrated Brownian motion the process
$t\mapsto (I^{\ell}W^*)_{t(x_1-x_0)}$ is distributed as $t\mapsto (x_1-x_0)^{\ell+1/2}(I^{\ell}W^*)_{t}$. Therefore,
by the small ball probability of the integrated Brownian motion (see Lemma~11.30 in \cite{ghosal2017fundamentals}),
\begin{align*}&\log\Pr\Bigl(\sup_{0\leq t\leq x_1-x_0}\bigl|(I^{\ell}W^*)_t\bigr|<\eps\Bigr)\\
&\quad=\log \Pr\Big(\sup_{0\leq t\leq 1}\bigl|(I^{\ell}W^*)_t\bigr|<\frac\eps{(x_1-x_0)^{\ell+1/2}}\Big)
\asymp \frac{x_1-x_0}{\eps^{1/(\ell+1/2)}},
\end{align*}
provided the last expression remains bounded from below. To complete the proof of the lemma it suffices to consider the polynomial part.

For the lower bound of the lemma, we can just ignore the polynomial part, noting that
by Anderson's lemma (e.g.\ Corollary~A.2.11 in \cite{vdVW2})
the sum of the two independent centered processes gives less probability to the
centered small ball than the Brownian motion part on its own.

For the upper bound, we use that $\Pr(A+B<\eps)\ge \Pr(A<\eps/2)\Pr(B<\eps/2)$
for independent random variables $A$ and $B$, and hence only need to derive a lower bound
on  $\Pr(B<\eps/2)$ for $B=\sup_{0\leq t\leq x_1-x_0}|\sum_{j=0}^{\ell}\frac{t^j}{j!}(I^{\ell-j}W)_{x_0}|$.
Since $B\le e\max_{0\le j\le l+1}|(I^{\ell-j}W)_{x_0}|$, the Gaussian correlation inequality
(see \cite{Royen}) gives that 
\begin{align*}
\Pr&\Bigl(\sup_{0\leq t\leq x_1-x_0}\Bigl|\sum_{j=0}^{\ell}\frac{t^j}{j!}(I^{\ell-j}W)_{x_0}\Bigr|<\eps\Bigr)\\
&\qquad\qquad\qquad\geq\prod_{j=0}^{\ell}\Pr\Bigl(\bigl|(I^{\ell-j}W)_{x_0}\bigr|<\frac\eps e\Bigr).
\end{align*}
For $c$ the maximal standard deviation of the centered Gaussian variables in the right side,
all probabilities are bounded below by $2\Phi\big(\eps/(ec)\bigr)-1$. Since
$\log(2\Phi(x)-1)\geq \log x-1$, for $0<x<2$, we obtain that minus the
logarithm of the preceding display is bounded above by $-(l+1)\log\eps+(l+1)\log (e^2c)\lesssim \log (2/\eps)$ if $\eps<1$. 
\end{proof}

\begin{lemma}
\label{LemmaMaternRKHS}
The RKHS $\mathbb{H}^\tau$ for the one-dimensional Mat\'ern process $(G^\tau_t: 0\le t\le 1)$ with spectral density $\rho_{\alpha,\tau}$ given
by \eqref{def:Matern:spectral} (with $d=1$) is the set of functions $h: [0,1]\to\mathbb{R}$ that can be written as $h(t)=g(\tau t)$, for
$g$ in the Sobolev space $H^{\alpha+1/2}([0,\tau])$, with norm $\|h\|_{\mathbb{H}^\tau}=\|g\|_{H^{\alpha+1/2}([0,\tau])}$. 
For $\alpha+1/2\in\NN$, the entropy of the unit ball in this space satisfies, for $\eps<1/2$ and $0<\tau_0\le \tau_1<\infty$,
and a multiplicative constant that depends on $\alpha$ only,
\begin{align*}&\log N\bigl(\eps, \cup_{\tau\in[\tau_0,\tau_1]}\mathbb{H}_1^\tau,\|\cdot\|_{\infty}\bigr)\\
&\ \lesssim 
\begin{cases}\tau_1^{\alpha/(\alpha+1/2)}\eps^{-1/(\alpha+1/2)}, &\text{if }\tau_0\ge 1,\\
\tau_1^{\alpha/(\alpha+1/2)}\eps^{-1/(\alpha+1/2)}+\log(1/(\sqrt {\tau_0} \eps)), &\text{if }\tau_0\le 1.
\end{cases}\end{align*}
\end{lemma}

 \begin{proof}
For simplicity take the constants $C_{\alpha,1}$ and $c_{\alpha,1}$ in the spectral density equal to 1.
In view of Lemma~11.35 of \cite{ghosal2017fundamentals}, the RKHS of $G^\tau$ is the set
of functions $h: [0,1]\to\mathbb{R}$ of the form $h(t)=\int e^{i\lambda t}\psi(\lambda)\rho_{\alpha,\tau}(\lambda)\,d\lambda$, for
functions $\psi$ such that $\int|\psi(\lambda)|^2\rho_{\alpha,\tau}(\lambda)\,d\lambda<\infty$, with
square norm $\|h\|_{\mathbb{H}}^2$ equal to the infimum of the latter expression over all $\psi$ that give the same
function $h$ (on $[0,1]$). Substituting $\phi(\lambda)=\psi(\tau\lambda)(1+\lambda^2)^{-\alpha-1/2}$,
we can write these functions as $h(t)=g(\tau t)$, for $g(t)=\int e^{it\lambda}\,\phi(\lambda)\,d\lambda$, with square norm
$\|h\|_{\mathbb{H}}^2$ the infimum of $\int |\phi(\lambda)|^2(1+\lambda^2)^{\alpha+1/2}\,d\lambda$ 
over all $\phi$ that give the same function $g$. When defined with domain the full line $\mathbb{R}$, these functions $g$ 
form the Sobolev space $H^{\alpha+1/2}(\mathbb{R})$ (see e.g.\ \cite{Evans10}, page~282), and the restrictions
to an interval are the Sobolev space corresponding to this interval, with equivalent norm the infimum of the norms of all extensions.
This proves the first assertion.

For an integer-valued index $k:=\alpha+1/2$, an equivalent norm of $H^{k}([0,\tau])$ is the root of 
$\|g \|_{H^{k}([0,\tau])}^2=\sum_{l =0}^{k} \int_0^\tau g^{(l)}(s)^2 \,ds$, where
the highest derivative may be in the distributional sense. When $\tau=1$ we also have that the $L_2([0,1])$-norms 
of lower-order derivatives of functions with $g(0)=\cdots=g^{(k-1)}(0)=0$
are bounded by the norm of the $k$th derivative and hence $\|g \|_{H^{k}([0,1])}^2\le (k+1)\int_0^1 g^{(k)}(s)^2 \,ds$. 
This follows
by repeatedly applying the inequality $|g^{(l)}(t)|=|\int_0^tg^{(l+1)}(s)\,ds|\le \|g^{(l+1)}\|_{L_2([0,1])}$, for $0\le t\le 1$ and
$l=k-1, \ldots,0$.

For the entropy bound, first consider the case of a single unit ball $\mathbb{H}_1^\tau$ with $\tau\ge 1$.
The map $T: H^{k}([0,\tau])\to H^{k}([0,1])$ defined by $Tg(t)=g(\tau t)$, for $t\in[0,1]$, has 
$\|Tg \|_{H^{k}([0,1])}^2=\sum_{l =0}^k\tau^{2l-1} \int_0^\tau g^{(l)}(s)^2 \,ds$ and hence
$\|Tg \|_{H^k([0,1])}\le \tau^{k-1/2}\|g \|_{H^k([0,\tau])}$, for $\tau\ge 1$. Since $\|g\|_{H^k([0,\tau])}=\|Tg\|_{\mathbb{H}^\tau}$,
it follows that $\mathbb{H}^\tau_1\subset \tau^{k-1/2}H_1^k([0,1])$ and hence its uniform covering number
is bounded above by a multiple of $(\eps/\tau^{k-1/2})^{-1/k}$ (see e.g.\ Proposition~C.7 of \cite{ghosal2017fundamentals}, with $r=\infty$).

Second, consider the case of a single RKHS with $\tau\le 1$. We decompose $g\in H^k([0,\tau])$ as $g=Pg+Qg$, for $Pg(t)=\sum_{l=0}^{k-1}g^{(l)}(0)t^{l}/l!$
the Taylor polynomial at zero. The set of functions $Qg=g-Pg$ has zero derivatives at 0 up to order $k-1$, and
the same is true for their scaled versions $TQg\in \mathbb{H}^\tau$. It follows that $\|TQg\|_{H^k([0,1])}\lesssim \|(TQg)^{(k)}\|_{L_2([0,1])}
=\tau^{k-1/2}\|(Qg)^{(k)}\|_{L_2([0,\tau])}=\tau^{k-1/2}\|g^{(k)}\|_{L_2([0,\tau])}$, because $(Pg)^{(k)}=0$. Consequently $\{TQg: g\in H_1^k([0,\tau])\}\subset C\tau^{k-1/2}H_1^k([0,1])$, for some constant $C$, and hence the entropy of this set is bounded as in the preceding paragraph.

It remains to cover the set of functions $\{TPg: g\in H_1^k([0,\tau])\}$. 
Because for $t\in[0,\tau]$, $|g^{(l)}(t)-g^{(l)}(0)|\le \int_0^\tau |g^{(l+1)}(s)|\,ds\le \|g^{(l+1)}\|_{L_2([0,\tau])}\sqrt \tau\le \sqrt\tau$, if $g\in H_1^k([0,\tau])$ and $l<k$, the triangle inequality gives that  $\sqrt \tau |g^{(l)}(0)|\le \|g^{(l)}(0)-g^{(l)}\|_{L_2([0,\tau])}+\|g^{(l)}\|_{L_2([0,\tau])}\lesssim \tau+1$, if $g\in H_1^k([0,\tau])$. Thus the $k$ coefficients of the polynomials $Pg$, for $g\in H_1^k([0,\tau])$, are bounded in absolute value by $C/\sqrt \tau$,  for some constant $C$. Discretising these coefficients on a grid of mesh width $\eps$ yields uniform approximations to the polynomials $TPg$ within a multiple of $\eps$ (since $\tau\le 1$). Therefore, the covering numbers of the set of polynomials are bounded above by $((C/\sqrt \tau)/\eps)^k$.

Finally consider the case of a union over $\tau\in[\tau_0,\tau_1]$. If $\tau_0\ge 1$, then the first argument shows that
$\mathbb{H}^\tau_1 \subset \tau^{k-1/2}H_1^k([0,1])\subset \tau_1^{k-1/2} H_1^k([0,1])$ and the bound follows. If $\tau_0\le 1$, then the arguments in the last two paragraphs show that $\mathbb{H}^\tau_1\subset \tau_1^{k-1/2} H_1^k([0,1])+\mathbb{P}^{\tau_0}$, where $\mathbb{P}^{\tau}$ are polynomials
with coefficients uniformly bounded by a multiple of $C/\sqrt{\tau_0}$, and the second bound follows.
\end{proof}

\begin{lemma}
\label{LemmaTaylorProjection}
Given a centered Gaussian process $G=(G_t: t\in [0,1])$ with  $l\ge 0$ times differentiable sample paths, the
RKHS $\bar\HH$ of the process $\bar G=(\bar G_t: t\in[0,1])$ defined by $\bar G_t=G_t-\sum_{j=0}^l G_0^{(j)}t^j/j!$
for $G_0^{(j)}$ the $j$th derivative at 0,
is the set of all functions $t\mapsto \bar h(t)=h(t)-\sum_{j=0}^lh^{(j)}(0)t^j/j!$ for $h$ ranging over the RKHS $\HH$ of $G$ with norm
$$\|\bar h\|_{\bar\HH}=\inf\biggl\{\|h\|_\HH: h\in\HH, \bar h(t)=h(t)-\sum_{j=0}^lh^{(j)}(0)\frac{t^j}{j!}\biggr\}.$$
\end{lemma}

\begin{proof}
By definition the RKHS $\bar\HH$ is the set of functions $t\mapsto\bar h(t)=\E \bar G_t\bar H$ when $\bar H$ ranges
over the closure of the linear span of the variables $\bar G_t$ with square norm $\|\bar h\|_{\bar \HH}^2= \E \bar H^2$.
This space can equivalently be described as the set of all functions $t\mapsto \E \bar G_t H$ with $H$ ranging over 
all square-integrable variables, with square norm equal to the infimum of $\E H^2$ over all $H$ that represent
the given function $\bar h$ as $\bar h(t)= \E \bar G_t H$. 
Similarly,  the RKHS $\HH$ of $G$ consists of the functions $t\mapsto h(t)= \E G_t H$
with square norm $\|h\|_{\HH}^2$ the infimum of $\E H^2$ over all variables that represent $h$ in this way.
By the definition of $\bar G_t$, we have $\E \bar G_t H=\E G_tH -\sum_{j=0}^l \E G_0^{(j)}H\, t^j/j!$. 
If $h(t)= \E G_tH$, then $h^{(j)}(0)= \E G_0^{(j)}H$, whence this identity can be written
as $\bar h(t)=h(t)-\sum_{j=0}^l h^{(j)}(0)t^j/j!$. 

Let $SH$ be the function given by $(SH)(t)=\E G_t H$
and let $Qh$ be the function given by $Qh(t)=h(t)-\sum_{j=0}^l h^{(j)}(0)t^j/j!$. Then we conclude that
\begin{align*}
\|\bar h\|_{\bar\HH}^2&=\inf_{H: QSH=\bar h} \E H^2=
\inf_{h: Qh=\bar h}\inf_{H: SH=h} \E H^2=\inf_{h: Qh=\bar h} \|h\|_\HH^2.
\end{align*}
The lemma follows.
\end{proof}

\begin{lemma}\label{lem:small:ball:rescaled:Matern} 
Consider the time-rescaled Mat\'ern  process $G_{t}^{\tau}$, with spectral measure given in \eqref{def:Matern:spectral}
and $\alpha+1/2\in\NN$, restricted to the interval $[0,1]$. Then, for $\eps<1$, and a multiplicative constant
that depends on $\alpha$ only,
\begin{align}
 -\log \Pr\Bigl(\sup_{t\in[0,1]}|G_t^{\tau}|\leq \eps\Bigr)\lesssim \tau \eps^{-1/\alpha}+\log \frac 2\e.\label{eq:small:ball:Matern}
\end{align}
\end{lemma}
\footnote{{During the review process we have learned about the recent, independently written paper \cite{fang2024} deriving posterior contraction rates for the rescaled Gaussian process with Mat\'ern covariance kernel. The authors consider scaling parameters $\tau\geq 1$ and derive lower bound for the prior small ball probability and upper bound for the decentering term. Our results, however, go beyond this setting and cover the $\tau< 1$ case as well. This is necessary in the distributed setting where computing the properties of the local posterior involves the rescaled Gaussian process $G^{\tau/m}$, with scaling parameter possibly of size $o(1)$.}}

\begin{proof}
We distinguish two cases depending on the value of $\tau$. 

For $\tau\geq 1$, we note that in view of \eqref{def:Matern:spectral},
\begin{align}
\rho_{\alpha,\tau} (\lambda)\leq C_{\alpha,d} \tau^{2\alpha}(c_{\alpha,d}+\|\lambda\|^2)^{-\alpha-d/2}=:\rho^*_{\alpha,\tau}(\lambda).\label{eq:spectral:matern:ineq}
\end{align}
The function $\rho^*_{\alpha,\tau}$ on the right hand side is $\tau^{2\alpha}$ times the spectral density of the Mat\'ern process 
with scale parameter $\tau=1$, which is the spectral density of the space-rescaled Mat\'ern process 
$t\mapsto\tau^{\alpha}G^1_t$. By Lemma 11.36 of \cite{ghosal2017fundamentals},
\begin{align}
-\log& \Pr\Bigl(\sup_{t\in[0,1]}|\tau^{\alpha}G^1_t|\leq \eps\Bigr)
\lesssim (\eps/\tau^{\alpha})^{-1/\alpha}=\tau \eps^{-1/\alpha}.\label{eq:help:small:ball}
\end{align}
Because $\rho^*_{\alpha,\tau}-\rho_{\alpha,\tau}\ge 0$, by Bochner's theorem (e.g.\ \cite{Gnedenko}, p. 275--277)
this difference function is the spectral density of a stationary Gaussian process
$Z$, and then
 $\tau^{\alpha}G^1\stackrel{d}{=}G^{\tau}+Z$, where $G^\tau$ and $Z$ are independent.
Then Anderson's lemma gives that
 \begin{align*}
 \Pr\Bigl(\sup_{t\in[0,1]}|\tau^{\alpha}G_t^1|\leq \eps\Bigr)
 &\leq \Pr\Bigl(\sup_{t\in[0,1]}|G_{t}^{\tau}|\leq \eps\Bigr),
 \end{align*}
  which together with \eqref{eq:help:small:ball} implies statement \eqref{eq:small:ball:Matern} (without the logarithmic term).

To deal with the case $\tau\leq 1$, we decompose the process as $G_t^\tau=(PG^\tau)_t+(QG^\tau)_t$,
for $(PG)^\tau_t=\sum_{l=0}^{k-1}G_0^{\tau,(j)} t^j/j!$, where $k=\alpha+1/2$ and $G_0^{\tau,(j)}$ is the $j$th derivative of
$G_t^{\tau}$ at $t=0$. By the triangle inequality and the Gaussian correlation inequality (see \cite{Royen}) we have
\begin{align*}
&\Pr\Bigl(\sup_{0<t<1}|G_t^\tau|<\eps\Bigr)\\
&\quad\ge \Pr\Bigl(\sup_{0<t<1}|(PG^\tau)_t|<\frac\eps2\Bigr)
\Pr\Bigl(\sup_{0<t<1}|(QG^\tau)_t|<\frac\eps2\Bigr).
\end{align*}
We shall prove that minus the logarithm of the two probabilities on the right are bounded above
by the second and first terms in the bound of the lemma, respectively.

Since $G^\tau_t=G^1_{\tau t}$, it follows that $|G^{\tau,(j)}_0|=\tau^j  |G^{1,(j)}_0|\le |G^{1,(j)}_0|$, for $\tau\le 1$.
Thus the supremum in the first term on the right is bounded above by $e\max_{0\le j<k} |G_0^{1,(j)}|$, whence the
probability is of the order $\eps^k$, as $\eps\downarrow 0$. This term gives rise to the $\log (2/\e)$ term
in the lemma.

By Lemmas~\ref{LemmaTaylorProjection} and~\ref{LemmaMaternRKHS}, the RKHS of the process $QG^\tau$
is the set of functions $t\mapsto \bar h(t)=g(\tau t) - \sum_{l=0}^{k-1}g^{(l)}(0)(\tau t)^l/l!$, with
norm the infimum of $\|g\|_{H^{k}([0,\tau])}$ over all functions $g\in H^{k}([0,\tau])$
that represent $\bar h$ in this way.  In the notation of the proof of Lemma~\ref{LemmaMaternRKHS} this is the set
of functions $\bar h=TQg$, and the unit ball of the RKHS consists of the functions $\bar h$
such that there exists $g$ with $\bar h=TQg$ and $\|g\|_{H^{k}([0,\tau])}\le 1$. It is seen in the proof of
Lemma~\ref{LemmaMaternRKHS} that  $\{TQg: g\in H_1^k([0,\tau])\}\subset C \tau^{k-1/2}H_1^k([0,1])$. It follows that
the unit ball of the RKHS of the process $QG^\tau$ is contained in the latter set, whence its $\eps$-entropy is bounded above by
$(\e/\tau^{k-1/2})^{-1/k}$.  This verifies condition  \eqref{eq:cond:entropy} of Lemma~\ref{prop:small:ball:small:tau}
with $\gamma=1/(\alpha+1/2)$ and $J=C_1\tau^{\alpha/(\alpha+1/2)}$, for some constant $C_1$.
We also have the crude bound
 \begin{align*}
& -\log \Pr(\sup_{t\in[0,1]}|(QG)^{\tau}_t|\leq \eps)\\
&\qquad\qquad\qquad\qquad\leq  -\log \Pr(\sup_{t\in[0,1]}|(QG)^1_t|\leq \eps)\lesssim\eps^{-1/\alpha}, 
 \end{align*}
by the preceding for $\tau=1$. This verifies \eqref{eq:cond:crude:UB}, with $f(J)$ a constant and $c=1/\alpha$.
An application of Lemma \ref{prop:small:ball:small:tau} concludes the proof of \eqref{eq:small:ball:Matern}.
\end{proof}

The following lemma is a  modified version of  Proposition~3.1 of \cite{li:linde:99}, in which 
the constants have been made explicit, using
a preliminary crude bound on the small ball probability.

\begin{lemma}\label{prop:small:ball:small:tau}
Consider a centered Gaussian variable $W$ in a separable Banach space, with norm $\|\cdot\|$, and let $\mathbb{H}_1$ be the unit ball of its  RKHS. 
Assume that there exist constants $\eps_0, J>0$, $\gamma\in (0,2)$, $c>0$ and a function $f:(0,\infty)\to(0,\infty)$ such that, for all $\eps\in(0,\eps_0)$,
  \begin{align}
  \log N\bigl(\eps, \mathbb{H}_1,\|\cdot\|\bigr)\leq J \eps^{-\gamma},\label{eq:cond:entropy}\\
-\log \Pr(\|W\|\leq \eps) \leq f(J)\eps^{-c}.\label{eq:cond:crude:UB}
  \end{align}
Then, for all $\eps\in(0,\eps_0)$,
$$-\log \Pr(\|W\|\leq \eps) \leq c_2\max\Big\{J^{2/(2-\gamma)}\eps^{-2\gamma/(2-\gamma)} ,1\Big\},$$
for a constant  $c_2$ that depends only on $\gamma$ and $c$.
\end{lemma}

\begin{proof}
We follow the lines of the proof of Proposition~3.1 in \cite{li:linde:99}, making changes when needed.  Let $\phi_0(\eps)=-\log \Pr(\|W\|\leq \eps)$. 
Since the statement is trivially true if $\phi_0(\eps)$ is bounded above by a (universal) constant, we may assume that $\phi_0(\eps)>\log 2$.
Then $\theta_\eps$ defined by  $\log \Phi (\theta_\eps)=-\phi_0(\eps)$, for $\Phi$ the standard normal distribution function,
is negative and hence $\phi_0(\eps)\ge \theta_\eps^2/2$, by the standard normal tail bound, Lemma~\ref{lem: norm-cdf-tail}.
It is shown in Lemma~1 in \cite{KuelbsLi} that, for any $\lambda, \eps>0$,
\begin{align}
\log N\bigl(\eps, \lambda \mathbb{H}_1,\|\cdot\|\bigr) -\phi_0(2\eps)\geq \log \Phi(\lambda+\theta_\eps).\label{lem:Kuelbs}
\end{align}
For $\lambda=\sqrt{2\phi_0(\eps)}$, we have that $\lambda+\theta_\eps=\sqrt{2\phi_0(\eps)}+\theta_\eps\geq |\theta_\eps|+\theta_\eps\ge 0$.
Together with \eqref{lem:Kuelbs} this gives
$$\log N\bigl(\eps/\sqrt{2\phi_0(\eps)}, \mathbb{H}_1,\|\cdot\|\bigr) \geq \phi_0(2\eps)+\log (1/2).$$
Changing $\eps$ to $\eps/2$ and combining the above display with assumption \eqref{eq:cond:entropy}, we see that
\begin{align*}
\phi_0(\eps)&\leq \log 2+  N\bigl(\eps/\sqrt{8\phi_0(\eps/2)}, \mathbb{H}_1,\|\cdot\|\bigr)\\
& \leq   2 \max\{\log 2,  J\eps^{-\gamma} 8^{\gamma/2}\phi_0(\eps/2)^{\gamma/2}\}.
\end{align*}
If the maximum is taken in the first term, then our statement holds. Suppose that the maximum is taken in the second term.
Since the function $\eps\mapsto \eps^{-\gamma}\phi_0(\eps/2)^{\gamma/2}$ is decreasing, the maximum is then also taken in the
second term for every smaller value of $\eps$, whence for every such $\eps$ and with $\psi(\eps)=16J \eps^{-\gamma}$,
\begin{align*}
\log \phi_0(\eps)\leq (\gamma/2)\log \phi_0(\eps/2)+\log \psi (\eps).
\end{align*}
Iterating this inequality $K$ times, we obtain
\begin{align*}
\log \phi_0(\eps)&\leq  (\gamma/2)^K \log \phi_0(\eps/2^K)+\sum_{j=0}^{K-1} (\gamma/2)^{j}\log\psi(\eps/2^j)\\
&\le (\gamma/2)^K \log\bigl( f(J)2^{c K} \eps^{-c}\bigr)+\sum_{j=0}^{K-1} (\gamma/2)^{j}\log\psi(\eps/2^j),
\end{align*}
in view of \eqref{eq:cond:crude:UB}.
The first term on the right tends to zero as $K\rightarrow\infty$, while the second tends to 
\begin{align*}
&\frac{2}{2-\gamma}\log \psi(\eps)+\sum_{j=0}^{\infty} (\gamma/2)^{j}\log\frac{\psi(\eps/2^j)}{\psi(\eps)}\\
&\qquad\qquad=\frac{2}{2-\gamma}\log \psi(\eps)+\sum_{j=0}^{\infty} (\gamma/2)^{j}\gamma j\log 2.
\end{align*}
Thus the right side, in which the second term is a constant that depends only on $\gamma$,  is a bound
on $\log \phi_0(\eps)$. The lemma follows upon exponentiating.
\end{proof}

In addition to the notation in Assumption~\ref{ass:metric}, let $E(p_{f_0,i},p_{f,i})$ be the square of the $\psi_2$-Orlicz
norm of the (centered) variable $Z_i(f):=\log \bigl(p_{f,i}/p_{f_0,i}(Y_i|x_i)\bigr)+K(p_{f_0,i},p_{f,i})$ under $f_0$, i.e.\ the smallest constant $M$
such that $\E _{f_0}\psi_2(Z_i(f)/\sqrt M)\le 1$, for $\psi_2(x)=e^{x^2}-1$. Furthermore, set 
$$B_{f_0}(\eps)=\Bigl\{f: \sum_{i=1}^n \!K(p_{f_0,i},p_{f,i})\le n\eps^2,\sum_{i=1}^n \!E(p_{f_0,i},p_{f,i})\le n\eps^2\Bigr\}.$$

\begin{lemma}\label{lem: LB:denom}
There exists a universal constant $D>0$ such that, for any prior distribution $\Pi$,
\begin{align*}
P_{f_0}\Bigl(\int \prod_{i=1}^n \frac{p_{f,i}}{p_{f_0,i}}(Y_i|x_i)\,d\Pi(f)\geq \Pi\bigl(B_{f_0}(\eps)\bigr)e^{-2n\eps^2}\Bigr)\le 2 e^{-Dn\eps^2}.
\end{align*}
\end{lemma}

\begin{proof}
Let $\Pi_\eps$ be the probability measure obtained by restricting and renormalising $\Pi$ to $B_{f_0}(\eps)$. By first restricting the integral to
$B_{f_0}(\eps)$ and next using Jensen's inequality we see
\begin{align*}
&\log \int \prod_{i=1}^n \frac{p_{f,i}}{p_{f_0,i}}(Y_i|x_i)\,\frac{d\Pi(f)}{\Pi\bigl(B_{f_0}(\eps)\bigr)}\\
&\qquad\geq \int  \sum_{i=1}^n  \log   \frac{p_{f,i}}{p_{f_0,i}}(Y_i|x_i)\,d\Pi_\eps(f)\\
&\qquad=\int \sum_{i=1}^n Z_i(f)\,d\Pi_\eps(f)-\int \sum_{i=1}^n K(p_{f_0,i},p_{f,i})\,d\Pi_\eps(f),
\end{align*}
where $Z_i(f)$ is as indicated preceding the lemma. By the definitions of $B_{f_0}(\eps)$ and $\Pi_\eps$, the second term on the right is bounded
below by $-n\eps^2$. It follows that the left side of the display is bounded below by $Z-n\eps^2$, for $Z$ the first
integral on the far right. By convexity the Orlicz norm of this variable satisfies $\|Z\|_{\psi_2}\le \int \bigl\|\sum_{i=1}^n Z_i(f)\bigr\|_{\psi_2}\,d\Pi_\eps(f)$.
By general bounds on Orlicz norms of sums of centered variables (see e.g.\ the third inequality in Proposition~A.1.6 in \cite{vdVW2}), this is further bounded
above by a multiple of $\int \sqrt{\sum_{i=1}^n \|Z_i(f)\|_{\psi_2}^2}\,d\Pi_\eps(f)\lesssim \sqrt{n\eps^2}$,
by the definitions of $B_{f_0}(\eps)$ and $\Pi_\eps$, since $\|Z_i(f)\|_{\psi_2}^2=E(p_{f_0,i},p_{f,i})$. 

We conclude that the left side of the
first display of the proof is bounded below by $Z-n\eps^2$, for $Z$ a random variable with $\|Z\|_{\psi_2}\lesssim \sqrt{n\eps^2}$. Consequently
the probability in the lemma is bounded above by $\Pr(Z-n\eps^2\le -2n\eps^2)\le \Pr(|Z|\ge n\eps^2)\le 1/\psi_2(n\eps^2/\|Z\|_{\psi_2})\le 1/\psi_2(\sqrt{Dn\eps^2})$, for some constant $D>0$, by Markov's inequality. We finish by noting that $\min(1,1/(e^{x^2}-1))\le 2 e^{-x^2}$, for $x>0$.
\end{proof}

In the Gaussian nonparametric regression model \eqref{def:regression}, we have that $Z_i(f):=\eps_i(f-f_0)(x_i)/\sigma$ and hence 
$E(p_{f_0,i},p_{f,i})=(f-f_0)(x_i)^2\|\eps_i\|_{\psi_2}^2/\sigma^2$. In the logistic regression model
\eqref{def:log:regression}, we have $Z_i(f)=\bigl(Y_i-\psi(f_0(x_i)\bigr)\bigl[\log \bigl(\psi(f)/\psi(f_0)\bigr)-\log \bigl((1-\psi(f))/(1-\psi(f_0))\bigr)\bigr](x_i)$. Since $\log\big(\psi(f)/(1-\psi(f)) \big)(x_i)=f(x_i)$, the term in square brackets is $(f-f_0)(x_i)$, and hence  $E(p_{f_0,i},p_{f,i})= (f-f_0)(x_i)^2\|Y_i-\psi(f_0(x_i))\|_{\psi_2}^2$. In both cases $\sum_{i=1}^n E(p_{f_0,i},p_{f,i})$ is bounded above by a multiple of $\sum_{i=1}^n \bigl(f(x_i)-f_0(x_i)\bigr)^2$
and hence $B_{f_0}(\eps)\supset \{f: \|f-f_0\|_n\le c\eps\}$, for some constant $c>0$.

\section{Auxiliary Lemmas}\label{sec:proof:lemmas}

\begin{lemma}[Lemma K.6 of \cite{ghosal2017fundamentals}]
\label{lem: norm-cdf}
The standard normal quantile function $\Phi^{-1}$ satisfies
$\Phi^{-1}(u)\geq -\sqrt{2\log(1/u)}$ for $u\in(0,1)$ and $\Phi^{-1}(u)\leq -1/2\sqrt{\log(1/u)}$ for $u\in(0,1/2)$.
\end{lemma}

\begin{lemma}[Lemma K.6 of \cite{ghosal2017fundamentals}]
\label{lem: norm-cdf-tail}
The standard normal cumulative distribution function $\Phi$ satisfies, for $x>0$,
$$\frac{e^{-x^2/2}}{\sqrt{2\pi}}\Big(\frac{1}{x}-\frac{1}{x^3}\Big)\leq 1-\Phi(x)\leq e^{-x^2/2}\Bigl(\frac1{\sqrt{2\pi}x}\wedge \frac12\Bigr) .$$
\end{lemma}


\section{Numerical analysis on real world data sets}\label{sec: realworld:extra}

We extend the numerical analysis on real world data by considering three more datasets. For all datasets below we compare the distributed Gaussian Process methods (M1-M4)  to the benchmark, non-distributed approach. We have used the gpml MATLAB package for computing the Gaussian process posterior and the built-in  minimize function for hyper-parameter selection.

We found that, similarly to the superconductivity dataset in Section \ref{sec:realworld}, in all the examples below the spatially distributed methods based on the naive (M2) and exponentially weighted (M4) aggregation techniques performed similarly well to the benchmark non-distributed approach. The product of experts (M1) and the inverse variance weighting (M3) approaches, at the same time, had less consistent performance. While sometimes they performed similarly well, at other occasions they were substantially worse than the benchmark.

\subsection{Combined Cycle Power Plant dataset}
In the first real world application we consider the problem of predicting the energy production of a power plant. In our analysis we use the Combined Cycle Power Plant data set \cite{combined_cycle_power_plant_294}. In the experiment 4 covariates were measured, including the hourly average Temperature (T), Ambient Pressure (AP), Relative Humidity (RH) and Exhaust Vacuum (V) and in total 9568 data points were collected over 6 years (2006-2011). The goal is to predict the net hourly electrical energy output (EP) of the plant. The data set was divided randomly into 8000 training and 1568 test data points. We repeated the analysis ten times to measure the uncertainty of the results. As benchmark we used the standard Gaussian process with squared exponential covariance kernel and compared the performance of the considered distributed GP methods (M1-M4) to it. In the distributed architecture we split the data into ten approximately equal sized groups.  In the spatially distributed architecture the training data was split into groups with respect to the Exhaust Vacuum (V) variable. In method M4 we set $\rho=4$.

The average root mean squared errors (RMSEs) with their standard deviations across the repetitions and the mean run times with the corresponding standard deviations are collected in Table \ref{table: powerplant}. One can observe that the distributed methods are slightly less accurate (have approximately $5\%$ higher RMSE), while they offer 30 times faster algorithms than the benchmark non-distributed method. The distributed methods can be further scaled up by an order of magnitude if we run them parallel to each other and not sequentially as in our experiment for simplicity. The distributed methods performed similarly, with the spatially distributed approaches providing a slightly better accuracy. We note that the differences between the RMSEs are within one standard deviation, hence statistically are not significant.

\begin{table}[!h]
\centering
\begin{tabular}{c|c|c|c|c|c}
Methods &BM&M1&M2&M3&M4\\
\hline\hline
RMSE & 4.046 (0.168) & 4.246 (0.161) & 4.211 (0.153) & 4.213 (0.156) & 4.210 (0.153)\\
\hline
runtime & 2106s (32s) & 70.3s (2.0s) & 70.7s (4.8s) & 70.7s (4.8s)& 70.7s (4.8s)
\end{tabular}
\caption{\scriptsize{Average RMSEs and run times (with standard deviations in brackets) over ten runs of non-distributed (BM) and distributed (M1-M4) GP regression with squared exponential covariance kernel for the Combined Cycle Power Plant data set \cite{combined_cycle_power_plant_294}.} }\label{table: powerplant}
\end{table}

\subsection{Gas Turbine emission dataset}

Next we estimate the energy yield in a gas turbine. We consider the Gas Turbine CO and NOx Emission Data Set \cite{gas_turbine_co_and_nox_emission_data_set_551}, containing 36733 measurements of 11 sensors from a gas turbine collected over five years (from 01/01/2011 to 31/12/2015). We use 10 input variables (i.e. Ambient temperature, Ambient humidity, Air filter difference pressure, Gas turbine exhaust pressure, Turbine inlet temperature, Turbine after temperature, Compressor discharge pressure, Carbon monoxide concentration, and Nitrogen oxides concentration) to predict the Turbine energy yield.  The experiments were run five times to measure the uncertainty of the approximation error. Similarly as before, the benchmark squared exponential GP was compared to its distributed counterparts M1-M4. In each distributed architecture we split the data into 10 subsets and ran the computations sequentially, on a single machine. In the spatially distributed architecture we considered the split with respect to the AT (first) feature variable. In method M4 we set $\rho=1$. 

We report the average RMSEs and run times together with their standard deviations in Table \ref{table: emission}. One can observe that methods M1 and M2 have similar accuracy as the benchmark, while M3 and M4 improve on the non-distributed architecture. In this example smaller choices of $\rho$ in method M4 is better as it gets closer to M3, which performed the best. The run times of the distributed approaches were approximately 40 times faster than the non-distributed counterpart.

\begin{table}[!h]
\centering
\begin{tabular}{c|c|c|c|c|c}
Methods &BM&M1&M2&M3&M4\\
\hline\hline
RMSE & 0.751 (0.032) & 0.747 (0.005) & 0.769 (0.045) & 0.677 (0.012) & 0.713 (0.036)\\
\hline
runtime & 14867s (870s) & 385s (37s) & 384s (43s) & 384s (43s) & 384s (43s)
\end{tabular}
\caption{\scriptsize{Average RMSEs and run times and their standard deviations of the non-distributed (BM) and distributed (M1-M4) GP regression methods with squared exponential covariance kernel for the Gas Turbine CO and NOx Emission Data Set \cite{gas_turbine_co_and_nox_emission_data_set_551}.} }\label{table: emission}
\end{table}

\subsection{Appliances Energy Prediction dataset}
In the final experiment we predict the energy use of appliances in a low-energy house. We use the Appliances Energy Prediction data set \cite{appliances_energy_prediction_374} in our analysis. The data set contains  19735 measurements with 28 feature variables (two of them are random noise variables) and the appliance energy consumption as target variable. To simplify the data set the time variable was replaced by the index of the measurement. In our experiment we randomly selected 16.000 data points for training and used the rest for testing. We ran the experiment five times to measure the variability of the prediction error. As before, we used the non-distributed squared exponential GP as benchmark and investigate the performance of the distributed architectures M1-M4. We used 20 machines in each case and ran the computations sequentially.  In the spatially distributed architecture we considered the split with respect to the index of the experiment (first covariate). In M4 we set $\rho=4$. The rest of the setting is as before.

We report the average RMSEs and run times, and their standard deviations in Table \ref{table: energy}. One can observe that the spatially distributed naive and exponential aggregation techniques performed the best with $5\%$ worse accuracy than the benchmark non-distributed method. At the same time, the product of experts (M1) and spatial approach with inverse variance weights (M3) were substantially worse. The distributed methods were 140-170 times faster, even when sequentially running them.

\begin{table}[!h]
\centering
\begin{tabular}{c|c|c|c|c|c}
Methods &BM&M1&M2&M3&M4\\
\hline\hline
RMSE & 80.48 (1.35) & 97.35 (2.58) & 84.14 (3.04) & 108.48 (3.46) & 84.13 (3.04)\\
\hline
runtime & 39621s (2705s) & 284s (64s) & 233s (47s) & 233s (47s) & 233s (47s)
\end{tabular}
\caption{\scriptsize{Average RMSE and runtime (and their standard deviations) of non-distributed (BM) and distributed (M1-M4) GP regression with squared exponential covariance kernel for the Appliances Energy Prediction data set \cite{appliances_energy_prediction_374}.} }\label{table: energy}
\end{table}

\section{Numerical Analysis on synthetic data sets}\label{sec: simulations:extra}
In this section we extend the numerical analysis on synthetic data sets presented in Section \ref{sec:synthetic}. First we provide the missing tables from  Section \ref{sec:synthetic} and the pointwise analysis for the Mat\'ern covariance kernel. Then we discuss spatial adaptation both with empirical and hierarchical Bayes methods. Finally we consider the squared exponential kernel, and demonstrate similar behaviour both for the hierarchical and empirical Bayes methods as in the Mat\'ern case.

\subsection{Mat\'ern covariance kernel: $L_2$ credible sets}

We start by reporting the average sizes and coverages of the  $L_2$ credible sets in case of the Mat\'ern covariance kernel with oracle rescaling for the true function given in Section \ref{sec:matern:synthetic}.

\begin{table}[!h]
\begin{subtable}[c]{\textwidth}
\centering
\begin{tabular}{c|c|c|c}
(n,m)&$ (2000,10)$&  $(5000,20)$ & $(10000,50)$\\ \hline
 BM& 0.192 ($<$0.001) & 0.146 ($<$0.001)  & 0.119 (0.001) \\
 M1& 0.199 (0.001) & 0.153 ($<$0.001)  & 0.129 (0.001) \\
 M2& 0.225 (0.001) & 0.184 ($<$0.001)  &  0.177 ($<$0.001) \\
 M3& 0.250 (0.001)  & 0.213 ($<$0.001) &  0.227 ($<$0.001) \\
 M4& 0.232 (0.001) & 0.191 ($<$0.001) & 0.182 ($<$0.001) \\
\end{tabular}
	\caption{\scriptsize Average radius of the  $L_2$-credible ball}
\end{subtable}
\newline
\begin{subtable}[c]{\textwidth}
\centering
\begin{tabular}{c|c|c|c}
(n,m)&$ (2000,10)$&  $(5000,20)$ & $(10000,50)$\\ \hline
 BM& 1.00 & 1.00 & 1.00\\
 M1& 1.00 & 1.00 & 1.00 \\
 M2& 1.00 & 1.00 & 1.00\\
 M3& 1.00   & 1.00 & 1.00\\
 M4& 1.00 & 1.00 & 1.00\\
\end{tabular}
	\caption{\scriptsize Proportion of experiments where the true function  $ f_0$ was inside the $L_2$-credible ball.}
\end{subtable}
	\caption{\scriptsize  Deterministic (oracle)  rescaling of the Mat\'ern process prior (with $\alpha=3$). BM: Benchmark, Non-distributed method. M1: Random partitioning, M2: Spatial partitioning, M3: Spatial partitioning with inverse variance weights, M4: Spatial partitioning with exponential weights. Average values over 100 replications of the experiment with standard error in brackets.}
\label{table:credible_nonadapt}
\end{table}

\subsection{Mat\'ern covariance kernel: pointwise behaviour}\label{sec:Matern:example}
We continue our numerical analysis by investigating the pointwise behaviour of the distributed techniques for Mat\'ern covariance kernel. We provide both the size of the pointwise credible intervals and the proportion of cases the true parameter was included in the credible interval. First we considered the behaviour of the function at a randomly selected point $x=173/400$. One can observe that for the deterministic (oracle) choice of the hyper-parameter the posterior resulting in from randomly splitting the data (M1) performs similarly to the original posterior. At the same time the spatial distributed methods (M2-M4) are somewhat more conservative, providing larger credible intervals, while at the same time improving the frequentist coverage slightly as well, see Table \ref{table: pointwise:Matern:nonadapt:x=173/400}. Then we considered the behaviour of the procedures at $x=1/2$, see Table \ref{table: pointwise:Matern:nonadapt:x=1/2}. This point is at the boundary of the spatially distributed methods, hence forms a more adversary test case for methods M2-M4. One can observe that methods M2-M4 are substantially more conservative, providing larger credible intervals while improving the coverage slightly as well. Especially method M3 seems to provide overly large credible intervals.  A possible approach to improve the behaviour of the distributed methods around the boundaries is to consider overlapping regions, smoothing out the sharp edges around the boundaries. This, however, would increase the computational time.

\begin{table}[!h]
\begin{subtable}[c]{\textwidth}
\centering
\begin{tabular}{c|c|c|c}
(n,m)&$ (2000,10)$&  $(5000,20)$ & $(10000,50)$\\ \hline
 BM& 0.376 (0.009) & 0.288 (0.004) & 0.236 (0.002)\\
 M1&  0.387 (0.011) & 0.300 (0.004) & 0.254 (0.004) \\
 M2& 0.387 (0.010) & 0.309  (0.005) & 0.300 (0.004)\\
 M3& 0.421 (0.012)  & 0.347 (0.007) & 0.383 (0.007)\\
 M4& 0.391 (0.010) & 0.312 (0.006) & 0.307 (0.005)\\
\end{tabular}
	\caption{\scriptsize The diameter of the credible interval $4\sigma(x)$.}
\end{subtable}
\newline
\begin{subtable}[c]{\textwidth}
\centering
\begin{tabular}{c|c|c|c}
(n,m)&$ (2000,10)$&  $(5000,20)$ & $(10000,50)$\\ \hline
 BM& 0.98 & 0.99 & 0.95\\
 M1& 0.98 & 0.98 & 0.96 \\
 M2& 0.98 & 0.98 & 0.96\\
 M3& 0.99   & 1.00 & 0.99\\
 M4& 0.98 & 0.99 & 0.98\\
\end{tabular}
\caption{\scriptsize Proportion of experiments where the true functional $ f_0(x)$ was inside the credible interval.}
\end{subtable}
\caption{\scriptsize Deterministic (oracle) rescaling of the Mat\'ern process prior ($\alpha=3$) at the test covariate $x=173/400$}	\label{table: pointwise:Matern:nonadapt:x=173/400}
\end{table}

\begin{table}[!h]
\begin{subtable}[c]{\textwidth}
\centering
\begin{tabular}{c|c|c|c}
(n,m)&$ (2000,10)$&  $(5000,20)$ & $(10000,50)$\\ \hline
 BM& 0.375 (0.010) & 0.288 (0.006) & 0.236 (0.003)\\
 M1&  0.386 (0.011) & 0.300 (0.006) & 0.254 (0.004)\\
 M2& 0.663 (0.032) & 0.520 (0.018) & 0.455 (0.012)\\
 M3& 0.706 (0.024)  & 0.575 (0.017) & 0.559 (0.014)\\
 M4& 0.659 (0.021) & 0.520 (0.013) & 0.454 (0.009)\\
\end{tabular}
	\caption{\scriptsize The diameter of the credible interval $4\sigma(x)$.}
\end{subtable}
\newline
\begin{subtable}[c]{\textwidth}
\centering
\begin{tabular}{c|c|c|c}
(n,m)&$ (2000,10)$&  $(5000,20)$ & $(10000,50)$\\ \hline
 BM& 0.95 & 0.96 & 1.00\\
 M1& 0.96 & 0.97 & 1.00 \\
 M2& 0.98 & 1.00 & 0.98\\
 M3& 1.00   & 1.00 & 1.00\\
 M4& 1.00 & 1.00 & 1.00\\
\end{tabular}
\caption{\scriptsize Proportion of experiments where the true functional $ f_0(x)$ was inside the credible interval.}
\end{subtable}
\caption{\scriptsize Deterministic (oracle) rescaling of the Mat\'ern process prior ($\alpha=3$) at the test covariate $x=1/2$}	\label{table: pointwise:Matern:nonadapt:x=1/2}
\end{table}

Next we investigate the pointwise behaviour of the empirical Bayes approach. We consider the same test design points $x=173/400$ and $x=1/2$ as for the oracle rescaling above, see Tables \ref{table: pointwise:Matern:EB:x=173/400} and \ref{table: pointwise:Matern:EB:x=1/2}, respectively. One can observe that method M1 becomes over-confident. It produces overly small credible intervals and the frequentist coverage of the credible intervals is very low. Method M3 provides overly large credible intervals, while these intervals often do not cover the true functional value, resulting in a highly sub-optimal procedure. Finally, methods M2 and M4 are also more conservative than the benchmark non-distributed approach, but much less than M3. Furthermore, they provide reliable uncertainty statements, unlike M1 and M3.

\begin{table}[!h]
\begin{subtable}[c]{\textwidth}
\centering
\begin{tabular}{c|c|c|c}
(n,m)&$ (2000,10)$&  $(5000,20)$ & $(10000,50)$\\ \hline
 BM& 0.309 (0.021) & 0.229 (0.014) & 0.182 (0.009)\\
 M1&  0.236 (0.045) & 0.171 (0.020) & 0.129 (0.013) \\
 M2& 0.306 (0.047) & 0.266  (0.042) & 0.287 (0.013)\\
 M3& 0.370 (0.036)  & 0.315 (0.014) & 0.355 (0.009)\\
 M4& 0.307 (0.046) & 0.261 (0.015) & 0.289 (0.013)\\
\end{tabular}
	\caption{\scriptsize The diameter of the credible interval $4\sigma(x)$.}
\end{subtable}
\newline
\begin{subtable}[c]{\textwidth}
\centering
\begin{tabular}{c|c|c|c}
(n,m)&$ (2000,10)$&  $(5000,20)$ & $(10000,50)$\\ \hline
 BM& 0.99 & 1.00& 1.00\\
 M1& 0.52 & 0.30 & 0.05 \\
 M2& 0.97 & 1.00 & 0.90\\
 M3& 0.63   & 0.25 & 0.00\\
 M4& 0.98 & 1.00 & 0.95\\
\end{tabular}
\caption{\scriptsize Proportion of experiments where the true functional $ f_0(x)$ was inside the credible interval.}
\end{subtable}
\caption{\scriptsize  Empirical Bayes rescaling of the Mat\'ern process prior ($\alpha=3$) at the test covariate $x=173/400$}	\label{table: pointwise:Matern:EB:x=173/400}
\end{table}

\begin{table}[!h]
\begin{subtable}[c]{\textwidth}
\centering
\begin{tabular}{c|c|c|c}
(n,m)&$ (2000,10)$&  $(5000,20)$ & $(10000,50)$\\ \hline
 BM& 0.309 (0.022) & 0.229 (0.014) & 0.182 (0.009)\\
 M1&  0.235 (0.045) & 0.172 (0.020) & 0.127 (0.009)\\
 M2& 0.382 (0.146) & 0.284 (0.071) & 0.348 (0.095)\\
 M3& 0.371 (0.050)  & 0.313 (0.011) & 0.357 (0.010)\\
 M4& 0.337 (0.076) & 0.265 (0.011) & 0.315 (0.055)\\
\end{tabular}
	\caption{\scriptsize The diameter of the credible interval $4\sigma(x)$.}
\end{subtable}
\newline
\begin{subtable}[c]{\textwidth}
\centering
\begin{tabular}{c|c|c|c}
(n,m)&$ (2000,10)$&  $(5000,20)$ & $(10000,50)$\\ \hline
 BM& 0.93 & 1.00 & 1.00\\
 M1& 0.76 & 0.65 & 0.20 \\
 M2& 0.88 & 0.95 & 1.00\\
 M3& 0.91 & 0.85 & 0.65\\
 M4& 0.97 & 1.00 & 1.00\\
\end{tabular}
\caption{\scriptsize Proportion of experiments where the true functional $ f_0(x)$ was inside the credible interval.}
\end{subtable}
\caption{\scriptsize  Empirical Bayes  rescaling of the Mat\'ern process prior ($\alpha=3$) at the test covariate $x=1/2$}
\label{table: pointwise:Matern:EB:x=1/2}
\end{table}

Finally, we investigate the pointwise behaviour of the hierarchical Bayes approach at the above selected test design points $x=173/400$ and $x=1/2$, see Tables \ref{table: pointwise:Matern:HB:x=173/400} and \ref{table: pointwise:Matern:HB:x=1/2}, respectively. One can observe that the credible intervals provide more reliable frequentist uncertainty quantification, especially for method M3. However, this approach produces the largest credible intervals, providing a very conservative approach. We also note that methods M2 and M4 provide overly large intervals compared to the original posterior.

\begin{table}[!h]
\begin{subtable}[c]{\textwidth}
\centering
\begin{tabular}{c|c|c|c}
(n,m)&$ (2000,10)$&  $(5000,20)$ & $(10000,50)$\\ \hline
 BM& 0.376 (0.009) & 0.228 (0.011) & 0.179 (0.007)\\
 M1&  0.387 (0.011) & 0.173 (0.013) & 0.126 (0.002) \\
 M2& 0.387 (0.010) & 0.265  (0.003) & 0.285 (0.001)\\
 M3& 0.421 (0.012)  & 0.339 (0.004) & 0.423 (0.002)\\
 M4& 0.391 (0.010) & 0.273 (0.003) & 0.294 (0.001)\\
\end{tabular}
	\caption{\scriptsize The diameter of the credible interval $4\sigma(x)$.}
\end{subtable}
\newline
\begin{subtable}[c]{\textwidth}
\centering
\begin{tabular}{c|c|c|c}
(n,m)&$ (2000,10)$&  $(5000,20)$ & $(10000,50)$\\ \hline
 BM& 0.98 & 1.00 & 1.00\\
 M1& 0.98 & 0.45 & 0.10 \\
 M2& 0.98 & 1.00 & 0.80\\
 M3& 0.99   & 1.00 & 1.00\\
 M4& 0.98 & 1.00 & 1.00\\
\end{tabular}
\caption{\scriptsize Proportion of experiments where the true functional $ f_0(x)$ was inside the credible interval.}
\end{subtable}
\caption{\scriptsize  Hierarchical Bayes rescaling of the Mat\'ern process prior ($\alpha=3$) at the test covariate $x=173/400.$}	\label{table: pointwise:Matern:HB:x=173/400}
\end{table}

\begin{table}[!h]
\begin{subtable}[c]{\textwidth}
\centering
\begin{tabular}{c|c|c|c}
(n,m)&$ (2000,10)$&  $(5000,20)$ & $(10000,50)$\\ \hline
 BM& 0.375 (0.010) & 0.228 (0.011) & 0.179 (0.007)\\
 M1&  0.386 (0.011) & 0.173 (0.012) & 0.126 (0.003)\\
 M2& 0.663 (0.032) & 0.341 (0.022) & 0.317 (0.004)\\
 M3& 0.706 (0.024)  & 0.400 (0.014) & 0.437 (0.007)\\
 M4& 0.659 (0.021) & 0.339 (0.009) & 0.320 (0.004)\\
\end{tabular}
	\caption{\scriptsize The diameter of the credible interval $4\sigma(x)$.}
\end{subtable}
\newline
\begin{subtable}[c]{\textwidth}
\centering
\begin{tabular}{c|c|c|c}
(n,m)&$ (2000,10)$&  $(5000,20)$ & $(10000,50)$\\ \hline
 BM& 0.95 & 1.00 & 1.00\\
 M1& 0.96 & 1.00 & 0.80 \\
 M2& 0.98 & 1.00 & 1.00\\
 M3& 1.00   & 1.00 & 1.00\\
 M4& 1.00 & 1.00 & 1.00\\
\end{tabular}
\caption{\scriptsize Proportion of experiments where the true functional $ f_0(x)$ was inside the credible interval.}
\end{subtable}
\caption{\scriptsize  Hierarchical Bayes rescaling of the Mat\'ern process prior ($\alpha=3$) at the test covariate $x=1/2$.}	\label{table: pointwise:Matern:HB:x=1/2}
\end{table}

\subsection{Spatial adaptation}\label{sec:spatial}
In this subsection we investigate the spatial adaptation properties of the distributed GP methods. We consider the true function $f_0$  of the form
\begin{equation}\label{f0:spatial}
\begin{split}
f_0(x)=\big(\sum_{j=0}^\infty a_j \psi_j(x)\big)1_{x\in[0,0.5]}+\big(\sum_{j=0}^\infty b_j \psi_j(x)+c\big)1_{x\in(0.5,1]},\qquad\text{where}\\
c= \sum_{j=0}^\infty (a_j-b_j) \psi_j(0.5),\quad a_j=2j^{-1.3}\cos(2^{j/2}) ,\quad b_j= 2j^{-3.5}\sin(j/2).
\end{split}
\end{equation}
Note that the constant $c$ ensures that the function is continuous. Furthermore, observe that in the first half of the unit interval $[0,0.5]$ the function is less regular than in the second half $[0.5,1]$, see Figure \ref{fig:EB:spatial:adapt}. As prior we use the rescaled Gaussian Process with Mat\'ern covariance kernel and regularity parameter $\alpha=3$. The rescaling parameter is chosen via the MMLE empirical Bayes method. We report the $L_2$ estimation error of the posterior mean and the size of the $L_2$-credible sets for both parts of the signal in Tables \ref{table:error:spatial:first} and \ref{table:error:spatial:second}, and the size and coverage of the credible intervals at the points $x=1/3$ and $x=2/3$ in Tables \ref{table:pointwise:spatial:x=1/3} and  \ref{table:pointwise:spatial:x=2/3}, respectively. One can observe from the figures and the corresponding tables, that the non-distributed and randomly distributed methods do not pick up the different local behaviour of $f_0$ on the two sub-intervals, i.e. the posterior mean has similar smoothness and the credible bands are of similar size on the first and second half of the unit interval. In contrast to this Methods 2 and 4 adapt to the different local regularities. In the first half of the interval the posterior mean is rougher and the bands are wider, while in the second half the estimator is smoother and the bands are narrower. Method 3 is sub-optimal, as we have seen in the previous examples as well. Hence we can conclude that spatially distributed methods, beside substantially reducing the computational burden, have the additional benefit of locally adapting to the functional parameter.

 \begin{figure*}[!t]%
\centering
\includegraphics[width=\textwidth]{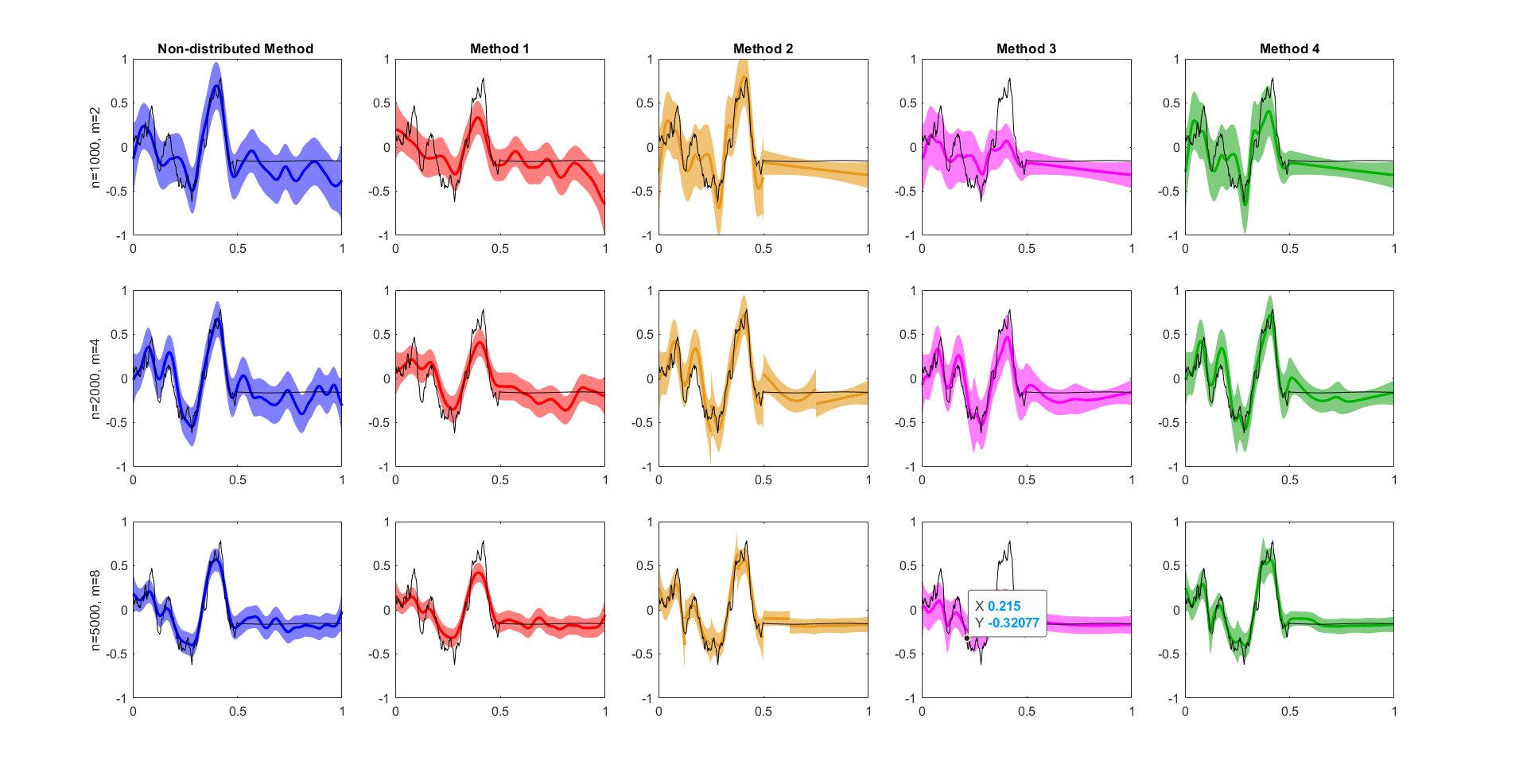}
	\caption{\scriptsize Data-based rescaled (MMLE) squared exponential Gaussian process prior. Benchmark and distributed GP posteriors. True function $f_0$ given in \eqref{f0:spatial} is drawn in black. Posterior means drawn by solid lines, surrounded by $95\%$ pointwise credible sets,  shaded between two dotted lines. The five columns correspond (left to right) to the non-distributed method, the distributed method with random partitioning, and the distributed methods with spatial partitioning without smoothing, with inverse variance weights and with exponential weights. From top to bottom the sample sizes are $n=1000,2000, 5000$ and the numbers of experts are $m=2,4,8$.}
	\label{fig:EB:spatial:adapt}
\end{figure*}

\begin{table}[h]
\begin{subtable}[c]{\textwidth}
\centering
\begin{tabular}{c|c|c|c}
(n,m)&$ (1000,2)$&  $(2000,4)$ & $(5000,8)$\\ \hline
 BM& 0.175 (0.028) & 0.143 (0.019) & 0.107 (0.013)\\
 M1& 0.201  (0.048) & 0.178 (0.036) & 0.135 (0.026)\\
 M2&  0.167 (0.025) & 0.145 (0.019) & 0.107 (0.013)\\
 M3& 0.263 (0.027) & 0.251 (0.029) & 0.248 (0.016)\\
 M4& 0.189 (0.024)  & 0.143 (0.017) & 0.104 (0.013)\\
\end{tabular}
	\caption{\scriptsize Average $L_2$-distance between $f_0$ and posterior mean on the interval $[0,0.5]$.}
\end{subtable}
\newline
\begin{subtable}[c]{\textwidth}
\centering
\begin{tabular}{c|c|c|c}
(n,m)&$ (1000,2)$&  $(2000,4)$ & $(5000,8)$\\ \hline
 BM& 0.119 (0.010) & 0.095 (0.007) & 0.071 (0.004)\\
 M1& 0.106 (0.020) & 0.081 (0.011) & 0.060 (0.006) \\
 M2& 0.144 (0.016) & 0.112 (0.015) &  0.087 (0.006)\\
 M3& 0.128 (0.008)  & 0.111 (0.016) &  0.094 (0.010))\\
 M4& 0.135 (0.014) & 0.108 (0.015) & 0.085 (0.006)\\
\end{tabular}
	\caption{\scriptsize Average radius of the $L_2$-credible ball on the interval $[0,0.5]$.}
\end{subtable}
	\caption{\scriptsize Empirical Bayes rescaling of the Mat\'ern Gaussian process prior with regularity parameter $\alpha=3$. The true regression function is given in \eqref{f0:spatial} and we consider the first half of the unit interval  $[0,0.5]$.}
	\label{table:error:spatial:first}
\end{table}

\begin{table}[h]
\begin{subtable}[c]{\textwidth}
\centering
\begin{tabular}{c|c|c|c}
(n,m)&$ (1000,2)$&  $(2000,4)$ & $(5000,8)$\\ \hline
 BM& 0.104 (0.032) & 0.082 (0.021) & 0.060 (0.009)\\
 M1& 0.098 (0.032) & 0.074 (0.019) & 0.055 (0.010)\\
 M2&  0.047 (0.028) & 0.051 (0.023) & 0.047 (0.012)\\
 M3& 0.047 (0.028) & 0.038 (0.018) & 0.030 (0.011)\\
 M4& 0.047  (0.028)  & 0.047 (0.021) & 0.043 (0.010)\\
\end{tabular}
	\caption{\scriptsize Average $L_2$-distance between $ f_0$ and posterior mean, second part.}
\end{subtable}
\newline
\begin{subtable}[c]{\textwidth}
\centering
\begin{tabular}{c|c|c|c}
(n,m)&$ (1000,2)$&  $(2000,4)$ & $(5000,8)$\\ \hline
 BM& 0.119 (0.010) & 0.095 (0.007) & 0.071 (0.004)\\
 M1& 0.106 (0.020) & 0.081 (0.011) & 0.061 (0.006) \\
 M2& 0.057 (0.004) & 0.053 (0.006) &  0.045 (0.004)\\
 M3& 0.057 (0.004) & 0.056 (0.004) &  0.050 (0.001))\\
 M4& 0.057 (0.004) & 0.053 (0.005)  & 0.044 (0.003)\\
\end{tabular}
	\caption{\scriptsize Average radius credible ball, second part}
\end{subtable}
	\caption{\scriptsize Empirical Bayes rescaling of the Mat\'ern Gaussian process prior with regularity parameter $\alpha=3$. The true regression function is given in \eqref{f0:spatial} and we consider the second half of the unit interval  $[0.5,1]$.}
	\label{table:error:spatial:second}
\end{table}

\begin{table}[!h]
\begin{subtable}[c]{\textwidth}
\centering
\begin{tabular}{c|c|c|c}
(n,m)&$ (1000,2)$&  $(2000,4)$ & $(5000,8)$\\ \hline
 BM& 0.462 (0.047) & 0.375 (0.033) & 0.280 (0.019)\\
 M1&  0.410 (0.085) & 0.317 (0.049) & 0.240 (0.023) \\
 M2& 0.556 (0.075) & 0.444 (0.044) & 0.300 (0.035)\\
 M3& 0.465 (0.030)  & 0.447 (0.040) & 0.348 (0.040)\\
 M4& 0.554 (0.072) & 0.444 (0.044) & 0.300 (0.035)\\
\end{tabular}
	\caption{\scriptsize The diameter of the credible interval $4\sigma(x)$.}
\end{subtable}
\newline
\begin{subtable}[c]{\textwidth}
\centering
\begin{tabular}{c|c|c|c}
(n,m)&$ (1000,2)$&  $(2000,4)$ & $(5000,8)$\\ \hline
 BM& 0.91 & 0.86 & 0.65\\
 M1& 0.88 & 0.81 & 0.55 \\
 M2& 0.86 & 0.89 & 0.70\\
 M3& 0.77 & 0.72 & 0.45\\
 M4& 0.87 & 0.89 & 0.70\\
\end{tabular}
\caption{\scriptsize Proportion of experiments where the true functional $ f_0(x)$ was inside the credible interval.}
\end{subtable}
\caption{\scriptsize  Empirical Bayes rescaling of the Mat\'ern process prior (with regularity parameter $\alpha=3$) at $x=1/3$ for true function $f_0$ given in \eqref{f0:spatial}.}	\label{table:pointwise:spatial:x=1/3}
\end{table}

\begin{table}[!h]
\begin{subtable}[c]{\textwidth}
\centering
\begin{tabular}{c|c|c|c}
(n,m)&$ (1000,2)$&  $(2000,4)$ & $(5000,8)$\\ \hline
 BM& 0.466 (0.045) & 0.374 (0.033) & 0.280 (0.015)\\
 M1&  0.412 (0.086) & 0.315 (0.048) & 0.236 (0.022)\\
 M2& 0.199 (0.013) & 0.190 (0.012) & 0.172 (0.023)\\
 M3& 0.199 (0.013)  & 0.207 (0.014) & 0.197 (0.013)\\
 M4&  0.199 (0.013)  & 0.190 (0.012) & 0.172 (0.023)\\
\end{tabular}
	\caption{\scriptsize The diameter of the credible interval $4\sigma(x)$.}
\end{subtable}
\newline
\begin{subtable}[c]{\textwidth}
\centering
\begin{tabular}{c|c|c|c}
(n,m)&$ (1000,2)$&  $(2000,4)$ & $(5000,8)$\\ \hline
 BM& 0.98 & 1.00 & 1.00\\
 M1& 0.96 & 0.98 & 1.00 \\
 M2& 0.96 & 0.96 & 0.95\\
 M3& 0.97  & 1.00 & 1.00\\
 M4& 0.97 & 0.97 & 0.95\\
\end{tabular}
\caption{\scriptsize Proportion of experiments where the true functional $ f_0(x)$ was inside the credible interval.}
\end{subtable}
\caption{\scriptsize  Empirical Bayes rescaling of the Mat\'ern process prior (with regularity parameter $\alpha=3$) at $x=2/3$ for true function $f_0$ given in \eqref{f0:spatial}.}	\label{table:pointwise:spatial:x=2/3}
\end{table}

 \subsection{Squared exponential kernel}\label{sec:SE:example}
The squared exponential process is the centered Gaussian process with covariance kernel $\E(G_sG_t)=e^{-|s-t|^2/2}$. Its sample paths are infinitely smooth, but with appropriately changed length scale the process is an accurate prior also for finitely smooth functions $f_0$ (see \cite{vandervaart2009}). The optimal length scale for a $\beta$ regular function $f_0$ is $\tau_n=n^{1/(1+2\beta)}$. However, since the regularity is typically not available, we consider data driven choices of the rescaling parameter. In our study we generated data from the regression function $f_0(x)=\sum_{j=0}^{\infty}f_{0,j}\psi_j(x)$ where $\psi_j(x)=\sqrt{2}\cos(\pi(j-1/2)x)$ is the cosine basis and the corresponding coefficients are taken as $2.5 j^{-1/2-\beta}\sin(2j)$ with $\beta=3/2$ for $j\geq 3$ and $f_{0,j}=0$ for $j\leq 2$.

We considered pairs of sample sizes and numbers of machines $(n,m)$ equal to  $(1000,10)$, $(3000,20)$ and $(5000,20)$. For each setting we report the average performance over 100 independent data sets, except in case $n\geq 5000$, where we considered only 20 repetitions, due to the overly slow non-distributed approach.

For visualizing our numerical findings, we plot the posterior means and the pointwise $95\%$-credible bands for the empirical and hierarchical Bayes procedures in Figures~\ref{fig:SE:EB} and~\ref{fig:SE:HB}, respectively. We start by discussing the empirical Bayes approach. We report the $L_2$ estimation error of the posterior mean, the size and coverage of the $L_2$-credible sets, the size and coverage of the pointwise credible intervals at $x=1/3$ and $x=2/3$, and the runtimes in Tables \ref{table: radiusSE_EB}-\ref{table: timeSE_EB}.

In all investigated settings the consensus Monte-Carlo method (M1) performed substantially worse than the true posterior (BM). The corresponding posterior mean is overly smooth and the credible bands are very narrow, providing an overly confident, unreliable uncertainty statement. This suboptimal performance was expected from the fact that this method of data splitting cannot adapt to the unknown smoothness.  The standard spatially distributed approach (M2) provided good estimation and reliable uncertainty quantification. However, the aggregated posterior contains jumps, which are less appealing in practice. The aggregation method (M3) proposed in the literature performed also sub-optimally, providing overly smooth estimators and large credible bands which at the same time still did not provide good frequentist coverage. Our approach with exponential weights (M4), however, outperformed all the other distributed methods. First, it provided better estimation with the posterior mean. Second, although the corresponding credible sets are larger than for the true posterior, they are of similar size as for M2. This, however, resulted in a substantially higher frequentist coverage for this method than for M1 and M3 and even improving in certain cases the coverage of the non-distributed benchmark.

 The run times of the distributed methods were substantially shorter than for the standard, non-distributed algorithm.  We note again, that these run times are for the sequentially executed computations and could be substantially shortened by parallelizing the local machines. For instance, in the last scenario $(5000,20)$ our approach (M4) would take a fraction of a second (around 0.5s) as compared to the average runtime of 745s of the non-distributed algorithm, while providing  comparable reliability and accuracy.

 \begin{figure*}[!t]%
\centering
\includegraphics[width=\textwidth]{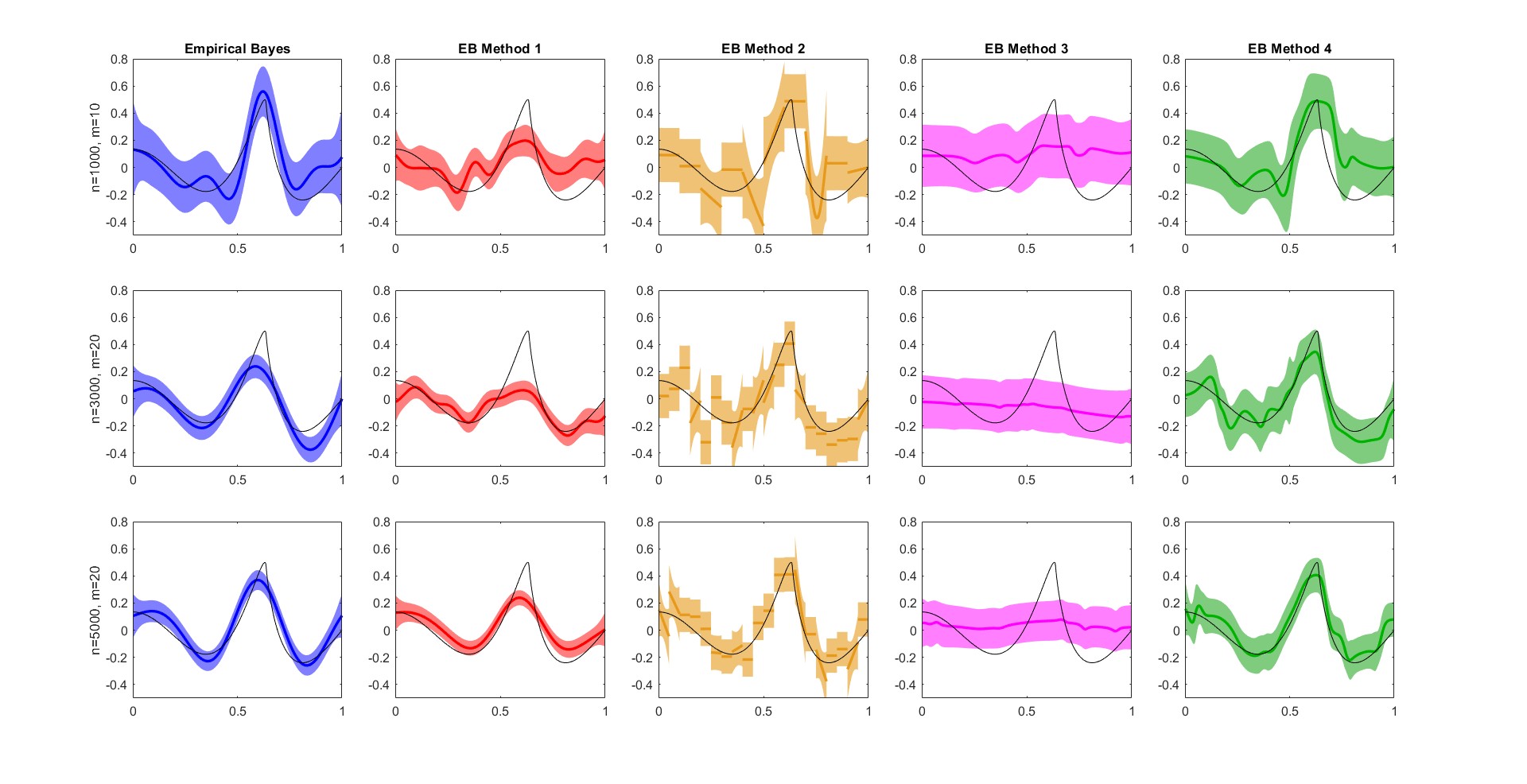}
	\caption{\scriptsize Data-based rescaled (MMLE) squared exponential Gaussian process prior. Benchmark and distributed GP posteriors. True function $ f_0(x)=\sum_{j=3}^{\infty}2.5 j^{-2}\sin(2j)\psi_j(x)$ drawn in black. Posterior means drawn by solid lines, surrounded by $95\%$ pointwise credible sets,  shaded between two dotted lines. The five columns correspond (left to right) to the non-distributed method, the distributed method with random partitioning, and the distributed methods with spatial partitioning without smoothing, with inverse variance weights and with exponential weights. From top to bottom the sample sizes are $n=1000,5000,10000$ and the numbers of experts are $m=5,10,100$.}
	\label{fig:SE:EB}
\end{figure*}

 \begin{table}[!h]
\begin{subtable}[c]{\textwidth}
\centering
\begin{tabular}{c|c|c|c}
(n,m)&$ (1000,10)$&  $(3000,20)$ & $(5000,20)$\\ \hline
 BM& 0.106 (0.026) & 0.069 (0.011) & 0.057 (0.009)\\
 M1& 0.147  (0.024) & 0.129 (0.021) & 0.119 (0.016)\\
 M2&  0.135 (0.024) & 0.099 (0.014) & 0.077 (0.009)\\
 M3& 0.169 (0.012) & 0.170 (0.006) & 0.164 (0.006)\\
 M4& 0.107 (0.021)  & 0.070 (0.012) & 0.057 (0.009)\\
\end{tabular}
\caption{\scriptsize Average $L_2$-distance between $ f_0$ and posterior mean.}
	\label{table: errorSE_EB}
\end{subtable}
\newline
\begin{subtable}[c]{\textwidth}
\centering
\begin{tabular}{c|c|c|c}
(n,m)&$ (1000,10)$&  $(3000,20)$ & $(5000,20)$\\ \hline
 BM& 0.162 (0.023) & 0.107 (0.007) & 0.088 (0.006)\\
 M1& 0.116 (0.018) & 0.071 (0.008) & 0.056 (0.004) \\
 M2& 0.231 (0.013) & 0.179 (0.006) &  0.139 (0.005)\\
 M3& 0.231 (0.04)  & 0191 (0.002) &  0.152 (0.001))\\
 M4& 0.220 (0.010) & 0.174 (0.005) & 0.135 (0.004)\\
\end{tabular}
	\caption{\scriptsize Average size of the credible balls.}
\end{subtable}
	\caption{\scriptsize Empirical Bayes  (MMLE) squared exponential Gaussian process prior. BM: Benchmark, Non-distributed method. M1: Random partitioning, M2: Spatial partitioning, M3: Spatial partitioning with inverse variance weights, M4: Spatial partitioning with exponential weights. We report the average  (and standard deviations) out of 100 repetitions (except of the last column when the number are based on 20 repetitions)}\label{table: radiusSE_EB}
\end{table}

\begin{table}[!h]
\centering
\begin{tabular}{c|c|c|c}
(n,m)&$ (1000,10)$&  $(3000,20)$ & $(5000,20)$\\ \hline
 BM& 0.94 & 1.00 & 1.00\\
 M1& 0.22 & 0.02 & 0.00 \\
 M2& 1.00 & 1.00 & 1.00\\
 M3& 1.00  & 1.00 & 0.05\\
 M4& 1.00 & 1.00 & 1.00\\
\end{tabular}
	\caption{\scriptsize Empirical Bayes (MMLE) squared exponential Gaussian process prior. Proportion of runs when the $L_2$-credible ball contains $f_0$.}
	\label{table: coverage:SE_EB}
\end{table}

\begin{table}[!h]
\begin{subtable}[c]{\textwidth}
\centering
\begin{tabular}{c|c|c|c}
(n,m)&$ (1000,10)$&  $(3000,20)$ & $(5000,20)$\\ \hline
 BM& 0.295 (0.045) & 0.196 (0.014) & 0.163 (0.012)\\
 M1&  0.204 (0.035) & 0.124 (0.017) & 0.098 (0.008)\\
 M2& 0.427 (0.052) & 0.335 (0.010) & 0.257 (0.011)\\
 M3& 0.450 (0.016)  & 0.371 (0.009) & 0.295 (0.008)\\
 M4& 0.423 (0.038) & 0.334 (0.013) & 0.259 (0.012)\\
\end{tabular}
	\caption{\scriptsize The diameter of the credible interval $4\sigma(x)$.}
\end{subtable}
\newline
\begin{subtable}[c]{\textwidth}
\centering
\begin{tabular}{c|c|c|c}
(n,m)&$ (1000,10)$&  $(3000,20)$ & $(5000,20)$\\ \hline
 BM& 0.87 & 0.97 & 0.90\\
 M1& 0.28 & 0.14 & 0.00 \\
 M2& 0.95 & 0.98 & 0.90\\
 M3& 0.96   & 0.93 & 0.65\\
 M4& 0.99 & 1.00 & 0.95\\
\end{tabular}
\caption{\scriptsize Proportion of experiments where the true functional $ f_0(x)$ was inside the credible interval.}
\end{subtable}
\caption{\scriptsize  Empirical Bayes rescaling of the GP with squared exponential covariance kernel at $x=1/3$ for true function $ f_0(x)=\sum_{j=3}^{\infty}2.5 j^{-2}\sin(2j)\psi_j(x)$.}	\label{table:pointwise:SE:EB:x=1/3}
\end{table}

\begin{table}[!h]
\begin{subtable}[c]{\textwidth}
\centering
\begin{tabular}{c|c|c|c}
(n,m)&$ (1000,10)$&  $(3000,20)$ & $(5000,20)$\\ \hline
 BM& 0.295 (0.045) & 0.195 (0.014) & 0.163 (0.012)\\
 M1&  0.203 (0.037) & 0.124 (0.017) & 0.098 (0.008)\\
 M2& 0.454 (0.060) & 0.356 (0.033) & 0.285 (0.029)\\
 M3& 0.458 (0.017)  & 0.356 (0.032) & 0.309 (0.010)\\
 M4& 0.444 (0.044) & 0.350 (0.027) & 0.283 (0.026)\\
\end{tabular}
	\caption{\scriptsize The diameter of the credible interval $4\sigma(x)$.}
\end{subtable}
\newline
\begin{subtable}[c]{\textwidth}
\centering
\begin{tabular}{c|c|c|c}
(n,m)&$ (1000,10)$&  $(3000,20)$ & $(5000,20)$\\ \hline
 BM& 0.88 & 0.91 & 0.90\\
 M1& 0.33 & 0.10 & 0.10 \\
 M2& 0.89 & 0.92 & 0.85\\
 M3& 0.80 & 0.61 & 0.20\\
 M4& 0.95 & 0.97 & 0.90\\
\end{tabular}
\caption{\scriptsize Proportion of experiments where the true functional $ f_0(x)$ was inside the credible interval.}
\end{subtable}
\caption{\scriptsize  Empirical Bayes rescaling of the GP with squared exponential covariance kernel at $x=2/3$ for true function $ f_0(x)=\sum_{j=3}^{\infty}2.5 j^{-2}\sin(2j)\psi_j(x)$.}	\label{table:pointwise:SE:EB:x=2/3}
\end{table}

\begin{table}[!h]
\centering
	\begin{tabular}{c|c|c|c}
(n,m)&$ (1000,10)$&  $(3000,20)$ & $(5000,20)$\\ \hline
Benchmark & 15.27s (3.48s) & 194.8s (24.4s) & 744.5s (107.5s) \\
 Random & 2.37s  (1.39s) & 6.4s (2.0s) & 11.3s (2.5s)\\
Spatial & 2.36s (1.60s) & 6.0s (1.5s) & 10.7s (2.4s)\\
\end{tabular}
	
	\caption{\scriptsize Data-based rescaled (MMLE) squared exponential Gaussian process prior. Average run time for the computation of the posterior. Benchmark: Non-distributed method. Method 1: Random partitioning, Method 2: Spatial partitioning.}
\label{table: timeSE_EB}
\end{table}

Finally, we investigate the hierarchical Bayes rescaling of the squared exponential Gaussian Process prior. We report the $L_2$ estimation error of the posterior mean, the size and coverage of the $L_2$-credible sets, the size and coverage of the pointwise credible intervals at $x=1/3$ and $x=2/3$, and the runtime in Tables \ref{table: HB:SE:error}-\ref{table: timeSE_HB}. One can note, that similarly to the empirical Bayes procedure, the random distribution of the data (M1) performs sub-optimally, providing overly smooth estimators and too narrow, overconfident credible bands with unreliable uncertainty quantification. The spatially distributed methods provide good estimators of the true function, especially M4, even improving occasionally the non-distributed approach. However, they provide larger credible sets, especially method M3 is very conservative. In contrast to the empirical Bayes approach, however, the uncertainty statements resulting in from this approach are more reliable and comparable with the other spatially distributed methods and the original posterior. Finally, the run times of the distributed methods were substantially faster than for the non-distributed approach and could be even improved by running them in parallel, reducing it with an additional factor of $m=20$. We also note that the empirical Bayes approach, as expected, was faster than the hierarchical Bayes method.

 \begin{figure*}[!t]%
\centering
\includegraphics[width=\textwidth]{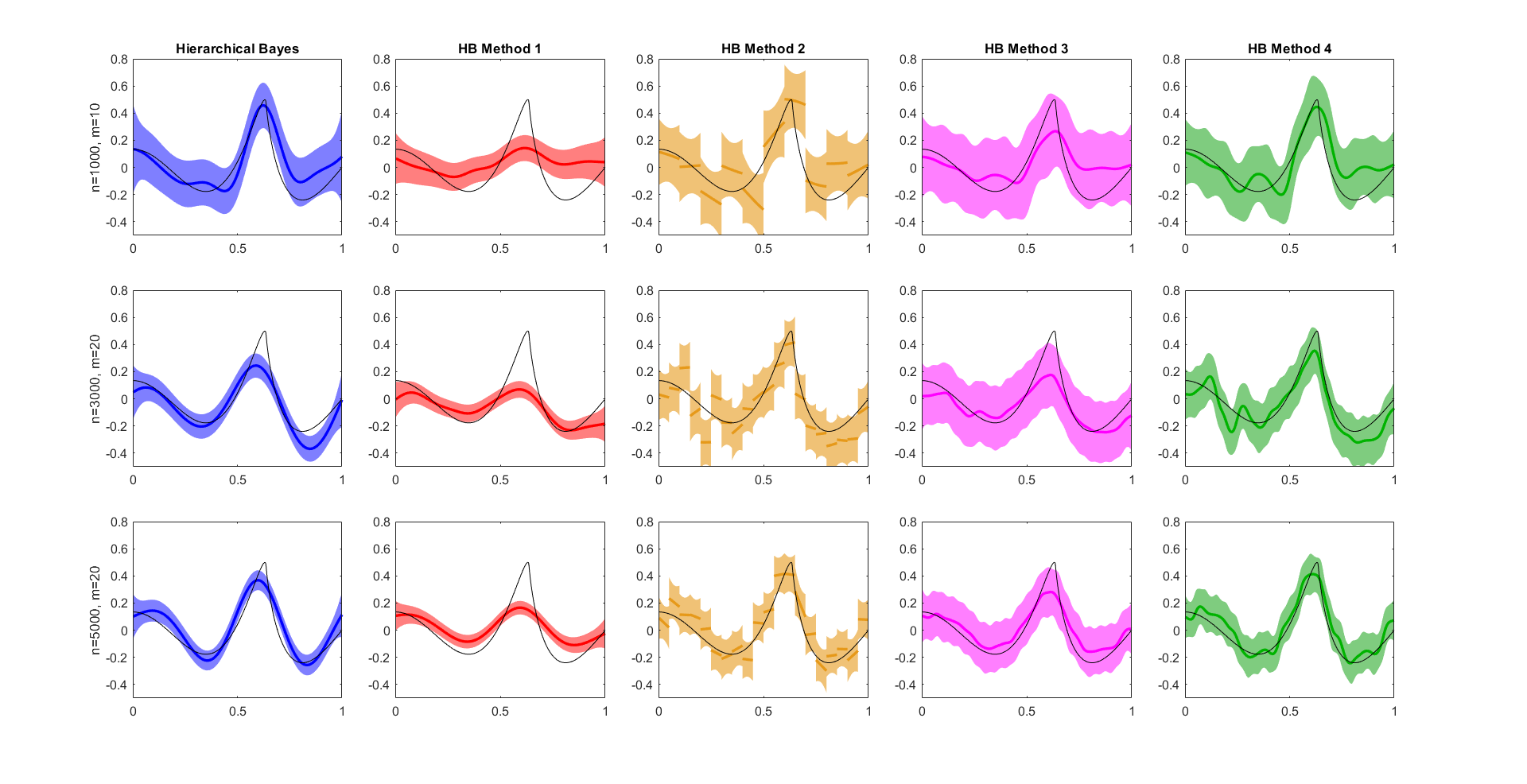}
	\caption{\scriptsize Hierarchical Bayes squared exponential Gaussian process prior. Benchmark and distributed GP posteriors. True function $ f_0(x)=\sum_{j=3}^{\infty}2.5 j^{-2}\sin(2j)\psi_j(x)$ drawn in black. Posterior means drawn by solid lines, surrounded by $95\%$ pointwise credible sets,  shaded between two dotted lines. The five columns correspond (left to right) to the non-distributed method, the distributed method with random partitioning, and the distributed methods with spatial partitioning without smoothing, with inverse variance weights and with exponential weights. From top to bottom the sample sizes are $n=1000,3000,5000$ and the numbers of experts are $m=10,20,20$.}
	\label{fig:SE:HB}
\end{figure*}

 \begin{table}[!h]
\begin{subtable}[c]{\textwidth}
\centering
\begin{tabular}{c|c|c|c}
(n,m)&$ (1000,10)$&  $(3000,20)$ & $(5000,20)$\\ \hline
 BM& 0.108 (0.026) & 0.068 (0.011) & 0.056 (0.009)\\
 M1& 0.151  (0.016) & 0.137 (0.013) & 0.126 (0.011)\\
 M2&  0.118 (0.020) & 0.086 (0.013) & 0.070 (0.007)\\
 M3& 0.108 (0.016) & 0.085 (0.010) & 0.071 (0.006)\\
 M4& 0.094 (0.021)  & 0.064 (0.012) & 0.057 (0.010)\\
\end{tabular}
\caption{\scriptsize Average $L_2$-distance between $ f_0$ and posterior mean.}
	\label{table: errorSE_HB}
\end{subtable}
\newline
\begin{subtable}[c]{\textwidth}
\centering
\begin{tabular}{c|c|c|c}
(n,m)&$ (1000,10)$&  $(3000,20)$ & $(5000,20)$\\ \hline
 BM& 0.156 (0.019) & 0.105 (0.005) & 0.087 (0.004)\\
 M1& 0.113 (0.005) & 0.068 (0.003) & 0.055 (0.002) \\
 M2& 0.222 (0.002) & 0.173 (0.001) &  0.137 (0.001)\\
 M3& 0.274 (0.001)  & 0229 (0.001) &  0.177 (0.001))\\
 M4& 0.229 (0.002) & 0.179 (0.001) & 0.141 (0.001)\\
\end{tabular}
	\caption{\scriptsize Average size of the credible balls.}
	\label{table: radiusSE_HB}
\end{subtable}
	\caption{\scriptsize    Hierarchical Bayes rescaling of the GP with squared exponential covariance kernel. BM: Benchmark, Non-distributed method. M1: Random partitioning, M2: Spatial partitioning, M3: Spatial partitioning with inverse variance weights, M4: Spatial partitioning with exponential weights. We report the average  (and standard deviations) out of 100 repetitions (except of the last column when the number are based on 20 repetitions)}\label{table: HB:SE:error}
\end{table}

\begin{table}[!h]
\centering
\begin{tabular}{c|c|c|c}
(n,m)&$ (1000,10)$&  $(3000,20)$ & $(5000,20)$\\ \hline
 BM& 0.91 & 1.00 & 1.00\\
 M1& 0.05 & 0.00 & 0.00 \\
 M2& 1.00 & 1.00 & 1.00\\
 M3& 1.00 & 1.00 & 1.00\\
 M4& 1.00 & 1.00 & 1.00\\
\end{tabular}
	\caption{\scriptsize Hierarchical Bayes rescaling of the GP with squared exponential covariance kernel. Proportion of runs when the $L_2$-credible ball contains $f_0$.}
	\label{table: coverage:SE_HB}
\end{table}

\begin{table}[!h]
\begin{subtable}[c]{\textwidth}
\centering
\begin{tabular}{c|c|c|c}
(n,m)&$ (1000,10)$&  $(3000,20)$ & $(5000,20)$\\ \hline
 BM& 0.282 (0.039) & 0.193 (0.011) & 0.160 (0.009)\\
 M1&  0.193 (0.012) & 0.116 (0.006) & 0.095 (0.003)\\
 M2& 0.417 (0.007) & 0.332 (0.003) & 0.259 (0.002)\\
 M3& 0.535 (0.011)  & 0.455 (0.005) & 0.348 (0.004)\\
 M4& 0.435 (0.008) & 0.345 (0.003) & 0.268 (0.002)\\
\end{tabular}
	\caption{\scriptsize The diameter of the credible interval $4\sigma(x)$.}
\end{subtable}
\newline
\begin{subtable}[c]{\textwidth}
\centering
\begin{tabular}{c|c|c|c}
(n,m)&$ (1000,10)$&  $(3000,20)$ & $(5000,20)$\\ \hline
 BM& 0.85 & 0.99 & 0.90\\
 M1& 0.22 & 0.03 & 0.00 \\
 M2& 0.96 & 0.97 & 0.90\\
 M3& 1.00  & 1.00 & 1.00\\
 M4& 0.98 & 1.00 & 1.00\\
\end{tabular}
\caption{\scriptsize Proportion of experiments where the true functional $ f_0(x)$ was inside the credible interval.}
\end{subtable}
\caption{\scriptsize  Hierarchical Bayes rescaling of the GP with squared exponential covariance kernel at $x=1/3$ for true function $ f_0(x)=\sum_{j=3}^{\infty}2.5 j^{-2}\sin(2j)\psi_j(x)$.}	\label{table:pointwise:SE:HB:x=1/3}
\end{table}

\begin{table}[!h]
\begin{subtable}[c]{\textwidth}
\centering
\begin{tabular}{c|c|c|c}
(n,m)&$ (1000,10)$&  $(3000,20)$ & $(5000,20)$\\ \hline
 BM& 0.282 (0.039) & 0.193 (0.011) & 0.160 (0.009)\\
 M1& 0.191 (0.010) & 0.116 (0.006) & 0.095 (0.004)\\
 M2& 0.418 (0.011) & 0.338 (0.006) & 0.269 (0.009)\\
 M3& 0.538 (0.017) & 0.462 (0.010) & 0.361 (0.013)\\
 M4& 0.432 (0.012) & 0.351 (0.006) & 0.279 (0.009)\\
\end{tabular}
	\caption{\scriptsize The diameter of the credible interval $4\sigma(x)$.}
\end{subtable}
\newline
\begin{subtable}[c]{\textwidth}
\centering
\begin{tabular}{c|c|c|c}
(n,m)&$ (1000,10)$&  $(3000,20)$ & $(5000,20)$\\ \hline
 BM& 0.84 & 0.92 & 0.90\\
 M1& 0.17 & 0.03 & 0.00 \\
 M2& 0.88 & 0.97 & 0.85\\
 M3& 1.00 & 1.00 & 1.00\\
 M4& 0.97 & 1.00 & 1.00\\
\end{tabular}
\caption{\scriptsize Proportion of experiments where the true functional $ f_0(x)$ was inside the credible interval.}
\end{subtable}
\caption{\scriptsize  Hierarchical Bayes rescaling of the GP with squared exponential covariance kernel at $x=2/3$ for true function $ f_0(x)=\sum_{j=3}^{\infty}2.5 j^{-2}\sin(2j)\psi_j(x)$.}	\label{table:pointwise:SE:HB:x=2/3}
\end{table}

\begin{table}[!h]
\centering
	\begin{tabular}{c|c|c|c}
(n,m)&$ (1000,10)$&  $(3000,20)$ & $(5000,20)$\\ \hline
Benchmark & 27.55s (5.57s) & 334.3s (42.4s) & 1265.4s (153.8s) \\
 Random & 5.92s  (1.50s) & 14.7s (1.9s) & 24.7s (7.7s)\\
Spatial & 5.85s (1.91s) & 14.4s (2.7s) & 22.8s (6.3s)\\
\end{tabular}
	
	\caption{\scriptsize Hierarchical Bayes rescaling of the GP with squared exponential covariance kernel. Average run time for the computation of the posterior. Benchmark: Non-distributed method. Method 1: Random partitioning, Method 2: Spatial partitioning.}
\label{table: timeSE_HB}
\end{table}

\clearpage

\bibliographystyle{plain}
\bibliography{references}


\end{document}